\documentclass{fodsOF}
\usepackage{amsmath}
  \usepackage{paralist}
  \usepackage{graphics} 
  \usepackage{epsfig} 
\usepackage{graphicx}  \usepackage{epstopdf}
 \usepackage[colorlinks=true]{hyperref}
\hypersetup{urlcolor=blue, citecolor=red}

\def\currentvolume{X}
 \def\currentissue{X}
  \def\currentyear{200X}
   \def\currentmonth{XX}
    \def\ppages{X--XX}
     \def\DOI{10.3934/fods.2019001}

\usepackage{amsmath,bbm,url}
\usepackage[large]{subfigure}

\newtheorem{thm}{Theorem}
\newtheorem{conj}{Conjecture}

\newtheorem{assum}{Assumption}

\theoremstyle{definition}

\newtheorem{Def}[thm]{Definition}
\newtheorem{examp}{Example}

\makeatletter
\newcommand{\xleftrightarrow}[2][]{\ext@arrow 3359\leftrightarrowfill@{#1}{#2}}
\newcommand{\xdashrightarrow}[2][]{\ext@arrow 0359\rightarrowfill@@{#1}{#2}}
\newcommand{\xdashleftarrow}[2][]{\ext@arrow 3095\leftarrowfill@@{#1}{#2}}
\newcommand{\xdashleftrightarrow}[2][]{\ext@arrow 3359\leftrightarrowfill@@{#1}{#2}}
\def\rightarrowfill@@{\arrowfill@@\relax\relbar\rightarrow}
\def\leftarrowfill@@{\arrowfill@@\leftarrow\relbar\relax}
\def\leftrightarrowfill@@{\arrowfill@@\leftarrow\relbar\rightarrow}
\def\arrowfill@@#1#2#3#4{%
  $\m@th\thickmuskip0mu\medmuskip\thickmuskip\thinmuskip\thickmuskip
   \relax#4#1
   \xleaders\hbox{$#4#2$}\hfill
   #3$%
}
\makeatother

\title[Consistent Manifold Representation] 
      {Consistent manifold representation for topological data analysis}

\author[Tyrus Berry and Timothy Sauer]{}


 \email{tberry@gmu.edu}
 \email{tsauer@gmu.edu}

\thanks{The authors are partially supported by NSF grant DMS-1723128}

\thanks{$^*$ Corresponding author: Timothy Sauer}

\begin{document}

\maketitle
\centerline{\scshape Tyrus Berry and Timothy Sauer$^*$}
\medskip
{\footnotesize
 \centerline{Department of Mathematical Sciences}
   \centerline{ Fairfax, VA 22030, USA}
} 

%


\begin{abstract}
For data sampled from an arbitrary density on a manifold embedded in Euclidean space, the \emph{Continuous k-Nearest Neighbors} (CkNN) graph construction is introduced.  It is shown that CkNN is geometrically consistent in the sense that under certain conditions, the unnormalized graph Laplacian converges to the Laplace-Beltrami operator, spectrally as well as pointwise.   It is proved for compact (and conjectured for noncompact) manifolds that CkNN is the unique unweighted construction that yields a geometry consistent with the connected components of the underlying manifold in the limit of large data. Thus CkNN produces a single graph that captures all topological features simultaneously, in contrast to persistent homology, which represents each homology generator at a separate scale.  As applications we derive a new fast clustering algorithm and a method to identify patterns in natural images topologically. Finally, we conjecture that CkNN is topologically consistent, meaning that the homology of the Vietoris-Rips complex (implied by the graph Laplacian) converges to the homology of the underlying manifold (implied by the Laplace-de Rham operators) in the limit of large data.
\end{abstract}


\section{Introduction}\label{intro}

Building a discrete representation of a manifold from a finite data set is a fundamental problem in machine learning. Particular interest pertains to the case where a set of data points in a possibly high-dimensional Euclidean space is assumed to lie on a relatively low-dimensional manifold. The field of topological data analysis  (TDA) concerns the extraction of topological invariants such as homology from discrete measurements.

Currently, there are two major methodologies for representing manifolds from data sets. One approach is an outgrowth of Kernel PCA \cite{scholkopf1998nonlinear}, using graphs with {\it weighted}  edges formed by localized kernels to produce an operator that converges to the Laplace-Beltrami operator of the manifold. These methods include versions of diffusion maps \cite{belkin2003laplacian,diffusion,localk,BH14} that reconstruct the geometry of the manifold with respect to a desired metric. Convergence of the weighted graph to the Laplace-Beltrami operator in the large data limit is called {\it consistency} of the graph construction. Unfortunately, while such constructions implicitly contain all topological information about the manifold, it is not yet clear how to use a weighted graph to build a simplicial complex from which simple information like the Betti numbers can be extracted.

A second approach, known as persistent homology \cite{carlsson2009topology,edels2010,ghrist2008}, produces a series of {\it unweighted} graphs that reconstructs topology one scale at a time, tracking homology generators as a scale parameter is varied. The great advantage of an unweighted graph is that the connection between the graph and a simplicial complex is immediate, since the Vietoris-Rips
construction can be used to build an abstract simplicial complex from the graph. However, the persistent homology approach customarily creates a family of graphs, of which none is guaranteed to contain all topological information. The goal of a consistent theory is not possible since there is not a single unified homology in the large data limit.

In this article we propose replacing persistent homology with consistent homology in data analysis applications. In other words, our goal is to show that it is possible to construct a single unweighted graph from which the underlying manifold's topological information can be extracted. We introduce a specific graph construction from a set of data points, called continuous k-nearest neighbors (CkNN),  that achieves this goal for any compact Riemannian manifold. Theorem \ref{consgeom} states that the CkNN is the unique unweighted graph construction for which the (unnormalized) graph Laplacian converges spectrally to a Laplace-Beltrami operator on the manifold in the large data limit. Furthermore, even when the manifold is not compact we compute the optimal bias-variance tradeoff, and show that CkNN always results in a geometry where the Laplace-Beltrami operator has a discrete spectrum.

The proof of consistency for the CkNN graph construction is carried out in Section \ref{background} for both weighted and unweighted graphs. There we complete the theory of graphs constructed from variable bandwidth kernels, computing for the first time the bias and variance of both pointwise and spectral estimators. Combined with existing work on spectral convergence \cite{von2008consistency,spec0,spec1,spec2,spec3} we obtain consistency.  Our analysis reveals the surprising fact that the optimal bandwidth for spectral estimation is significantly smaller than the optimal choice for pointwise estimation (see Fig.~\ref{figureApp}).  This is crucial because existing statistical estimates \cite{SingerEstimate,BH14} imply very different parameter choices that are not optimal for the spectral convergence desired in most applications.  Moreover, requiring the spectral variance to be finite allows us to specify which geometries are accessible on non-compact manifolds.  
Details on the relationship of our new results to previous work are given in Section \ref{background}.

As mentioned above, the reason for focusing on unweighted graphs is their relative simplicity for topological investigation.  In an unweighted graph construction, one can apply the depth first search algorithm to determine the zero homology. On the other hand, there are many weighted graph constructions that converge to the Laplace-Beltrami operator with respect to various geometries \cite{belkin2003laplacian,diffusion,localk,BH14}.  Although these methods are very powerful, they are not convenient for extracting topological information.  For example, to determine the zero homology from a weighted graph requires numerically estimating the dimension of the zero eigenspace of the graph Laplacian, much less efficient than a depth-first search.  Secondly, determination of the number of zero eigenvalues requires setting a nuisance parameter as a numerical threshold. For higher-order homology generators, the problem is even worse, as weighted graphs require the construction of the Laplace-de Rham operators which act on differential forms.  (We note that the $0$-th Laplace-de Rham operator acts on function, or $0$-forms, and is called the Laplace-Beltrami operator.)  In contrast, the unweighted graph construction allows the manifold to be studied using topological data analysis methods that are based on simplicial homology (e.g.~computed from the Vietoris-Rips complex).

The practical advantages of the CkNN are: (1) a single graph representation of the manifold that captures topological features at multiple scales simultaneously (see Fig.~\ref{clusteringFig3}) and (2) identification of the correct topology even for non-compact sampling measures (see Fig.~\ref{clusteringFig2}).  In CkNN, the length parameter $\epsilon$ is eliminated, and replaced with a unitless scale parameter $\delta$. Fortunately, for computational purposes, the consistent homology in terms of $\delta$ uses the same efficient computational homology algorithms as conventional persistent homology.

For some applications, the unitless parameter $\delta$ may be a disadvantage; for example, if the distance scales of particular features need to be  explicitly separated. In such a case, absolute distances are meaningful, units are important, and the standard $\epsilon$ ball persistence diagram may be more relevant to the problem than a single consistent topology. The point of this article is that if one is truly interested only in the topology in TDA, a unitless $\delta$ is more appropriate and leads to a consistency theory.  Finally, for a fixed data set, the consistent homology approach requires choosing the parameter $\delta$ (which determines the CkNN graph) and we re-interpret the classical persistence diagram as a tool for selecting $\delta$.

We introduce the CkNN in Section \ref{sec2}, and demonstrate its advantages in topological data analysis by considering a simple but illustrative example. In Section \ref{TDA} we show that consistent spectral estimation of the Laplace-de Rham operators is the key to consistent estimation of topological features.  In particular, the key to estimating the connected components of a manifold is spectral estimation of the Laplace-Beltrami operator. We show how these results guarantee consistency of the connected components (clustering), and combined with recent results \cite{berry2018spectral} on spectral exterior calculus (SEC), imply consistency of the higher order homology. In Section \ref{uniqueconstruction}, these results are used to show that the CkNN is the unique unweighted graph construction which yields a consistent geometry via the Laplace-Beltrami operator on functions.  We give several examples that demonstrate the consistency of the CkNN construction in Section \ref{spectralclustering}, including a fast and consistent clustering algorithm that allows more general sampling densities than existing theories.  Theoretical results are given in Section \ref{background}. We conclude in Section \ref{conclusion} by discussing the relationship of CkNN to classical persistence. In this article, we focus on applications to TDA, but the theoretical results will be of independent interest to those studying the geometry as well as topology of data.

\section{Continuous scaling for unweighted graphs}\label{sec2}

We begin by describing the CkNN graph construction and comparing it to other approaches.  Then we discuss the main issues of this article as applied to a simple example of data points arranged into three rectangles with nonuniform sampling.

\subsection{Continuous k-Nearest Neighbors}
Our goal is to create an unweighted, undirected graph from a point set with interpoint distances given by a metric $d$. Since the data points naturally form the vertices of a graph representation, for each pair of points we only need to decide whether or not to connect these points with an edge.  There are two standard approaches for constructing the graph:
\begin{enumerate}
\item {\bf Fixed $\epsilon$-balls}: For a fixed $\epsilon$, connect the points $x,y$ if $d(x,y) < \epsilon$.
%

\item {\bf k-Nearest Neighbors (kNN)}: For a fixed integer $k$, connect the points $x,y$ if either $d(x,y) \leq d(x,x_{k})$ or $d(x,y) \leq d(y,y_{k})$ where $x_{k}, y_k$ are the $k$-th nearest neighbors of $x,y$ respectively.
\end{enumerate}
The fixed $\epsilon$-balls choice works best when the data is uniformly distributed on the manifold, whereas the kNN approach adapts to the local sampling density of points.  However, we will see that even when answering the simplest topological questions, both standard approaches have severe limitations. For example, when clustering a data set into connected components, they may underconnect one part of the graph and overestimate the number of components, while overconnecting another part of the graph and bridging parts of the data set that should not be connected.
Despite these drawbacks, the simplicity of these two graph constructions has led to their widespread use in manifold learning and topological data analysis methods \cite{carlsson2009topology}.

Our main point is that a less discrete version of kNN sidesteps these problems, and can be proved to lead to a consistent theory in the large data limit. Define the \emph{Continuous k-Nearest Neighbors} (CkNN) graph construction by
\begin{enumerate}
\item[3.] {\bf CkNN}: Connect the points $x,y$ if ${\displaystyle d(x,y) < \delta \sqrt{ d(x,x_k)d(y,y_k)}}$
\end{enumerate}
where the parameter $\delta$ is allowed to vary continuously.  Of course, the discrete nature of the (finite) data set implies that the graph will change at only finitely many values of $\delta$.  The continuous parameter $\delta$ has two uses.  First, it allows asymptotic analysis of the graph Laplacian in terms of $\delta$, where we interpret the CkNN graph construction as a kernel method.  Second, it allows the parameter $k$ to be fixed for each data set, which allows us to interpret $d(x,x_k)$ as a local density estimate.

The CkNN construction is closely related to the ``self-tuning" kernel introduced in \cite{ZP} for the purposes of spectral clustering, which was defined as
\begin{equation}\label{zpkernel} K(x,y) = \exp\left(-\frac{d(x,y)^2}{d(x,x_k)d(y,y_k)}\right). \end{equation}
The kernel (\ref{zpkernel}) leads to a weighted graph, but replacing the exponential kernel with the indicator function
\begin{equation}\label{cutoffkernel} K(x,y) = \mathbbm{1}_{\left\{\frac{d(x,y)^2}{d(x,x_k)d(y,y_k)}<1 \right\}} \end{equation}
 and introducing the continuous parameter $\delta$ yields the CkNN unweighted graph construction.  The limiting operator of the graph Laplacian based on the kernels \eqref{zpkernel} and \eqref{cutoffkernel} was first analyzed pointwise in \cite{Ting2010,BH14}.  In Sec.~\ref{background} we provide the first complete analysis of the spectral convergence of these graph Laplacians, along with the bias and variance of the spectral estimates.

The CkNN is an instance of a broader class of multi-scale graph constructions:
\begin{enumerate}
 \item[(*)] {\bf Multi-scale}: Connect the points $x,y$ if $d(x,y) < \delta \sqrt{\rho(x)\rho(y)}$
 \end{enumerate}
where $\rho(x)$ defines the local scaling near the point $x$.  In Section \ref{uniqueconstruction} we will show that
\begin{equation}\label{densityBandwidth} \rho(x) \propto q(x)^{-1/m} \end{equation}
is the unique multi-scale graph construction that yields a consistent limiting geometry, where $q(x)$ is the sampling density and $m$ is the intrinsic dimension of the data.

In applications to data, neither $q(x)$ nor $m$ may be known beforehand.
Fortunately, for points on a manifold embedded in Euclidean space, the kNN itself provides a very simple density estimator, which for $k$ sufficiently small approximately satisfies
\begin{equation}\label{kNNdensity} ||x-x_k|| \propto q(x)^{-1/m}  \end{equation}
where $x_k$ is the $k$-th nearest neighbor of $x$ and $m$ is the dimension of the underlying manifold \cite{loftsgaarden65}. Although more sophisticated kernel density estimators could be used (see for example \cite{ScottVBK}), a significant advantage of \eqref{kNNdensity} is that it implicitly incorporates the exponent $-1/m$ without the need to estimate the intrinsic dimension $m$ of the underlying manifold.

In the next section, we demonstrate the advantages of the CkNN on a simple example before turning to the theory of consistent topological estimation.

\subsection{Example: Representing non-uniform data}\label{largedata}

\begin{figure}
\begin{center}
\subfigure[]{\includegraphics[width=.4\linewidth]{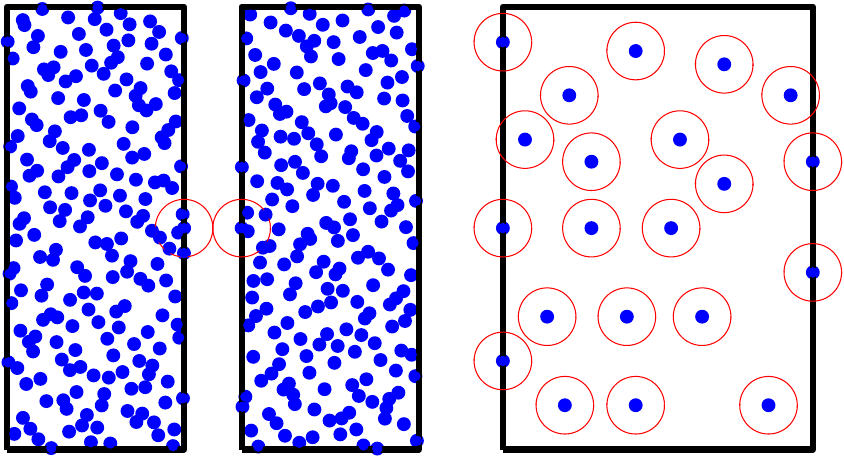}}\hspace{.1\linewidth}
\subfigure[]{\includegraphics[width=.41\linewidth]{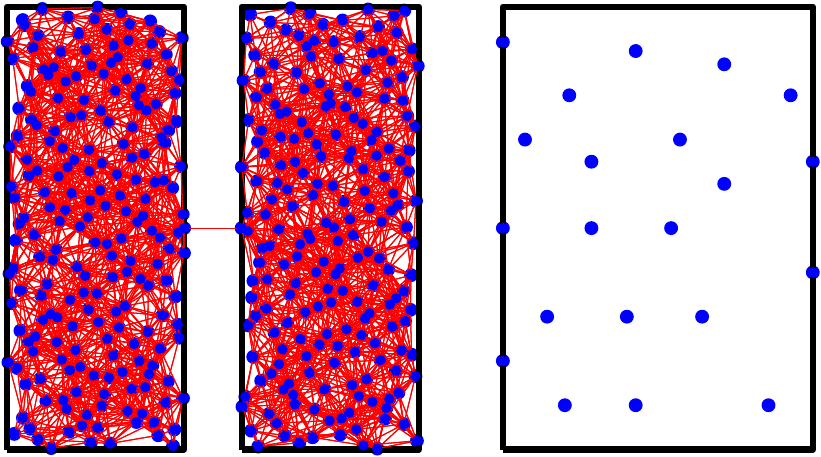}}
\end{center}
\caption{ \rm Rectangular regions indicate the true clusters. (a) Circles of a fixed radius $\epsilon$ show that bridging of the dense regions occurs before any connections are made in the sparse region. (b) Graph connecting all points with distance less than $\epsilon$. }\label{figure7}
\end{figure}

In this section we start with a very simple example  where both the fixed $\epsilon$-ball and simple kNN graph constructions fail to identify the correct topology.
\begin{examp} \rm
Fig.~\ref{figure7} shows a simple ``cartoon'' of non-uniform sampling of data that is common in real applications, and reveals the weakness of standard graph constructions. All the data points lie in one of the three rectangular connected components outlined in Fig.~\ref{figure7}(a). The left and middle components are densely sampled and the right component is more sparsely sampled.  Consider the radius $\epsilon$ indicated by the circles around the data points in Fig.~\ref{figure7}(a). At this radius, the points in the sparse component are not connected to any other points in that component.  This $\epsilon$ is too small for the connectivity of the sparse component to be realized, but at the same time is too large to distinguish the two densely sampled components.  A graph built by connecting all points within the radius $\epsilon$, shown in Fig.~\ref{figure7}(b), would find many spurious components in the sparse region while simultaneously improperly connecting the two dense components.  We are left with a serious failure: The graph cannot be tuned, with any fixed $\epsilon$, to identify the ``correct'' three boxes as components.
\end{examp}

\begin{figure}
\begin{center}
\subfigure[]{\includegraphics[width=.4\linewidth]{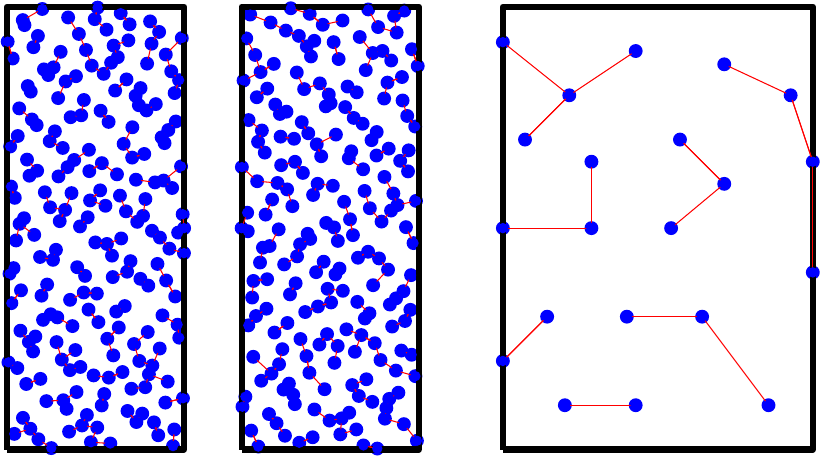}}\hspace{.1\linewidth}
\subfigure[]{\includegraphics[width=.4\linewidth]{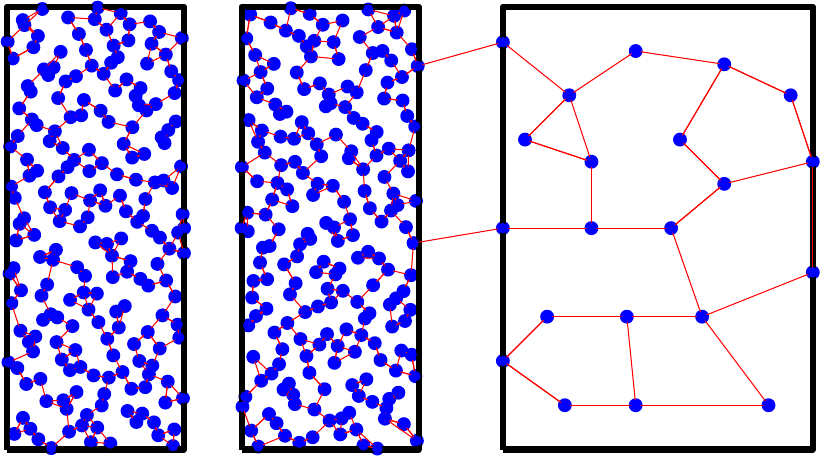}}
\end{center}
\caption{ \rm Data set from Fig.~\ref{figure7}. (a) Graph based on connecting each point to its (first) nearest neighbor leaves all regions disconnected internally.  (b) Connecting each point to its two nearest neighbors bridges the sparse region to one of the dense regions before fully connecting the left and center boxes.}\label{figure8}
\end{figure}

The kNN approach to local scaling is to replace the fixed $\epsilon$ approach with the establishment of edges between each point and its $k$-nearest neighbors.  While this is an improvement, in Fig.~\ref{figure8} we show that it still fails to reconstitute the topology even for the very simple data set considered in Fig.~\ref{figure7}.  Notice that in Fig.~\ref{figure8}(a) the graph built based on the nearest neighbor ($k=1$) leaves all regions disconnected, while using two nearest neighbors ($k=2$) incorrectly bridges the sparse region with a dense region, as shown in Fig.~\ref{figure8}(b). Of course, using kNN with $k>2$ will have the same problem as $k=2$. Fig.~\ref{figure8} shows that simple kNN is not a panacea for nonuniformly sampled data.

Finally, we demonstrate the CkNN graph construction in Fig.~\ref{figure9}. An edge is added between points $x$ and $y$ when $d(x,y) < \delta \sqrt{d(x,x_k) d(y,y_k)}$. We denote the $k$th-nearest neighbor of $x$ (resp., $y$) by $x_k$ (resp., $y_k$). The coloring in Fig.~\ref{figure9}(a) exhibits the varying density between boxes. Assigning edges according to CkNN with $k=10$ and optimally tuning $\delta$, shown in Fig.~\ref{figure9}(b), yields an unweighted graph whose connected components reflect the manifold in the large data limit. Theorem \ref{consgeom} guarantees the existence of such a $\delta$, that yields an unweighted CkNN graph with correct topology.

\begin{figure}
\begin{center}
\subfigure[]{\includegraphics[width=.4\linewidth]{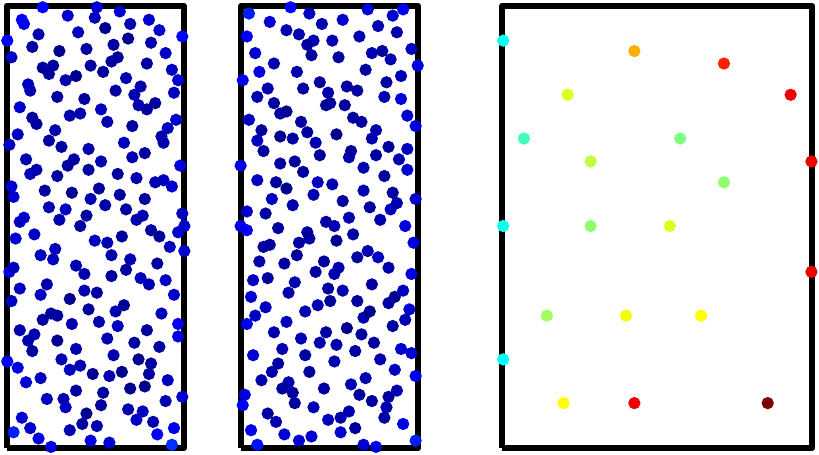}}\hspace{.1\linewidth}
\subfigure[]{\includegraphics[width=.4\linewidth]{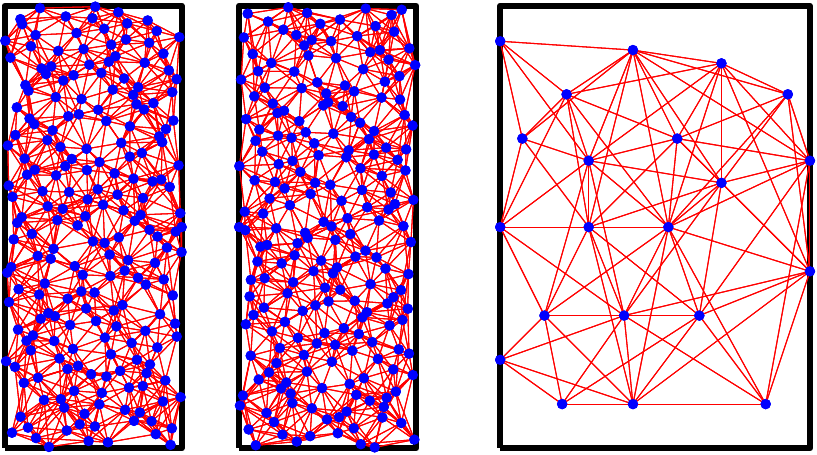}}
\end{center}
\caption{ \rm Data set from Fig.~\ref{figure7}.  (a) Color indicates the relative values of the bandwidth function $\rho(x) = d(x,x_k)$ using $k=10$ and optimally tuning $\delta$ (blue = low, red = high). (b) Graph connecting pairs of data points with $||x-y|| < \delta \sqrt{\rho(x)\rho(y)}$.  Notice that the connected components are fully connected and triangulated, yielding the correct homology in dimensions $0, 1$, and $2$.}\label{figure9}
\end{figure}

Although we have focused on the connected components of the point set, the CkNN graph in Fig.~\ref{figure9} fully triangulates all regions, which implies that the $1$-homology is correctly identified as trivial.  Clearly, the graph constructions in Figures \ref{figure7} and \ref{figure8} are very far from identifying the correct $1$-homology.

To complete our analysis of the three-box example, we compare CkNN to a further alternative.  Two crucial features of the CkNN graph construction are (1) symmetry in $x$ and $y$ which implies an undirected graph construction, and (2) introduction of the continuous parameter $\delta$ which allows $k$ to be fixed so that $\rho(x)=||x-x_k||$ is an estimator of $q(x)^{-1/m}$.  There are many alternative ways of combining the local scaling function $\rho(x)$ with the continuous parameter $\delta$.  Our detailed theoretical analysis in Sec.~\ref{background} shows that the geometric average $\delta\sqrt{\rho(x)\rho(y)}$ is consistent in the large data limit, but it does not discount all alternatives.

\begin{figure}
\begin{center}
\subfigure[]{\includegraphics[width=.45\linewidth]{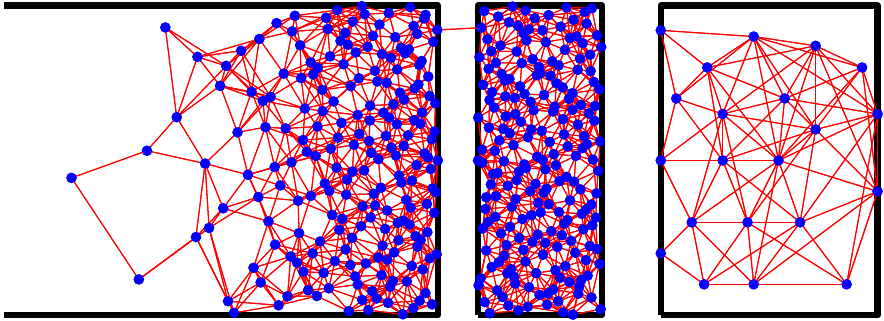}}\hspace{.05\linewidth}
\subfigure[]{\includegraphics[width=.45\linewidth]{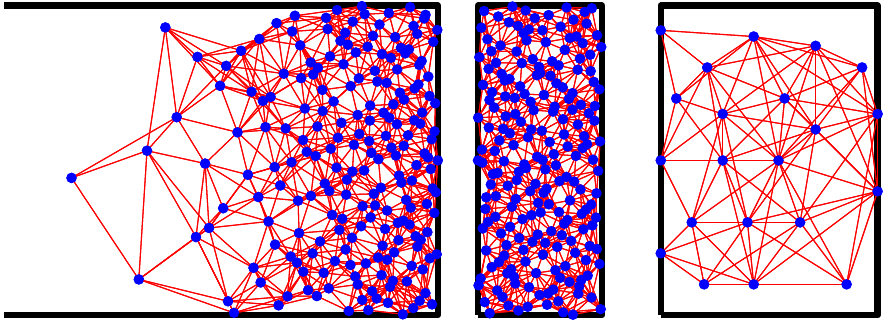}}
\end{center}
\vspace*{-8pt}\caption{ \rm Non-compact version of data set from Fig.~\ref{figure7}, density of leftmost `box' decays continuously to zero. (a) The kNN `AND' construction bridges the components while the 1-homology still has spurious generators (holes) in the sparse regions.  (b) The CkNN is capable of capturing the correct homology for this data set.  Both methods use $k=10$ and the value of $\delta$ was tuned to give the best results for each method.  Decreasing $\delta$ for (a) would obtain the correct clusters but introduce more spurious holes.}\label{SkNN2}\vspace*{-8pt}
\end{figure}

For example, we briefly consider the much less common `AND' construction for kNN, where points are connected when $d(x,y) \leq \min\{d(x,x_k),d(y,y_k)\}$ (as opposed to standard kNN which uses the $\max$).
Intuitively, the advantage of the `AND' construction is that it will not incorrectly connect dense regions to sparse regions because it takes the smaller of the two kNN distances.  However, on a non-compact domain, shown in Fig.~\ref{SkNN2}, this construction does not identify the correct homology whereas the CkNN does. We conclude the `AND' version of kNN is not generally superior to CkNN. Moreover, our analysis in Sec.~\ref{background} does not apply to this alternative, due to the fact that the $\max$ and $\min$ functions are not differentiable, making their analysis more difficult than the geometric average used by CkNN.

\subsection{Multiscale homology}\label{multiscaleSection}

In Section \ref{uniqueconstruction} we will see that the geometry represented by the CkNN graph construction captures the true topology with comparatively little data by implicitly choosing a geometry that is adapted to the sampling measure.  The CkNN construction yields a natural multi-scale geometry, which is assumed to be very smooth in regions of low density and can have finer features in regions of dense sampling.  Since all geometries yield the same topology, this multi-scale geometry is a natural choice for studying the topology of the underlying manifold, and this advantage is magnified for small data sets.  In Fig.~\ref{clusteringFig3} we demonstrate the effect of this geometry on the persistent homology for a small data set with multiple scales.  Following that, in Fig.~\ref{clusteringFig2} we show how the CkNN graph construction can capture the homology even for a non-compact manifold.

\begin{figure}
\begin{center}
\subfigure[]{\includegraphics[width=.4\linewidth]{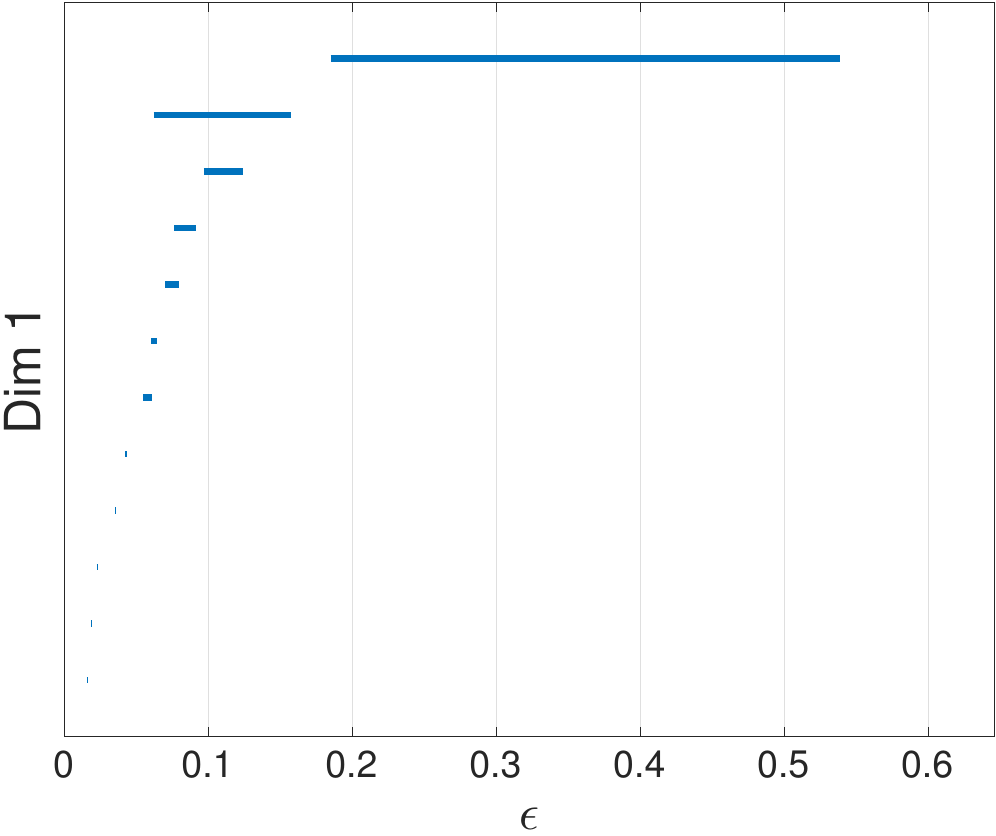}}\hspace{.1\linewidth}
\subfigure[]{\includegraphics[width=.4\linewidth]{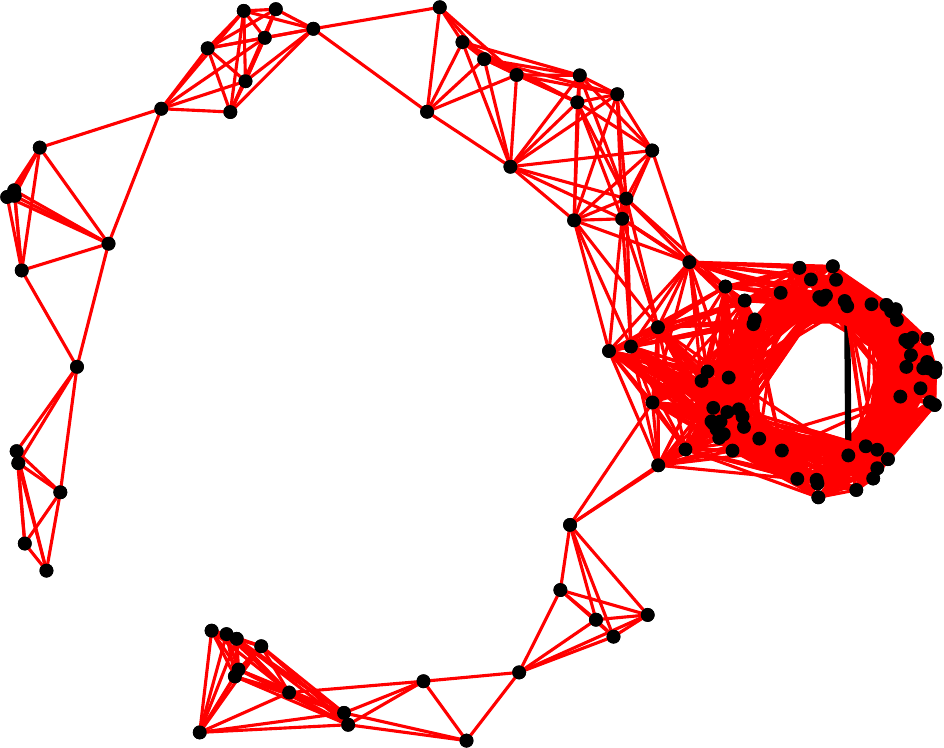}} \\
\subfigure[]{\includegraphics[width=.4\linewidth]{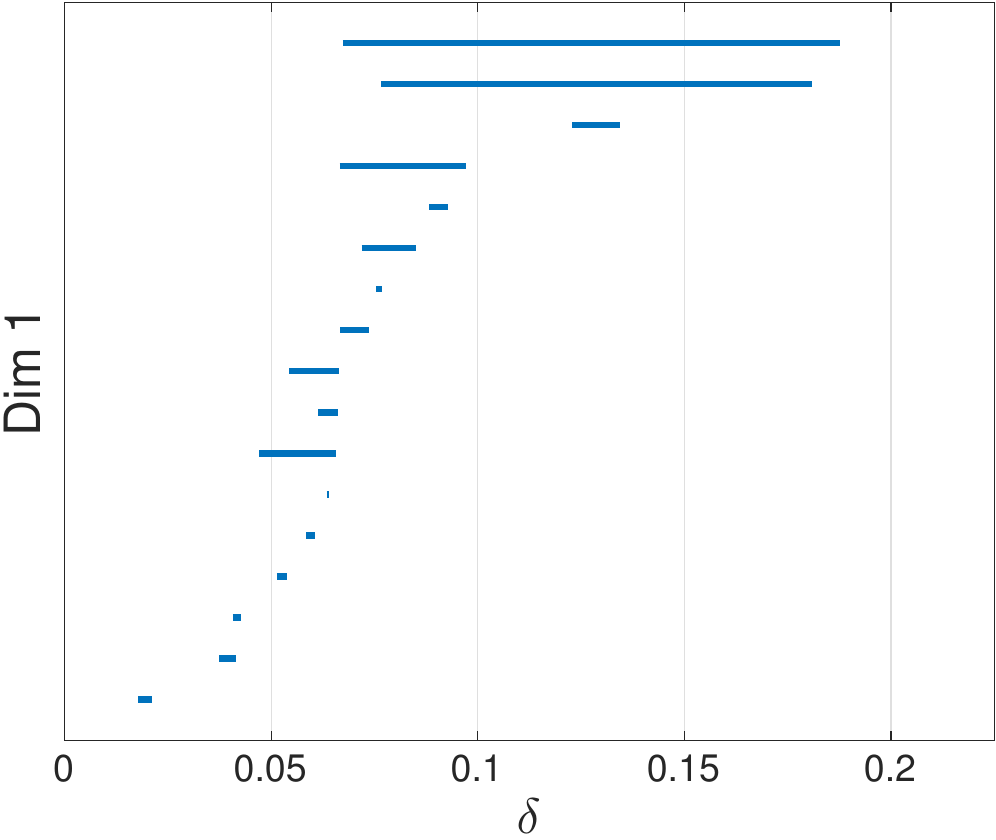}}\hspace{.1\linewidth}
\subfigure[]{\includegraphics[width=.4\linewidth]{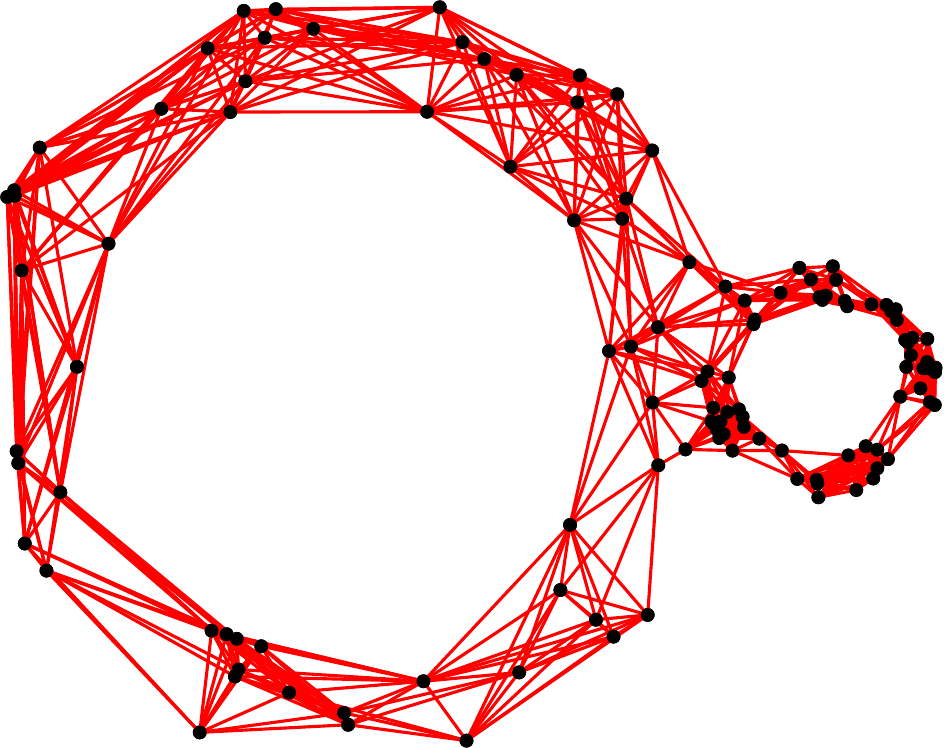}}
\end{center}
\vspace*{-6pt}\caption{ \rm (a) Standard $\epsilon$-ball persistent 1-homology reveals that no value of $\epsilon$ captures both holes simultaneously: Each homology class persists at separate $\epsilon$ scales. (b) When $\epsilon=0.16$ we see that the small hole is completely triangulated (the black segment completes the triangulation), while the large hole is still not connected. (c) CkNN persistent 1-homology as a function of $\delta$ shows that both homology classes (top two lines) are present simultaneously over a large range of $\delta$ values. (d) When $\delta=0.1$,  the resulting graph represents the true topology of the manifold.}
\label{clusteringFig3}\vspace*{-6pt}
\end{figure}

\begin{figure}
\begin{center}
\subfigure[]{\includegraphics[width=.32\linewidth]{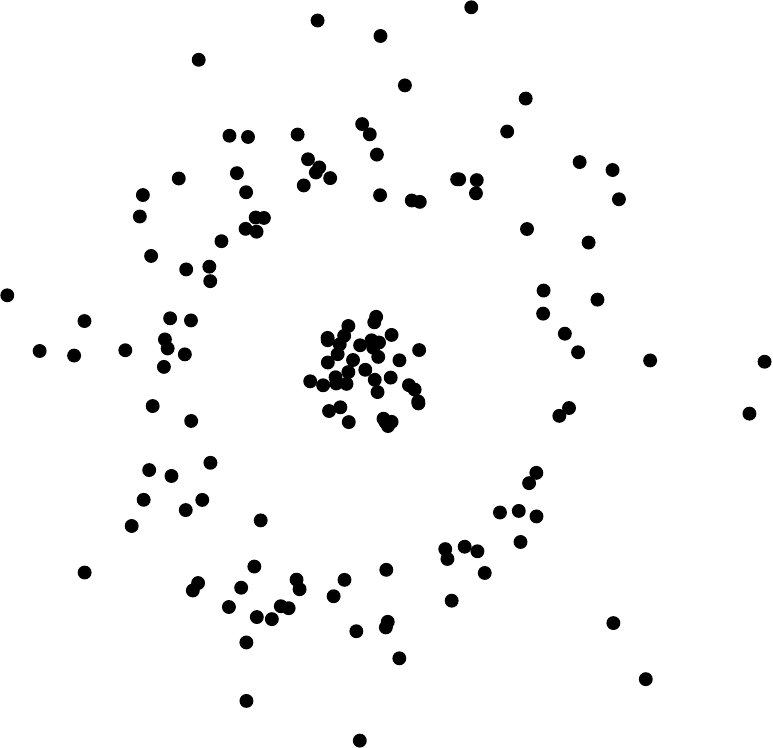}}\hspace{.1\linewidth}
\subfigure[]{\includegraphics[width=.38\linewidth]{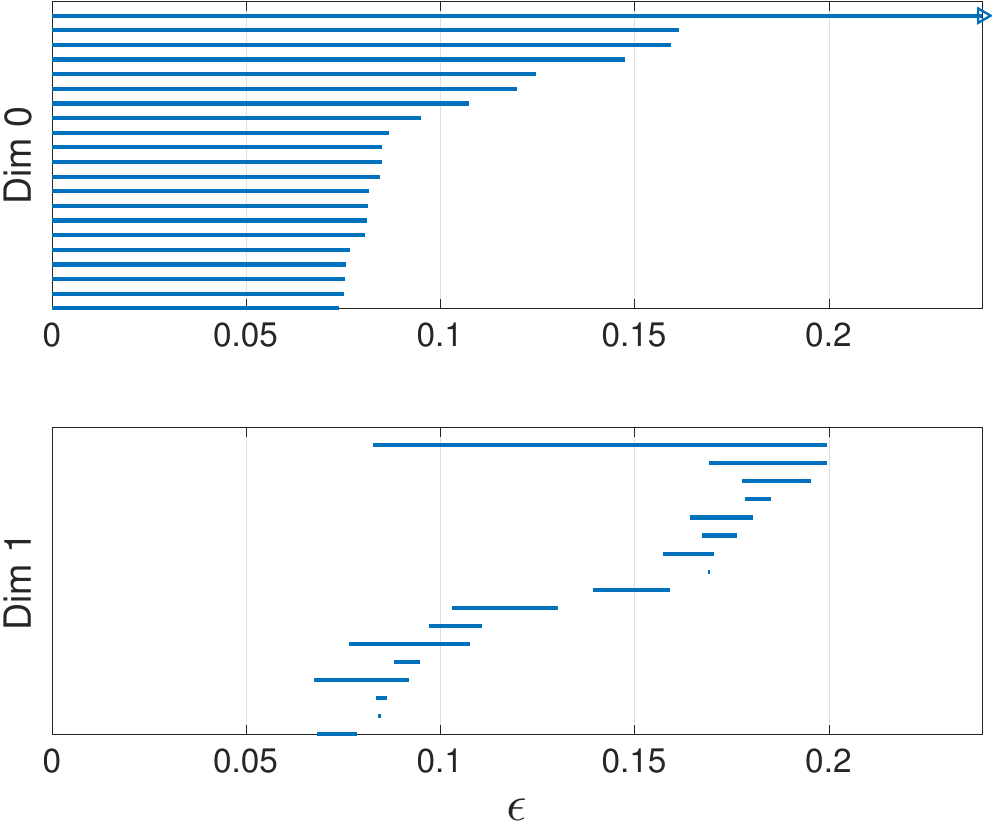}} \\
\subfigure[]{\includegraphics[width=.32\linewidth]{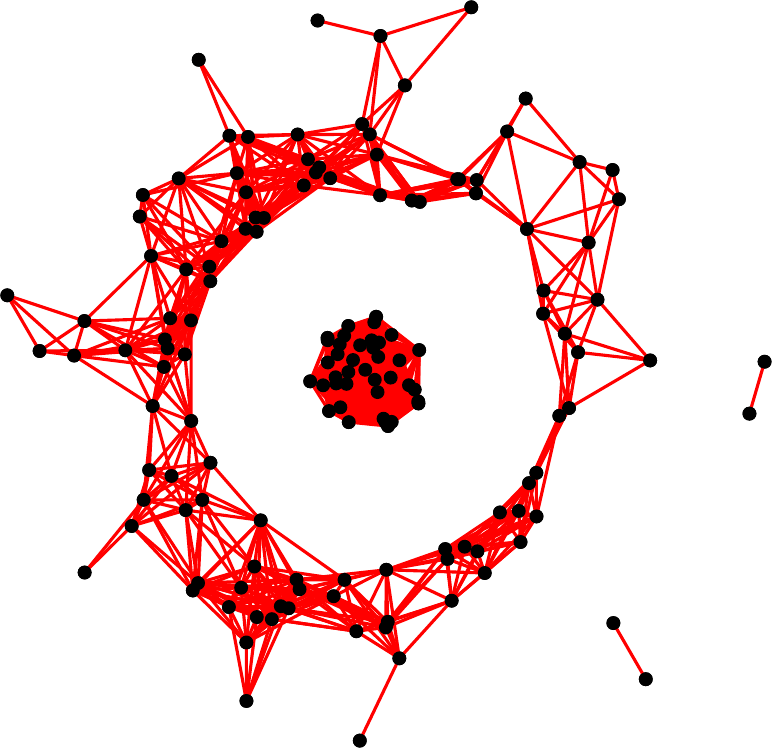}}\hspace{.16\linewidth}
\subfigure[]{\includegraphics[width=.32\linewidth]{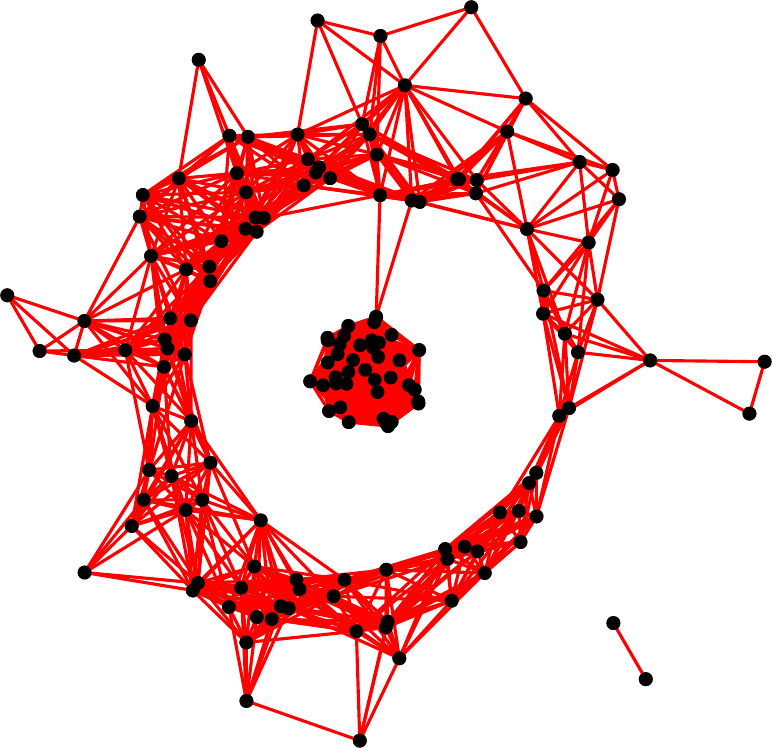}} \\
\subfigure[]{\includegraphics[width=.38\linewidth]{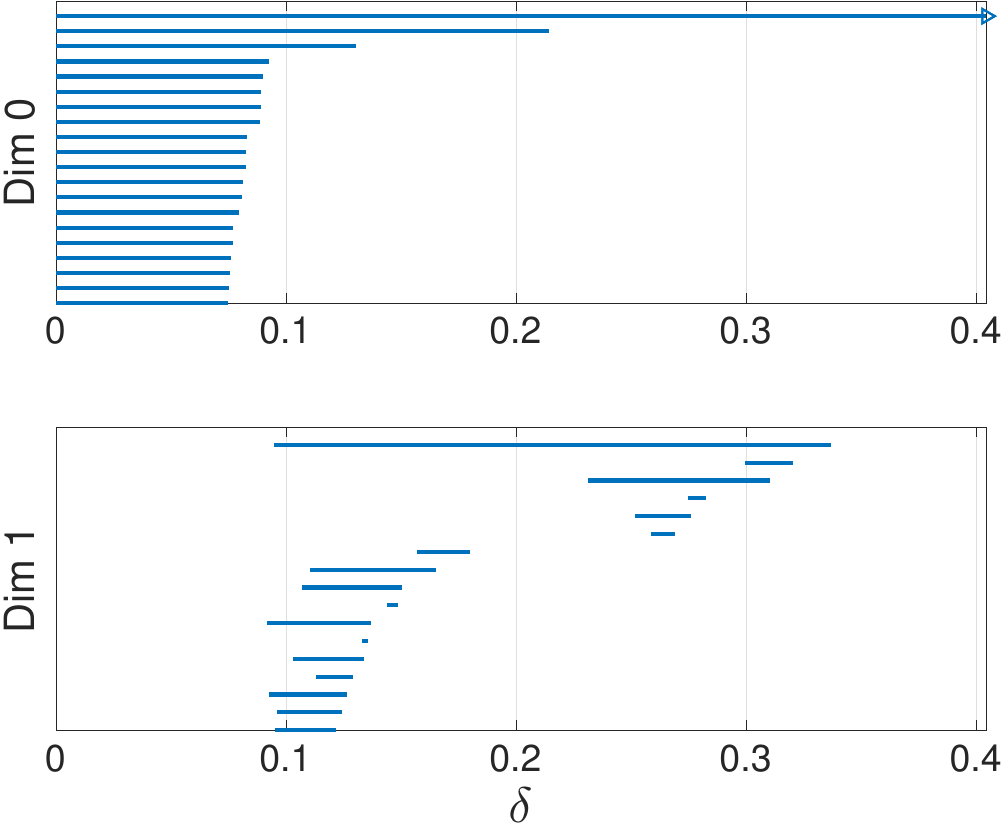}}\hspace{.1\linewidth}
\subfigure[]{\includegraphics[width=.32\linewidth]{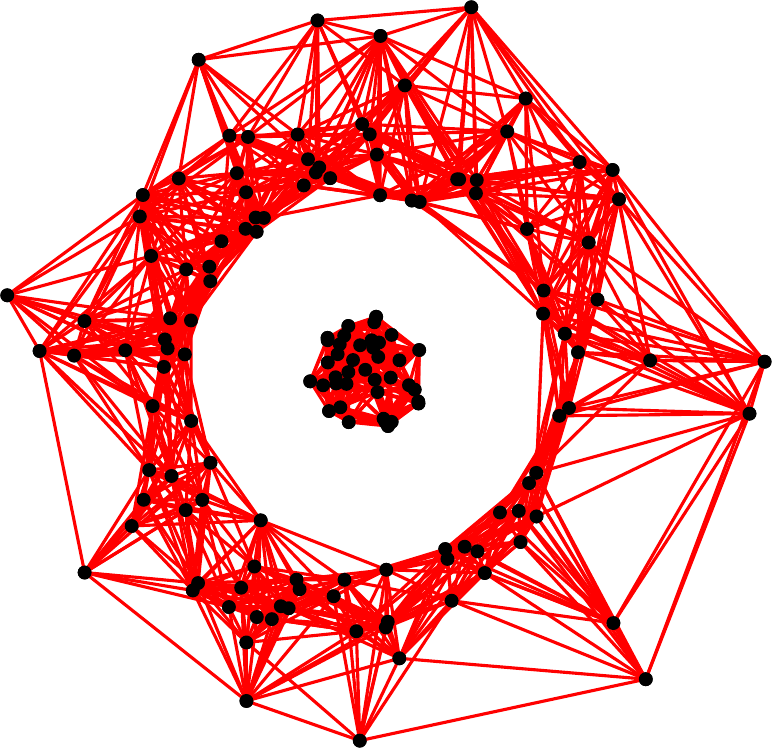}}
\end{center}
\caption{ \rm (a) Data sampled from a 2-dimensional Gaussian with a radial gap forming a non-compact manifold. (b) Standard fixed $\epsilon$-ball persistence diagram has no persistent region with 2 connected components and one non-contractible hole. (c) When $\epsilon=0.135$ the 1-homology is correct, but outliers are still not connected. (d) For $\epsilon=0.16$ the outliers are still not connected and the middle section is bridged.  For any size gap this bridge will happen in the limit of large data as the outliers become more spaced out.  (e) CkNN persistence in terms of $\delta$ shows the correct homology ($\beta_0=2$ and $\beta_1=1$) for $0.180<\delta<0.215$. (f) The CkNN construction for $\delta = 0.20$.}
\label{clusteringFig2}
\end{figure}

\begin{examp} \rm To form the data set in Fig.~\ref{clusteringFig3} we sampled 60 uniformly random points on a large annulus in the plane centered at $(-1,0)$ with radii in $[2/3,1]$ and another 60 uniformly random points on a much smaller annulus centered at $(1/5,0)$ with radii in $[1/5,3/10]$.  Together, these 120 points form a ``figure eight'' with a sparsely sampled large hole and a densely sampled small hole as shown in Fig.~\ref{clusteringFig3}(b)(d).  We then used the JavaPlex package \cite{Javaplex} to compute the $H^1$ persistence diagram for the VR complex based on the standard $\epsilon$-ball graph construction, shown in Fig.~\ref{clusteringFig3}(a).  Note that two major generators of $H^1$ are found along with some ``topological noise'', and the generators do not exist for a common $\epsilon$.

Since JavaPlex can build the VR complex persistence diagram for any distance matrix, we could easily compute the CkNN persistence diagram, shown in Fig.~\ref{clusteringFig3}(c), by using the `distance' matrix $d(x,y) = \frac{||x-y||}{\sqrt{||x-x_k|| \, ||y-y_k||}}$, and we used $k=10$ to form the matrix.  The CkNN captures both generators in a single graph, giving a multi-scale representation, whereas the standard $\epsilon$-ball graph captures only one scale at a time. Fig.~\ref{clusteringFig3}(d) shows the edges for $\delta = 0.1$, at which the graph captures the correct topology in all dimensions. Moreover, the CkNN construction is more efficient in this example, requiring only 934 edges to form a single connected component, whereas the $\epsilon$-ball construction requires 2306 edges.
\end{examp}

In the above examples, we saw that CkNN outperforms the standard $\epsilon$-ball approach, and kNN methods, for a given finite data set sampled nonuniformly from a compact manifold. However, in the large data limit, all of the above approaches converge to the correct homology, so the advantage of CkNN is only one of efficiency. Next we examine an example of non-compact data, where the CkNN construction converges to the correct topological information  in the large data limit, and the $\epsilon$-ball method cannot.

\begin{examp} \rm  \label{examp3} To form the data set in Fig.~\ref{clusteringFig2}(a) we sampled 150 points from a 2-dimensional Gaussian distribution and then removed the points of radius between $[1/4,3/4]$, leaving 120 points lying on two connected components with a single non-contractible hole.  In this case the standard $\epsilon$-ball persistence does not even capture the correct 0-homology for the manifold, due to the decreasing density near the outlying points, as shown in Fig.~\ref{clusteringFig2}(c-d).  Furthermore, since the true manifold is non-compact, there is no reason to expect the $\epsilon$-ball construction to converge to the correct topology even in the limit of large data.  In fact, as the amount of data is increased, the outlying points will become increasingly spaced out, leading to worse performance for the fixed $\epsilon$-ball construction, even for the $H^0$ homology.  In contrast, the CkNN construction is able to capture all the correct topological features for a large range of $\delta$ values, as shown in Fig.~\ref{clusteringFig2}(e-f).
\end{examp}

In Sections \ref{TDA} and \ref{uniqueconstruction} we prove that the CkNN is the unique graph construction that provides a consistent representation of the geometry of the underlying compact (or noncompact, under some technical assumptions) manifold in the limit of large data.  An immediate consequence is the consistency of the connected components.

\section{Manifold topology from graph topology}\label{TDA}

Our goal is to access the true topology of the underlying manifold using the Vietoris-Rips (VR) complex, which is an abstract simplicial complex on the finite data set. The VR complex is constructed inductively, first adding a triangle whenever all the faces are in the graph and then adding a higher order simplex whenever all the faces of the simplex are included.

\begin{Def} We say that a graph construction from a random sampling is \emph{topologically consistent} if the homology computed from the VR complex is isomorphic to the homology of the underlying manifold with probability approaching 1 as the number of data points goes to infinity.
\end{Def}

Topological consistency has been shown directly for the $\epsilon$-ball graph construction on compact manifolds without boundary \cite{VRcomplex,SmaleWeinberger,bobrowski2017topological}. However, Example \ref{examp3}  above shows why the $\epsilon$-ball construction cannot be guaranteed to be consistent for noncompact manifolds. In this section we delineate our notion of graph consistency for CkNN on Riemannian manifolds, compact and noncompact, based on spectral consistency of graph Laplacians with the Laplace-Beltrami operator.  Of course, this will establish geometric properties of the data that go beyond topological data analysis.

We first note that on a Riemannian manifold, the Laplace-Beltrami operator completely determines the Riemannian metric and thus the entire topology of a Riemannian manifold.  To make this connection explicit, in coordinates $x^1,...,x^m$ we can compute the Riemannian metric by
\[  g_{ij} = g_x(\nabla x^i,\nabla x^j) = \frac{1}{2}(\Delta(x^i x^j) - x^i \Delta x^j - x^j\Delta x^i). \]
The Riemannian metric $g$ is used to define the lengths of curves. In turn, the distance between points is defined as the infimum of these lengths over all curves between the points \cite{laplacianBook}.  The notion of distance determines the natural metric topology on the manifold.  In this formal way, perfect knowledge of the Laplace-Beltrami operator implies perfect knowledge of the topology.

Of course, this formal construction falls short of practical usage for two reasons. First, our discrete approximation does not yield perfect knowledge of the Laplace-Beltrami operator, but only convergence in a certain sense in the limit of large data. Second, extracting a given topological invariant from knowledge of the Laplace-Beltrami operator may be very difficult.

In this article, we are concerned mainly with addressing the first problem by showing the spectral convergence required to obtain certain types of topological invariants, namely cohomology.
In Sec.~\ref{background}, combined with a result of \cite{von2008consistency}, we prove pointwise and spectral convergence of the CkNN graph Laplacian to the Laplace-Beltrami operator, which we refer to as \emph{spectral consistency} of the graph Laplacian.
In particular, for a manifold without boundary or with a smooth boundary, convergence of the CkNN graph construction is guaranteed, and the manifold topology is uniquely determined.

The question remains, does spectral consistency to the Laplace-Beltrami operator imply topological consistency? The question can be visualized in the following diagram.

  \begin{center}\vspace{-15pt}
\[ \def\arraystretch{.3}  \begin{array}[c]{ccc}
\textup{Graph VR Homology }H_n(G)& \xdashrightarrow[\textup{consistency}]{\ \ \textup{topological} \ \ }
& \hspace{-10pt}\textup{Manifold Homology }H_n(\mathcal{M})\\ \\
\hspace{0pt}\scriptstyle{\textup{Graph Theory}}\left\uparrow\rule{0cm}{1cm}\right.   && \hspace{30pt}\left\uparrow\rule{0cm}{1cm}\right. \scriptstyle{\textup{Hodge Theory + SEC \cite{berry2018spectral} }}
\\ \\
\textup{Graph Laplacian }L_{\rm un}=\partial\partial^\top &\xrightarrow[\textup{consistency}]{\ \ \ \textup{spectral} \ \ \ }& \hspace{-20pt}\textup{Laplace-Beltrami }\Delta = \delta d
\end{array}\]
\end{center}

The left vertical arrow corresponds to the fact that the graph Laplacian $L_{un} = D-W$ can trivially be used to reconstruct the entire graph since the non-diagonal entries are simply the negative of the adjacency matrix.  Since the graph determines the entire VR complex, the graph Laplacian completely determines the VR homology of the graph.  This connection is explicitly spectral for the zero homology, since the zero-homology of a graph corresponds exactly to the zero-eigenspace of the graph Laplacian \cite{von2007tutorial}.

For higher homology, one would need to establish that the VR homology computed from the graph Laplacian agrees with the homology uniquely determined by the Laplace-Beltrami operator (which is uniquely determined by the converging CkNN graph Laplacians). We next explore the plausibility of this conjecture.

The right vertical arrow follows from the fact that from the Laplace-Beltrami operator, it is possible to reconstruct the metric on a Riemannian manifold which completely determines the homology of the manifold.  The Riemannian metric lifts to an inner product on forms $g(\omega,\nu)$ which defines the Hodge inner products
\[ \left<\omega,\nu\right> = \int_{\mathcal{M}} g(\omega,\nu) \, dV \]
on differential forms.  From the Hodge inner products, one defines the codifferential operators $\delta^k$ as the formal adjoint of the exterior derivative $d^k$ on $k$-forms.  Finally, the Laplace-de Rham operator on differential $k$-forms is defined by $\Delta^k = \delta^{k+1}d^k + d^{k-1}\delta^k$, and the Hodge theorem \cite{laplacianBook} states that the kernel of $\Delta^k$ is isomorphic to the $k$-th de Rham cohomology group and the $k$-th singular cohomology group, $\textup{ker}(\Delta^k) \cong H^k_{\textup{dR}}(\mathcal{M}) \cong H^k_{\textup{sing}}(\mathcal{M})$.  For closed manifolds without boundary, Poincare duality relates the homology to the cohomology $H_{n-k}(\mathcal{M}) \cong H^k(\mathcal{M})$. In general, cohomology is considered a stronger invariant.

A partial solution to this problem is given in \cite{berry2018spectral}, with a method of extracting cohomology information that relies on constructing estimators of Laplace-de Rham operators on $k$-forms.  The construction of these estimators  and their consistency relies on the spectral convergence results shown here.
 This construction is called the \emph{spectral exterior calculus} (SEC) since the Laplace-de Rham operators are approximated by Galerkin truncation on a spectral basis (as opposed to a finite element basis) constructed using the eigenfunction of the Laplace-Beltrami operator. While a finite element basis would require a simplicial complex, the SEC construction is valid for an abstract complex such as the VR complex, since it relies only on the spectral consistency of the graph Laplacian.  Estimators of the Laplace-de Rham operators on $1$-forms are constructed and are shown to converge spectrally, thus insuring that the kernel of these estimators yields the true cohomology in the limit of large data.  In \cite{berry2018spectral} a general strategy is outlined for lifting the results to the Laplace-de Rham operators on $k$-forms for $k>1$. The restriction to $1$-forms was simply due to the complexity of the explicit formulations.  The SEC together with the spectral convergence of the graph Laplacian to the Laplace-Beltrami operator completes the right vertical arrow by giving an explicit construction.

\vspace*{3pt}
An alternative approach that has not been fully explored yet would be to construct the Laplace-de Rham operators on $k$-forms directly from the VR-complex. This is an alternative to the SEC construction which represents these operators on a basis built from eigenfunction of the Laplace-Beltrami operator.  Such a construction would allow for the top horizontal arrow to be shown directly.  Establishing this connection requires defining a discrete analog of differential forms and discrete analogs of the higher-order Laplace-de Rham operators and then showing spectral convergence to the corresponding operators on the manifold.  One promising construction is the discrete exterior calculus \cite{DEC1,DEC2} but so far only very restricted consistency results have been shown \cite{SPEX,schulz2016convergence}.

\vspace*{3pt}
To summarize the above discussion, just as the discrete Laplacian completely determines the graph VR homology, the Laplace-Beltrami operator $\Delta$ completely determines the manifold homology.  In perfect analogy to the discrete case, the zero-homology of the manifold corresponds to the zero-eigenspace of the Laplace-Beltrami operator.  This connection is explicitly spectral, and we conjecture that the isomorphism generalizes to higher-order homology groups via the Laplace-de Rham operators on differential forms.

\vspace*{3pt}
Similar results exist for weighted graphs, which are less helpful to topological data analysis because they are not directly interpretable through VR calculations.  For example, \cite{diffusion} proves pointwise convergence on compact manifolds for any smooth sampling density, but their construction requires a weighted graph in general.  For graph Laplacian methods (including the results developed here), the dependence on the curvature and nearness to self-intersection appears in the bias term of the estimator as shown in \cite{diffusion,heinthesis}.
 Using more complicated weighted graph constructions, recent results show that for a large class of non-compact manifolds \cite{heinthesis} with smooth sampling densities that are allowed to be arbitrarily close to zero \cite{BH14}, the graph Laplacian can converge pointwise to the Laplace-Beltrami operator.  These results require weighted graph constructions due to several normalizations which are meant to remove the influence of the sampling density on the limiting operator.
\vspace*{-8pt}
\section{The unique consistent unweighted graph construction}\label{uniqueconstruction}

In this section we show that the continuous k-nearest neighbors (CkNN) construction is the unique unweighted graph construction that yields a consistent unweighted graph Laplacian for any smooth sampling density on manifolds of the class defined in \cite{heinthesis}, including many non-compact manifolds.

Consider a data set $\{x_i\}_{i=1}^N$ of independent samples from a probability distribution $q(x)$ that is supported on a $m$-dimensional manifold $\mathcal{M}$ embedded in Euclidean space.  For a smooth function $\rho(x)$ on $\mathcal{M}$, we will consider the CkNN graph construction, where two data points $x_i$ and $x_j$ are connected by an edge if
\begin{equation}\label{multiscale} d(x_i,x_j) < \delta \sqrt{\rho(x_i)\rho(x_j)}. \end{equation}
This construction leads to the $N\times N$ adjacency matrix $W$ whose $ij$th entry is $1$ if $x_i$ and $x_j$ have an edge in common, and $0$ otherwise. Let $D$ be the diagonal matrix of row sums of $W$, and define the ``unnormalized'' graph Laplacian $L_{\rm un} = D-W$.
In Sec.~\ref{background} we show that for an appropriate factor $c$ depending only on $\delta$ and $N$, in the limit of large data,  $c^{-1}L_{\rm un}$ converges both pointwise and spectrally to the operator defined by
\begin{equation}\label{LunOperator} \mathcal{L}_{q,\rho}f \equiv q \rho^{m+2} \left(\Delta f - \nabla \log\left(q^2\rho^{m+2}\right) \cdot \nabla f \right),  \end{equation}
where $\Delta$ is the positive definite Laplace-Beltrami operator and $\nabla$ is the gradient, both with respect to the Riemannian metric inherited from the ambient space.  In fact, pointwise convergence follows from a theorem of \cite{Ting2010}, and the pointwise bias and a high probability estimate of the variance was first computed in \cite{BH14}. Both of these results follow for a larger class of kernels than the one defined in (\ref{multiscale}).

Although the operator in \eqref{LunOperator} appears complicated, we now show that it is still a Laplace-Beltrami operator on the same manifold $\mathcal{M}$, but with respect to a different metric. A {\it conformal} change of metric corresponds to a new Riemannian metric $\tilde g \equiv \varphi g$, where $\varphi(x)>0$, and which has Laplace-Beltrami operator
\begin{equation}\label{Deltilde} \Delta_{\tilde g}f = \frac{1}{\varphi}(\Delta f - (m-2)\nabla \log \sqrt{\varphi} \cdot \nabla f).  \end{equation}
For expressions \eqref{LunOperator} and \eqref{Deltilde} to match, the function $\varphi$ must satisfy
\begin{equation}\label{conformal}\frac{1}{\varphi} = q\rho^{m+2} \ \ \ \ \ {\rm and}\ \ \ \ \ \varphi^{\frac{m-2}{2}} = q^2\rho^{m+2}.\end{equation}
Eliminating $\varphi$ in the two equations results in $q^{-(m+2)} = \rho^{m(m+2)}$, which implies
 $\rho \equiv q^{-\frac{1}{m}}$ as the only choice that makes the operator \eqref{LunOperator} equal to a Laplace-Beltrami operator $\Delta_{\tilde g}$. The new metric is $\tilde g = q^{2/m} g$.  Computing the volume form $d\tilde V$ of the new metric ${\tilde g}$ we find
\begin{equation}\label{volform} d\tilde V = \sqrt{|\tilde g|} = \sqrt{|q^{2/m}g|} = q \sqrt{|g|} = q\, dV \end{equation}
which is precisely the sampling measure.  Moreover, the volume form $d\tilde V$ is exactly consistent with the discrete inner product
\[  \mathbb{E}\left[\vec f \cdot \vec f \right] = \mathbb{E}\left[\sum_{i=1}^N f(x_i)^2\right] = N\int f(x)^2q(x)\, dV = N\left<f,f\right>_{d\tilde V}. \]
This consistency is crucial since the discrete spectrum of $L_{\rm un}$ are the minimizers of the functional
\[ \Lambda(f) = \frac{\vec f \,^\top c^{-1}L_{\rm un} \vec f}{\vec f \,^\top \vec f} \to_{N\to\infty} \frac{\left<f,\Delta_{\tilde g}f \right>_{d\tilde V}}{\left<f,f\right>_{d\tilde V}} \]
where $c$ is a scalar depending only on $\delta$ and $N$ (see Theorem \ref{pointwiseLun} in Sec.~\ref{background}).  If the Hilbert space norm implied by $\vec f \,^\top \vec f$ were not the volume form of the Riemannian metric $\tilde g$, then the eigenvectors of $L_{\rm un}$ would not minimize the correct functional in the limit of large data.  This shows why it is important that the Hilbert space norm is consistent with Laplace-Beltrami operator estimated by $L_{\rm un}$.

Another advantage of the geometry $\tilde g$ concerns the spectral convergence of $L_{\rm un}$ to $\mathcal{L}_{q,\rho}$ shown in Theorem \ref{spectralconv} in Sec.~\ref{background}, which requires the spectrum to be discrete.  Assuming a smooth boundary, the Laplace-Beltrami operator on any manifold with finite volume will have a discrete spectrum \cite{cianchi2011}.  This insures that spectral convergence always holds for the Riemannian metric $\tilde g = q^{2/m}g$ since the volume form is $d\tilde V = q dV$, and therefore the volume of the manifold is exactly
\[ \textup{vol}_{\tilde g}(\mathcal{M}) = \int_{\mathcal{M}} d\tilde V = \int_{\mathcal{M}} q \, dV = 1. \]
Since all geometries on a manifold have the same topology, this shows once again that the metric $\tilde g$ is completely natural for topological investigations since spectral convergence is guaranteed, and spectral convergence is crucial to determining the homology.

The fact that $\rho=q^{-1/m}$ is the unique solution to \eqref{conformal} along with the spectral consistency implies the following result.
\begin{thm}[Unique consistent geometry]\label{consgeom} Consider data sampled from a compact Riemannian manifold satisfying Assumption 1 (see Section \ref{background}).  Among unweighted graph constructions \eqref{multiscale}, $\rho = q^{-1/m}$ is the unique choice which yields a consistent geometry in the sense that the unnormalized graph Laplacian converges spectrally to a Laplace-Beltrami operator.  Thus, the CkNN graph construction yields a consistent clustering for any density $q$.
\end{thm}
Based on the results in Section \ref{background}, we conjecture that spectral convergence also holds for non-compact manifolds.  This would imply that CkNN is the unique consistent graph construction for clustering, since for any other $\rho$, there will exist densities $q$ where $\mathcal{M}$ has infinite volume with respect to $q^{\frac{4}{m-2}}\rho^{\frac{2m+4}{m-2}}dV$, precluding consistency.  Based on the discussion in the previous section, and in particular the fact that the Laplace-Beltrami operator determines the entire cohomology of the manifold, we propose the following conjecture:
\begin{conj}[Unique consistent topology] CkNN is the unique graph construction \eqref{multiscale} with an associated VR complex that is topologically consistent in the limit of large data.
\end{conj}
As mentioned in the previous section, completing the proof of this conjecture would require explicit estimation of the Laplace-de Rham operators, $\Delta^k$, and corresponding spectral convergence proofs.  In this paper we establish this fact for the Laplace-Beltrami operator, $\Delta =\Delta^0$, which is the key first step to supporting the conjecture.  The spectral exterior calculus (SEC) \cite{berry2018spectral} provides a construction which lifts this results to general $\Delta^k$, and explicitly proves spectral convergence for $\Delta^1$ which is the most challenging theoretical barrier.  The only remaining challenge to lifting the result to $\Delta^k$ for all $k$ is simply the problem of formulating the explicit construction.

The theorem and conjecture above have practical ramifications since (as shown in the previous section) a consistent graph construction will have the same limiting topology as the underlying manifold.  In Examples \ref{ex2} and \ref{ex3} we will empirically illustrate the consistency of the CkNN choice $\rho=q^{-1/m}$ as well as the failure of alternative constructions.

As a side note, we mention that the unnormalized graph Laplacian is not the only graph Laplacian. With the same notation as above, the ``normalized'', or ``random-walk'' graph Laplacian is often defined as $L_{\rm rw} = I - D^{-1}W = D^{-1}L_{\rm un}$, and has the limiting operator
\[ c^{-1}L_{\rm rw} \equiv c^{-1}D^{-1}L_{\rm un}  \to_{N\to\infty}  \rho^{2} \left(\Delta - \nabla \log\left(q^2\rho^{m+2}\right) \cdot \nabla \right) = q^{\frac{4}{m-2}}\rho^{\frac{4m}{m-2}} \Delta_{\tilde g} \]
(see for example \cite{Ting2010,BH14}; the constant $c$ is different from the unnormalized case).
Note that again $\rho = q^{-1/m}$ is the unique choice leading to a Laplace-Beltrami operator.  This choice implies that to leading order $D\vec f \approx q\rho^m f = f$, so the corresponding Hilbert space has norm $\vec f\,^\top D \vec f \to_{N\to\infty} \left<f,f\right>_{d\tilde V}$.  This implies spectral consistency since $L_{\rm rw}\vec f = \lambda \vec f$ is equivalent to $L_{\rm un}\vec f = \lambda D\vec f$, which is related to the functional
\[ \Lambda(f) = \frac{\vec f \,^\top c^{-1}L_{\rm un} \vec f}{\vec f \,^\top D \vec f} \to_{N\to\infty} \frac{\left<f,\Delta_{\tilde g}f \right>_{d\tilde V}}{\left<f,f\right>_{d\tilde V}}. \]
Therefore, for the choice $\rho = q^{-1/m}$, both the unnormalized and normalized graph Laplacians are consistent with the same underlying Laplace-Beltrami operator with respect to the metric $\tilde g$.

We emphasize that we are not using either graph Laplacian directly for computations. Instead, we are using the convergence of the graph Laplacian to show convergence of the graph connected components to those of the underlying manifold.  Since this consistency holds for an unweighted graph construction, we can make use of more computationally efficient methods to find the topology of the graph, such as depth-first search to compute the zero-level homology.  More generally we compute the higher order homology from the VR complex of the graph, which our conjecture suggests should converge to that of the underlying manifold.  A wider class of geometries are accessible via weighted graph constructions (see for example \cite{diffusion,BH14,localk}), but fast algorithms for analyzing graph topology only apply to unweighted graphs.

\section{Further applications to topological data analysis} \label{spectralclustering}

The fundamental idea of extracting topology from a point cloud by building unweighted graphs relies on determining what is considered an edge in the graph as a function of a parameter, and then considering the graph as a VR complex.
In the $\epsilon$-ball, kNN  and CkNN procedures, edges are added as a parameter is increased, from no edges for extremely small values of the parameter to full connectivity for sufficiently large values. From this point of view, the procedures differ mainly by the order in which the edges are added.

Classical persistence orders the addition of possible edges by $||x-y||$, whereas the CkNN orders the edges by $\frac{||x-y||}{\sqrt{||x-x_k|| \, ||y-y_k||}}$.  More generally, a multi-scale graph construction with bandwidth function $\rho(x)$ orders the edges by $\frac{||x-y||}{\sqrt{\rho(x)\rho(y)}}$. Our claim is that CkNN gives an order that allows graph consistency to be proved. In addition, we have seen in Figs. \ref{clusteringFig3} and \ref{clusteringFig2} that the CkNN ordering is more efficient. In this section we show further examples illustrating this fact.
We will quantify the persistence or stability of a feature by the percentage of edges (out of the total $N(N-1)/2$ possible in the given ordering) for which the feature persists.  This measure is an objective way to compare different orderings of the possible edges.

 The consistent homology approach differs from the persistent homology approach by using a single graph construction to simultaneously represent all the topological features of the underlying manifold.  This requires selecting the parameter $\delta$ which determines the CkNN graph.  The asymptotically optimal choice of $\delta$ in terms of the number of data points is derived in Sec.~\ref{background}, however the constants depend on the geometry of the unknown manifold. There are many existing methods of tuning $\delta$ for learning the geometry of data \cite{epsilontuning,BH14}.  As a practical method of tuning $\delta$ for topological data analysis, we can use the classical persistence diagram to find the longest range of $\delta$ values such that all of the homological generators do not change.  In a sense we are using the classical persistence diagram in reverse, looking for a single value of $\delta$ where all the homology classes are stable.  In Examples \ref{ex2} and \ref{ex3} below we validate this approach by showing that the percentage of edges which capture the true homology is longest when using ordering defined by $\rho=q^{-1/m}$ which is equivalent to the CkNN.

\subsection{A fast graph-based clustering algorithm}\label{clustering}

We consider the problem of identifying the connected components of a manifold from a data set using the connected components of the CkNN graph construction.  While clustering connected components is generally less difficult than segmenting a connected domain, outliers can easily confuse many clustering algorithms.  Many rigorous methods, including any results based on existing kernel methods \cite{heinHighDensity2,ZP,von2008consistency,heinCuts1}, require the sampling density to be bounded away from zero.  In other words, rigorous clustering algorithms require the underlying manifold to be compact.  A common work-around for this problem is the estimate the density of the data points and then remove points of low density. However, this leaves the removed points unclustered \cite{highdensity1rigourous,highdensity2,highdensity3,heinHighDensity1}.  We have shown that the CkNN method is applicable to a wide class of non-compact manifolds, and in particular the connected components of the CkNN graph will converge to the connected components of the underlying manifold.

Here we use the CkNN to put an ordering on the potential edges of the graph. While the full persistent homology of a large data set can be computationally very expensive, the 0-homology is easily accessible using fast graph theoretic algorithms.  First, the connected components of a graph can be quickly identified by the depth-first search algorithm.  Second, unlike the other homology classes, the 0-homology is monotonic; as edges are added, the number of connected components can only decrease or stay the same.  This monotonicity allows us to easily identify the entire 0-homology $\delta$-sequence by only finding the transitions, meaning the numbers of edges where the $0$-th Betti number changes.  We can quickly identify these transitions using a binary search algorithm as outlined below.

\noindent\makebox[\linewidth]{\rule{\textwidth}{0.8pt}}
{{\bf Algorithm 1. }  Fast Binary Search Clustering.}
\newline
\vspace{-15pt}

\noindent\makebox[\linewidth]{\rule{\textwidth}{0.8pt}}

\vspace{0pt}
\textbf{Inputs:} Ordering of the $N(N-1)/2$ possible edges, number of clusters $C>1$.

\textbf{Outputs:} Number of edges, $L$, such that the graph has $C$ components and adding an edge yields $C-1$ components.

\vspace{3pt}
\begin{enumerate}	
\item Initialize the endpoints $L=0$ and $R=N(N-1)/2$
\item while $L < R-1$
	\begin{enumerate}
	\item Set $M = \textup{floor}((L+R)/2)$	
	\item Build a graph using the first $M$ edges from the ordering
	\item Use depth-first search to find the number of components $\tilde C$
	\item If $\tilde C \geq C$ set $L=M$ otherwise set $R=M$
	\end{enumerate}
\item return $L$.
\end{enumerate}\vspace{-10pt}
\noindent\makebox[\linewidth]{\rule{\textwidth}{0.8pt}}

When the goal is to find all of the transition points, Algorithm 1 can easily be improved by storing all the numbers of clusters from previous computations and using these to find the best available left and right endpoints for the binary search.

\begin{examp} \rm In Fig.~\ref{spiralsFig}, we illustrate the use of Algorithm 1 on a point set consisting of the union of three spiral-shaped subsets with nonuniform sampling. In fact, the density of points falls off exponentially in the radial direction. Fig.~\ref{spiralsFig}(a) shows the original set, and panel (b) shows the number of components as a function of the proportion of edges. When the number of edges is between one and two percent of the possible pairs of points, the persistence diagram in (b) detects three components, shown in (c) along with the edges needed. A three-dimensional version of three spiral-shaped subsets is depicted in Fig.~\ref{spiralsFig}(d)-(f), with similar results.

\begin{figure}
\begin{center}
\subfigure[]{\includegraphics[width=.3\linewidth]{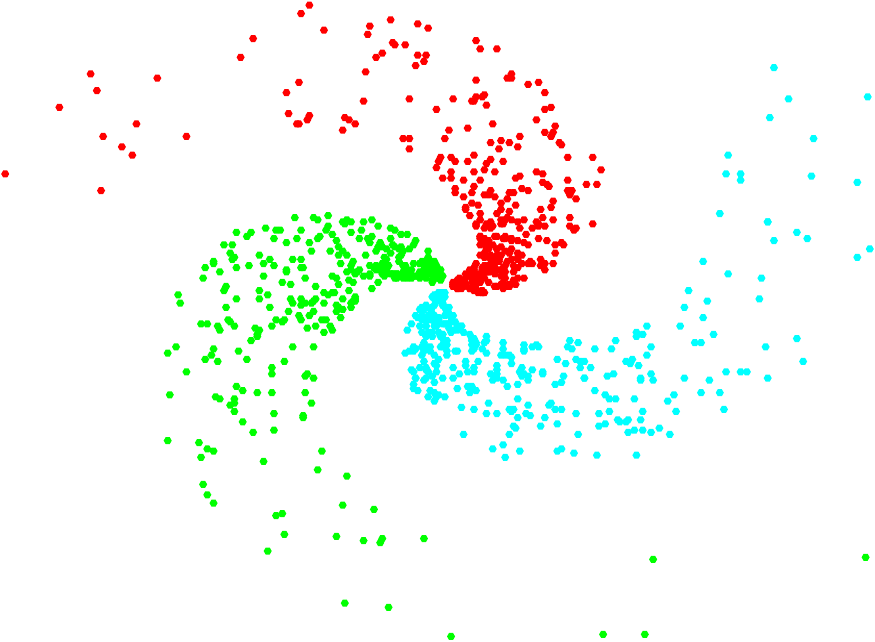}}\hspace*{.02\linewidth}
\subfigure[]{\includegraphics[width=.26\linewidth]{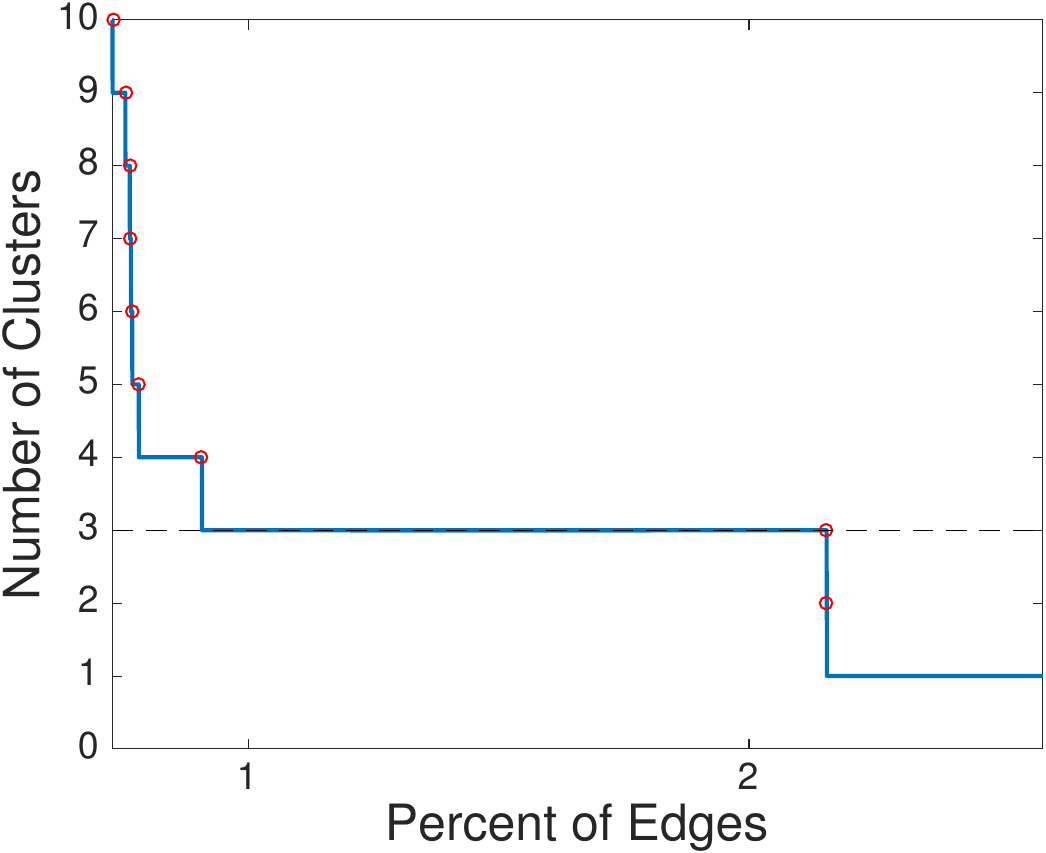}}\hspace*{.02\linewidth}
\subfigure[]{\includegraphics[width=.24\linewidth]{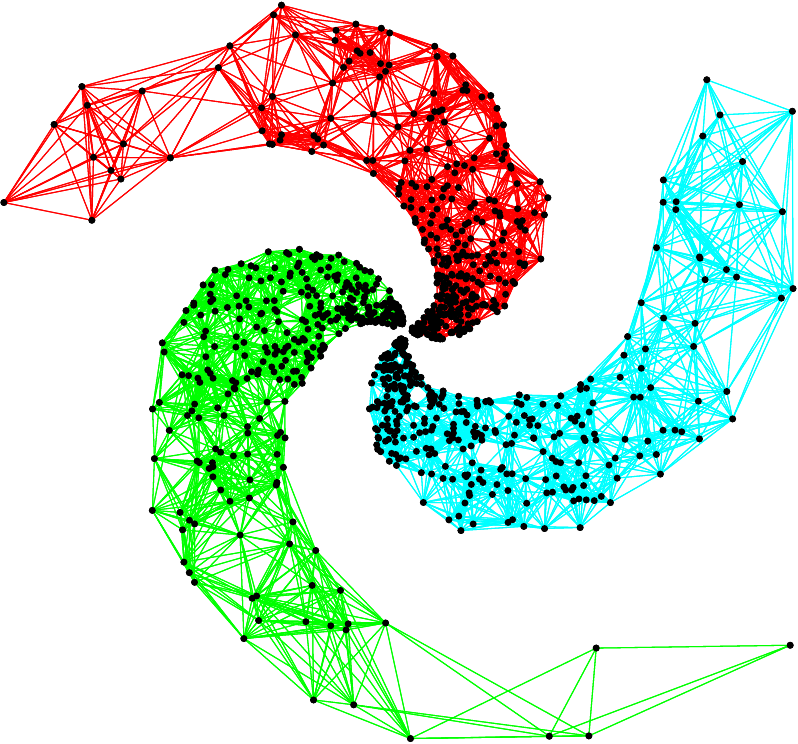}}
\subfigure[]{\includegraphics[width=.3\linewidth]{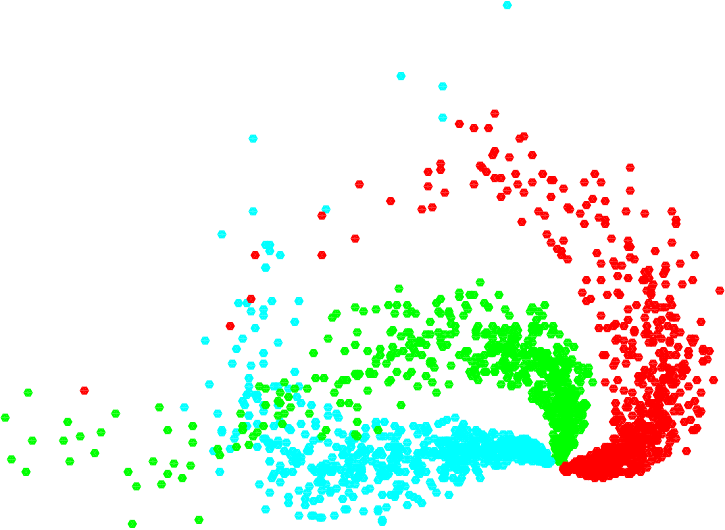}} \hspace*{.043\linewidth}
\subfigure[]{\includegraphics[width=.26\linewidth]{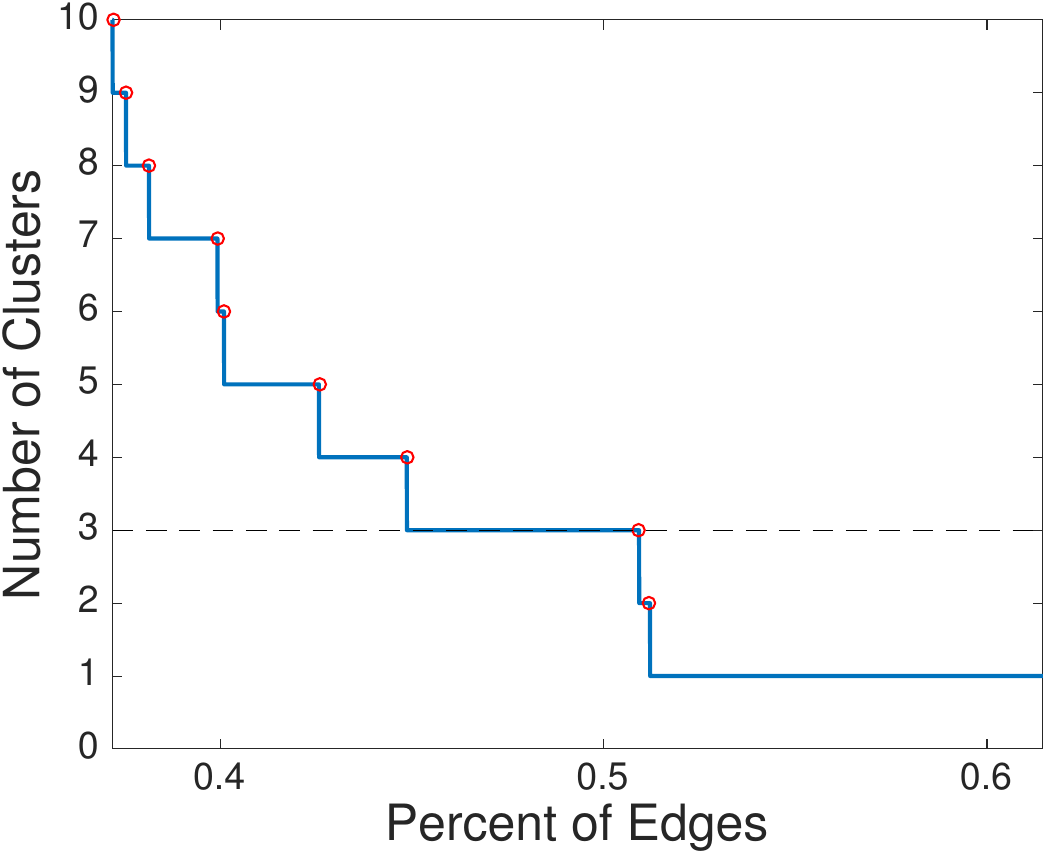}}\hspace*{.02\linewidth}
\subfigure[]{\includegraphics[width=.27\linewidth]{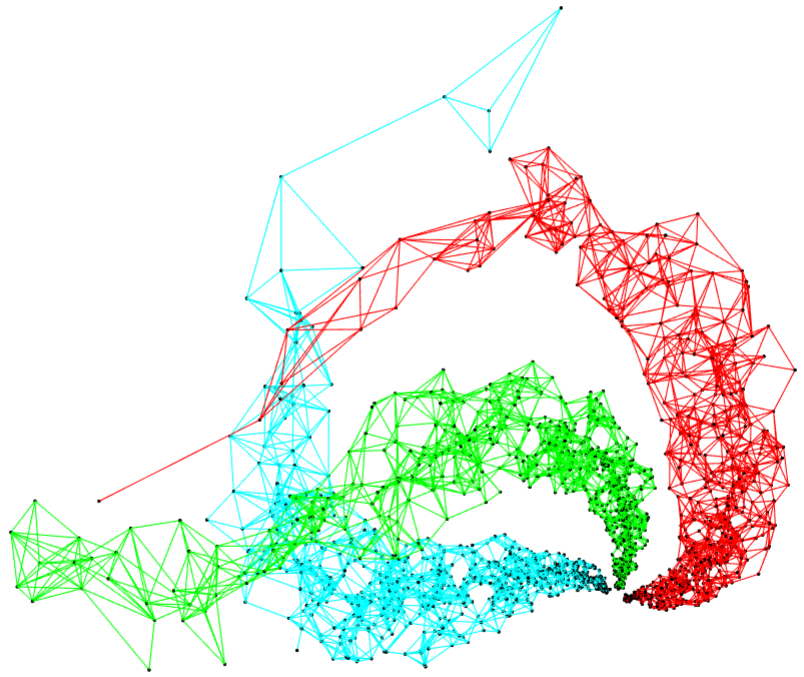}}
\end{center}
\caption{ \rm The fast clustering algorithm applied to a set of three spirals with densities which decay exponentially along the length of the spiral.  (a) 1000 total points sampled from three 2-dimensional spirals in the plane. (b) Diagram shows number of clusters as a function of the percentage of possible edges added to the graph (a unitless measure of persistence) (c) Graph corresponding to the maximum number of edges with 3 components, colored according to identification by depth-first-search.  (d) 2000 total points sampled from three 3-dimensional spirals in $\mathbb{R}^3$.  (e) Large interval identifies the correct number of clusters. (f) Graph with maximum number of edges having 3 components, colored by clustering algorithm.\label{spiralsFig}}
\end{figure}

\end{examp}

\begin{figure}
\begin{center}
\subfigure[]{\includegraphics[width=.38\linewidth]{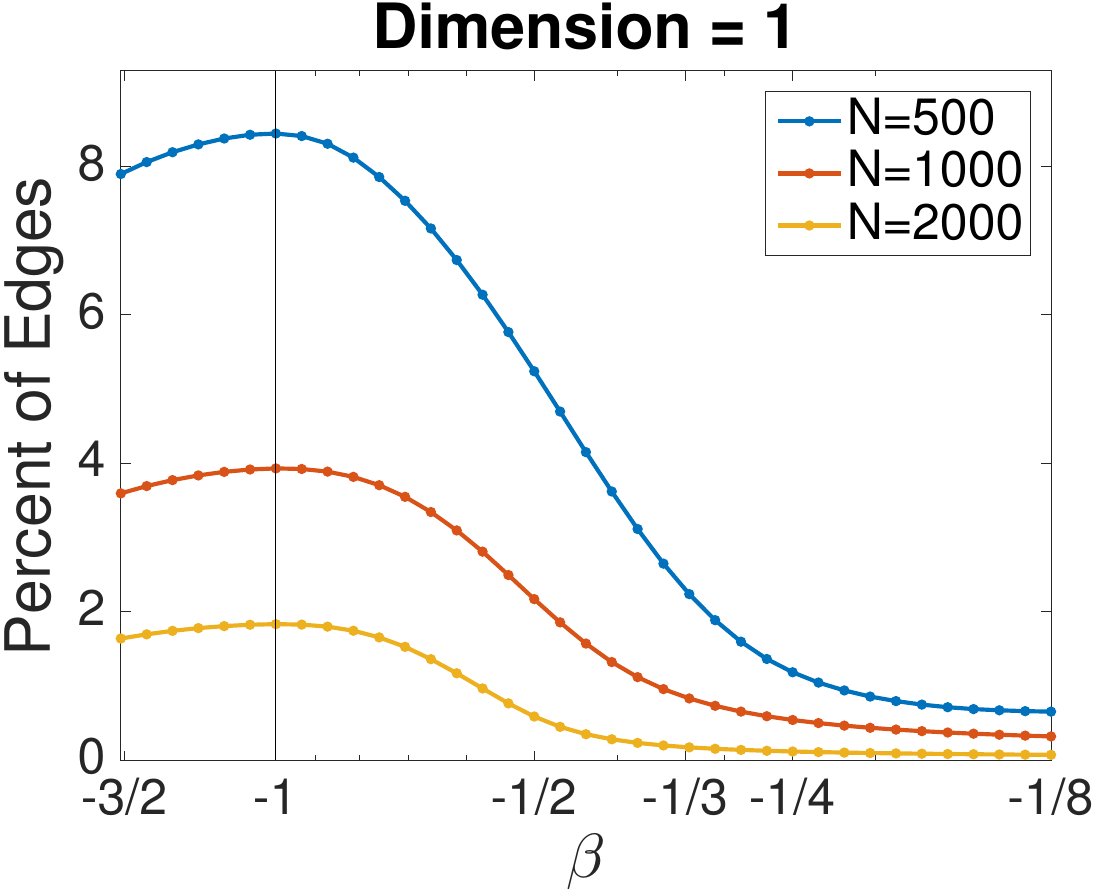}}
\subfigure[]{\includegraphics[width=.38\linewidth]{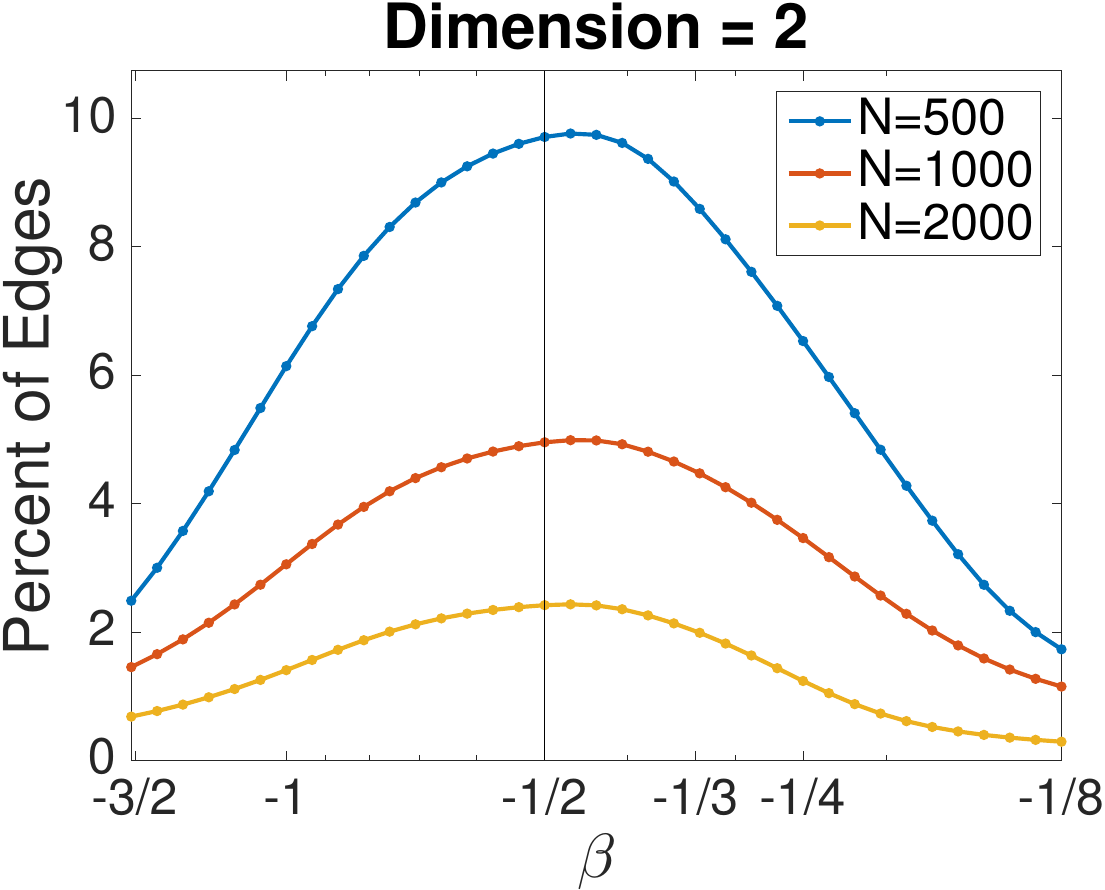}}\\
\subfigure[]{\includegraphics[width=.38\linewidth]{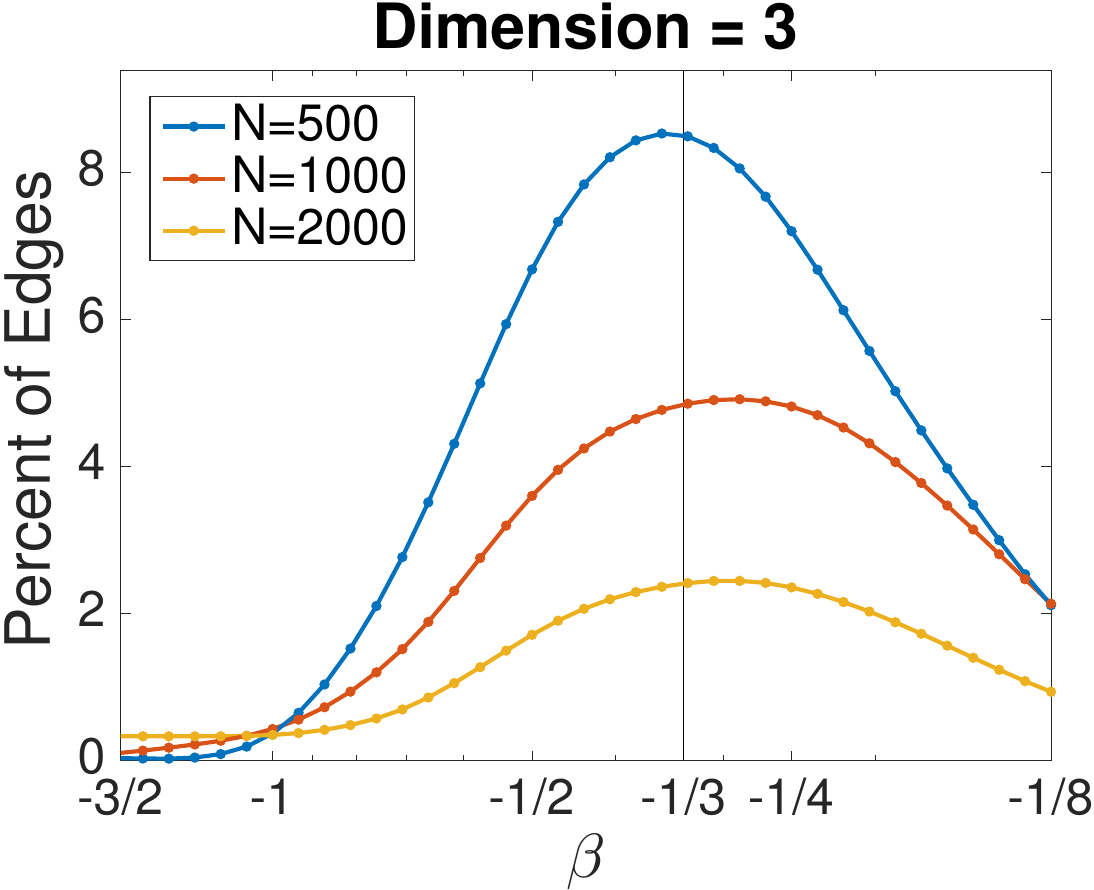}}
\subfigure[]{\includegraphics[width=.38\linewidth]{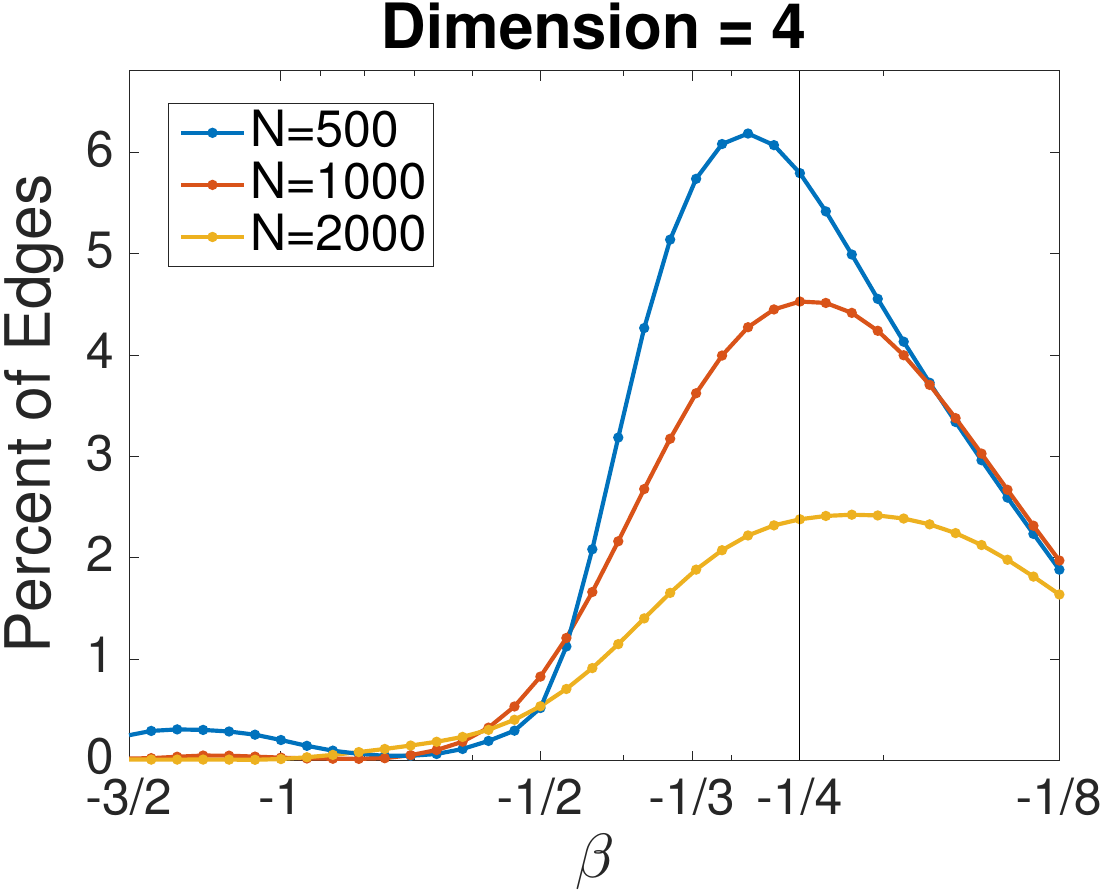}}
\end{center}
\caption{ \rm Persistence of the clustering of a cut-Gaussian distribution using the multi-scale graph construction with $\rho = q^{\beta}$, where the distribution is in dimension (a) $1$ (b) $2$ (c) $3$ (d) $4$.  Persistence is measured as a percentage of the total number of possible edges in the graph, $N(N-1)/2$ as a function of $\beta$.  Each data point represents an average over 500 data sets, each data set starts with sufficiently many points randomly sampled from an $m$-dimensional Gaussian so that after points in the gap region are rejected there are $N$ points remaining.  The correct homology (two connected components) is most persistent when $\beta$ is near $-1/m$ implying the optimal multi-scale graph construction is the CkNN construction where $\rho\propto q^{-1/m}$.  \label{clusteringFig1}}
\end{figure}

In the next two examples, we illustrate the theoretical result that the choice of $\beta = -1/m$ in the bandwidth function $\rho = q^{\beta}$ is optimal, where $m$ is the dimension of the data.

\begin{examp} \label{ex2} \rm We begin with the zero-order homology. We will demonstrate empirically that the choice $\beta = -1/m$ maximizes the persistence of the correct clustering, for $1 \leq m \leq 4$.  Consider data sampled from an $m$-dimensional Gaussian distribution with a gap of radial width $w=0.1^{1/m}$ centered at radius $w+3m/10$. (The dependence on the dimension $m$ is necessary to insure that there are two connected components for small data sets.)  The radial gap separates $\mathbb{R}^m$ into two connected components, the compact interior $m$-ball, and the non-compact shell extending to infinity with density decaying exponentially to zero.

Given a data set sampled from this density, we can construct a graph using the multi-scale graph construction which connects two points $x,y$ if $||x-y|| < \delta\sqrt{\rho(x)\rho(y)}$.  Since the true density is known, we consider the bandwidth functions $\rho = q^{\beta}$ for $\beta \in [-3/2,-1/8]$.  For each value of $\beta$ we used the fast clustering algorithm to identify the minimum and maximum numbers of edges which would identify the correct clusters.  We measured the persistence of the correct clustering as the difference between the minimum and maximum numbers of edges which identified the correct clusters divided by the total number of possible edges $N(N-1)/2$.  We then repeated this experiment for 500 random samples of the distribution and averaged the persistence of the correct clustering for each value of $\beta$. The results are shown in Fig.~\ref{clusteringFig1}.  Notice that for each dimension $m=1,...,4$ the persistence has a distinctive peak centered near $\beta = -1/m$ which indicates that the true clustering is the most persistent using the multi-scale graph construction that is equivalent to the CkNN.
\end{examp}

\begin{figure}
\begin{center}
\subfigure[]{\includegraphics[width=.4\linewidth]{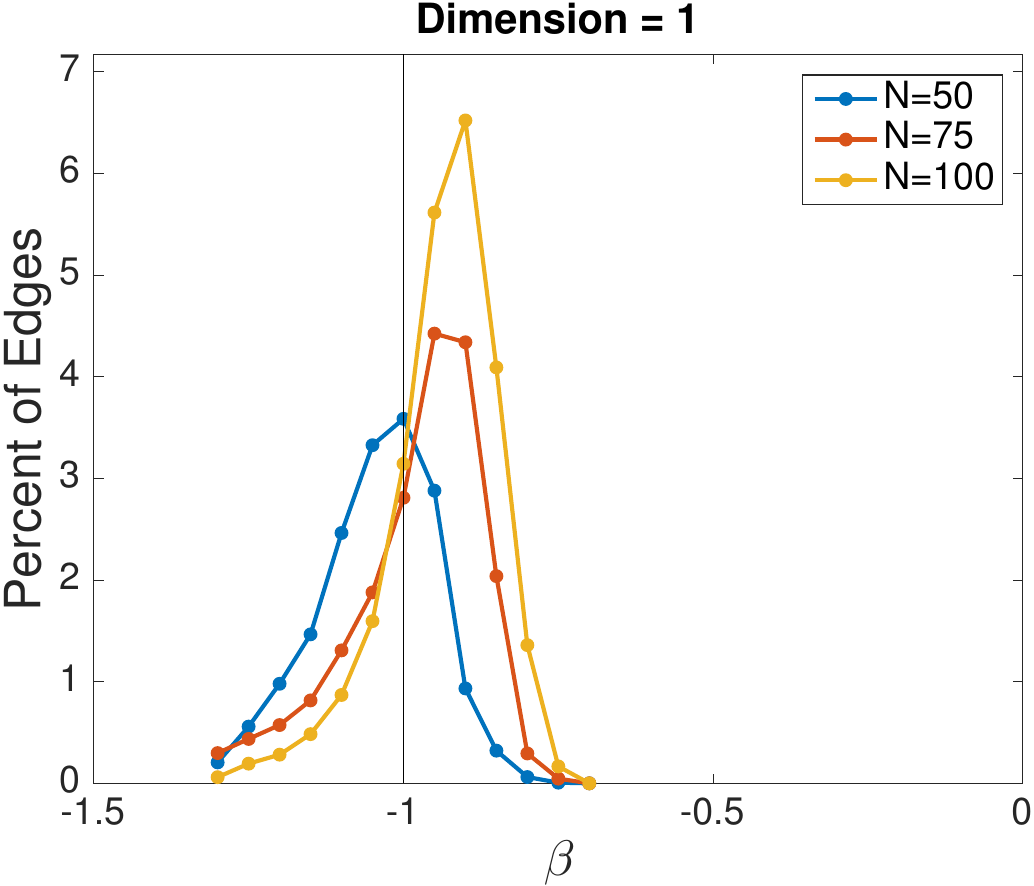}}\hspace{.1\linewidth}
\subfigure[]{\includegraphics[width=.4\linewidth]{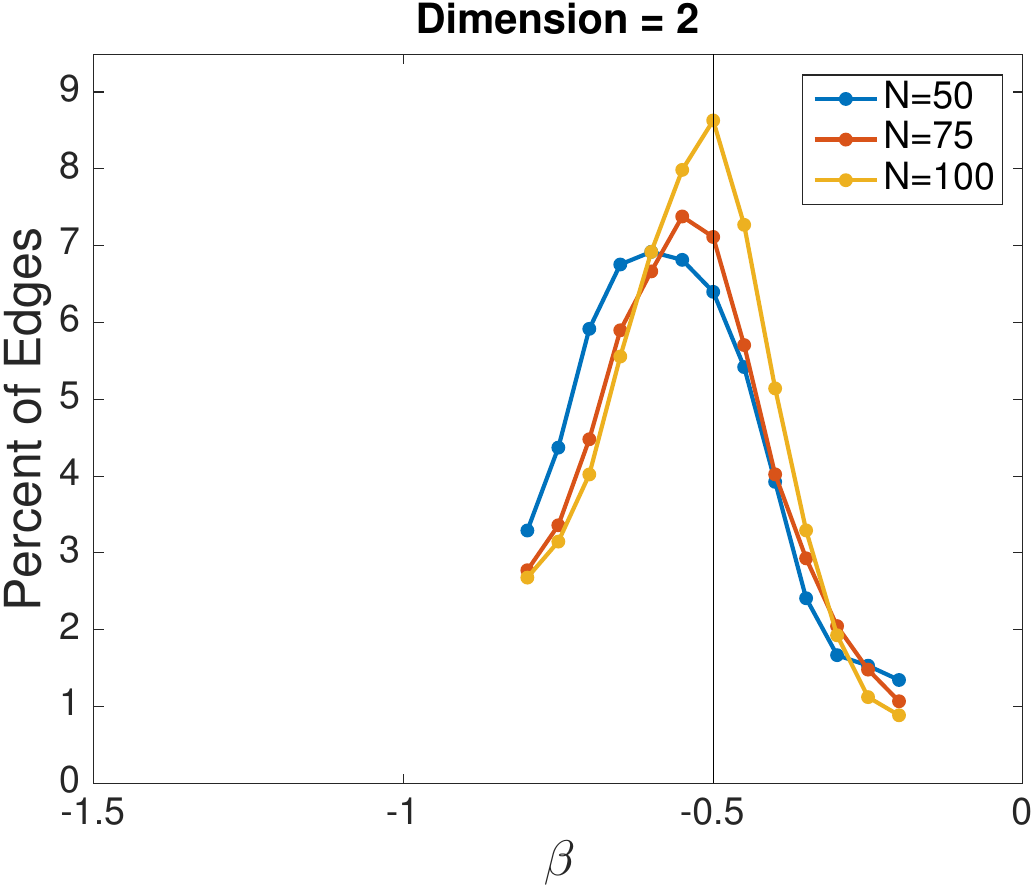}}
\end{center}
\caption{ \rm Persistence of the true homology using the multi-scale graph construction with $\rho = q^{\beta}$.  The underlying data sets are (a) a cut 1-dimensional Gaussian embedded in the plane with a loop introduced and (b) the same 2-dimensional cut Gaussian density as in Fig.~\ref{clusteringFig1}(b).  Persistence is measured as a percentage of the total number of possible edges in the graph, $N(N-1)/2$ as a function of $\beta$.  Each data point represents an average over 200 data sets.  The correct homology ($\beta_0=2$ and $\beta_1=1$ for both (a) and (b)) is most persistent when $\beta$ is near $-1/m$ implying the optimal multi-scale graph construction is the CkNN construction where $\rho\propto q^{-1/m}$.  \label{clusteringFig2H1}}
\end{figure}

\begin{examp} \label{ex3} \rm Next, we examine the discovery of the full homology for a 1-dimen\-sional and 2-dimensional example using Javaplex \cite{Javaplex}.  To obtain a one-dimensional example with two connected components we generated a set of points from a standard Gaussian on the $t$-axis with points with $0.4<t<0.8$ removed, and then mapped these points into the plane  via $t\mapsto (t^3-t,1/(t^2+1))^\top$. The embedding induces a loop, and so there is non-trivial 1-homology in the 1-dimensional example.  The correct homology for this example has Betti numbers $\beta_0=2$ and $\beta_1=1$, which is exactly the same as the true homology for the 2-dimensional cut Gaussian from Fig.~\ref{clusteringFig1}, which will be our 2-dimensional example.  In Fig.~\ref{clusteringFig2H1} we show the persistence of the correct homology in terms of the percentage of edges as a function of the parameter $\beta$ that defines the multi-scale graph construction.  As with the clustering example, the results clearly show that the correct homology is most persistent when $\beta$ is near $-1/m$ which corresponds to the CkNN graph construction.
\end{examp}

\subsection{Identifying patterns in images with homology}

In this section we consider the identification of periodic patterns or textures from image data. We take a topological approach to the problem, and attempt to classify the orbifold (the quotient of the plane by the group of symmetries) by its topological signature. Note that to achieve this, we will not need to learn the symmetry group, but will directly analyze the orbifold by processing the point cloud of small $s\times s$ pixel subimages of the complete image in $\mathbb{R}^{s^2}$ without regard to the original location of the subimages.

\begin{figure}
\begin{center}
\subfigure[]{\includegraphics[width=.95\linewidth,trim={4cm 0 3cm 0},clip]{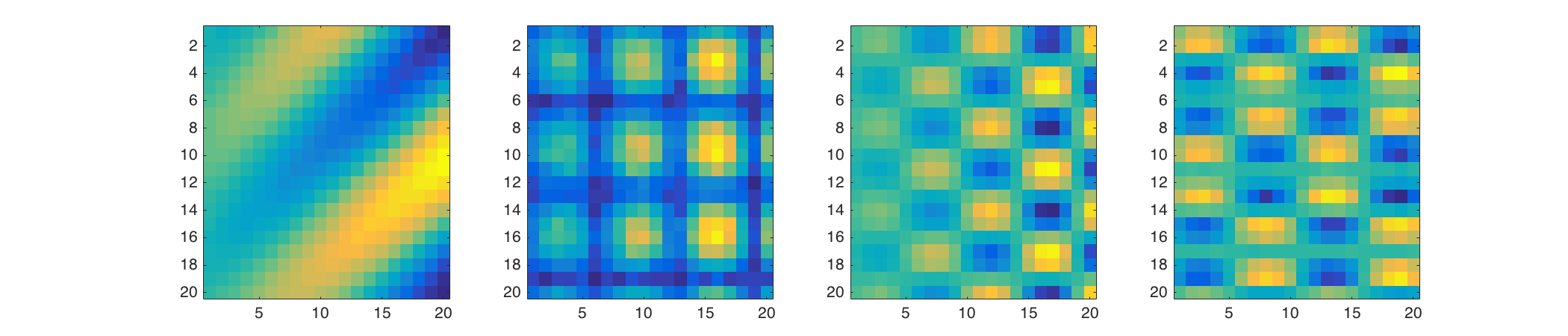}}\\
\subfigure[]{\includegraphics[width=.95\linewidth,trim={4cm 0 3cm 0},clip]{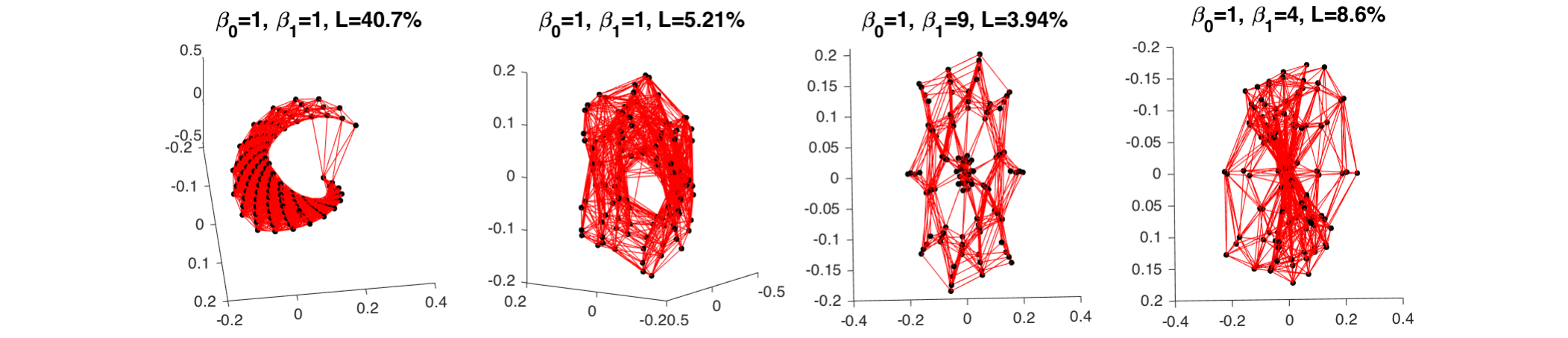}}\\
\subfigure[]{\includegraphics[width=.95\linewidth,trim={4cm 0 3cm 0},clip]{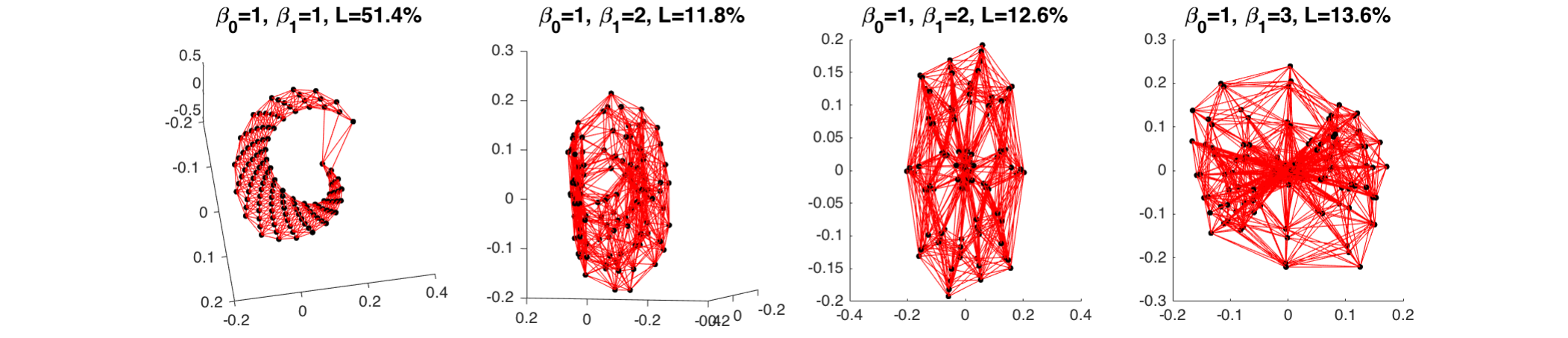}}
\end{center}
\caption{ \rm (a) Four simple patterns whose sub-image orbifolds exhibit nontrivial homology. The true values of $\beta_1$ for the four patterns are $1,2,2,$ and $3$ respectively.    (b) Using fixed $\epsilon$-ball graph construction and selecting $\epsilon$ from the region where the homology generators span the largest proportion $L$ of total edges without changing (the most persistent homology).  (c) Using the CkNN construction and selecting $\delta$ from the region where the homology generators span the largest $L$ without changing.  The homology was computed with JavaPlex \cite{Javaplex}. \label{patternFig}}
\end{figure}

\begin{examp} \rm
  In Fig.~\ref{patternFig}(a) we show four simple patterns that can be distinguished by homology.  To make the problem more difficult, the patterns are corrupted by a `brightness' gradient which makes identifying the correct homology difficult.  From left to right the patterns are: First, stripes have a single periodicity so that $\beta_1=1$; second, a pattern that is periodic in both the vertical and horizontal directions, implying $\beta_1=2$; third, a checkerboard pattern also has only two periodicities $\beta_1=2$, but they have different periods than the previous pattern; fourth, a hexagonal pattern has 3 periodicities so that $\beta_1=3$.  To see the three periodicities in the fourth pattern, notice that the pattern repeats when moving right two blocks, or down three blocks, or right one block and down two blocks; each of these periodicities yields a distinct homology class.

To identify the pattern in each image, we cut each full image into 9-by-9 sub-images, yielding 121 points in $\mathbb{R}^{81}$. In Fig.~\ref{patternFig}(b) we show the results of applying the fixed $\epsilon$-ball graph construction to each of the four sets of sub-images.  In order to choose $\epsilon$ we used JavaPlex \cite{Javaplex} to compute the persistent homology and then chose the region with the longest persistence (meaning the region of $\epsilon$ where the homology went the longest without changing).  In the title of each plot we show first two betti numbers for the graph constructed with this value of $\epsilon$, we also show the length of the persistence in terms of the percentage of edges for which the homology is unchanged.  In Fig.~\ref{patternFig}(c) we repeated this experiment using the CkNN construction, choosing $\delta$ from the region with the longest unchanging homology.  In this set of examples, the CkNN construction is more efficient, and finds the correct orbifold homology.

When the `brightness' gradient is removed, both the fixed $\epsilon$-ball and CkNN constructions identify the correct homology in the most persistent region.  However, the `brightness' gradient means that the patterns do not exactly meet (see the leftmost panels in Figs.~\ref{patternFig}(b,c)).  For the simple stripe pattern, the $\epsilon$-ball construction can still bridge the gap and identify the correct homology; however, for the more complex patterns, the $\epsilon$-ball construction finds many spurious homology classes which obscure the correct homology.
\end{examp}

\begin{figure}[h]
\begin{center}
\subfigure[]{\includegraphics[width=.35\linewidth]{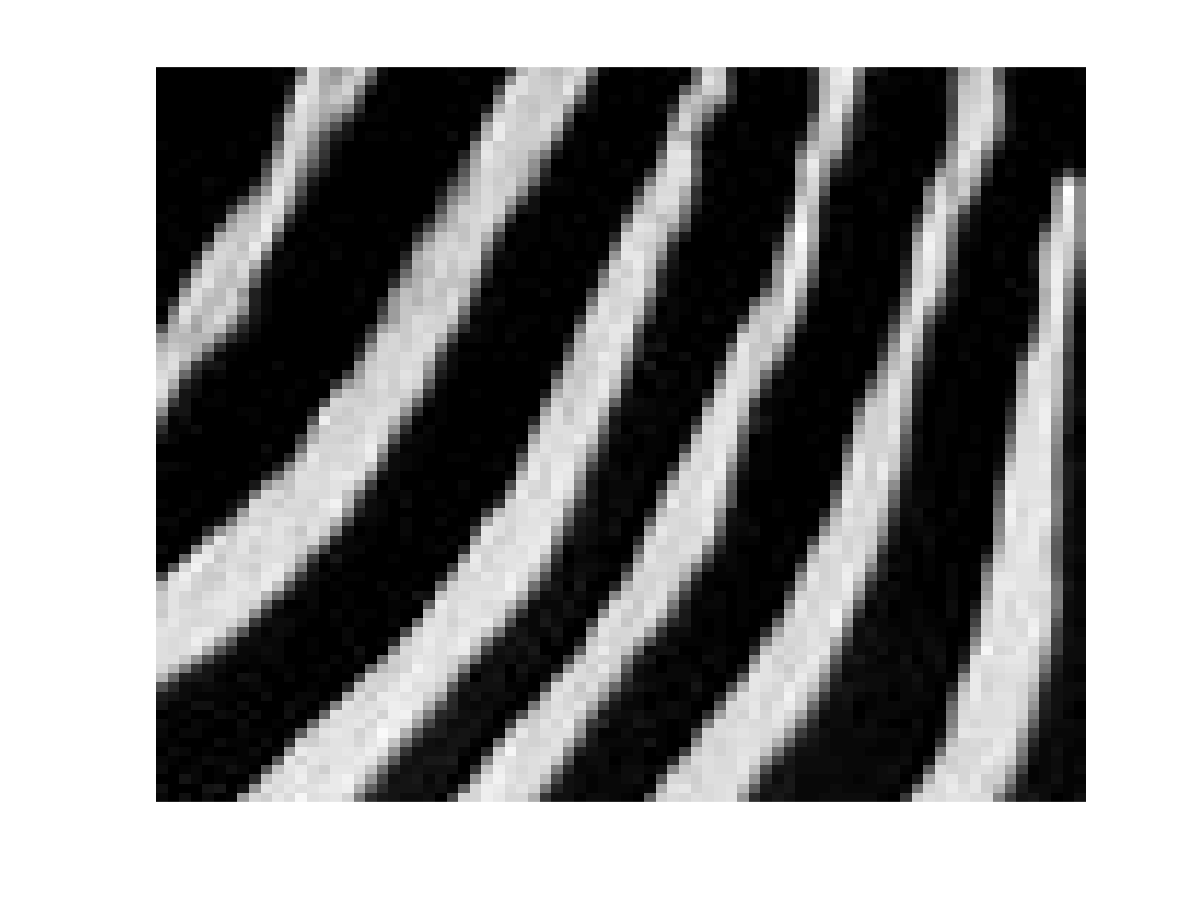}}
\subfigure[]{\includegraphics[width=.295\linewidth]{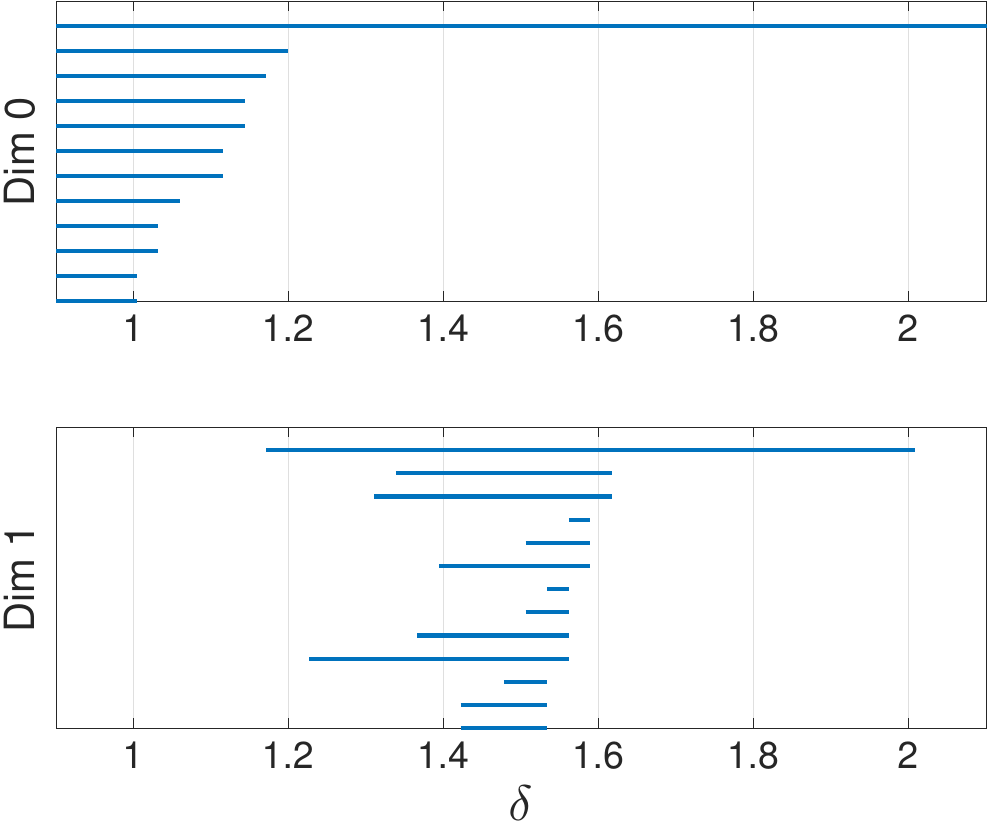}}
\subfigure[]{\includegraphics[width=.33\linewidth]{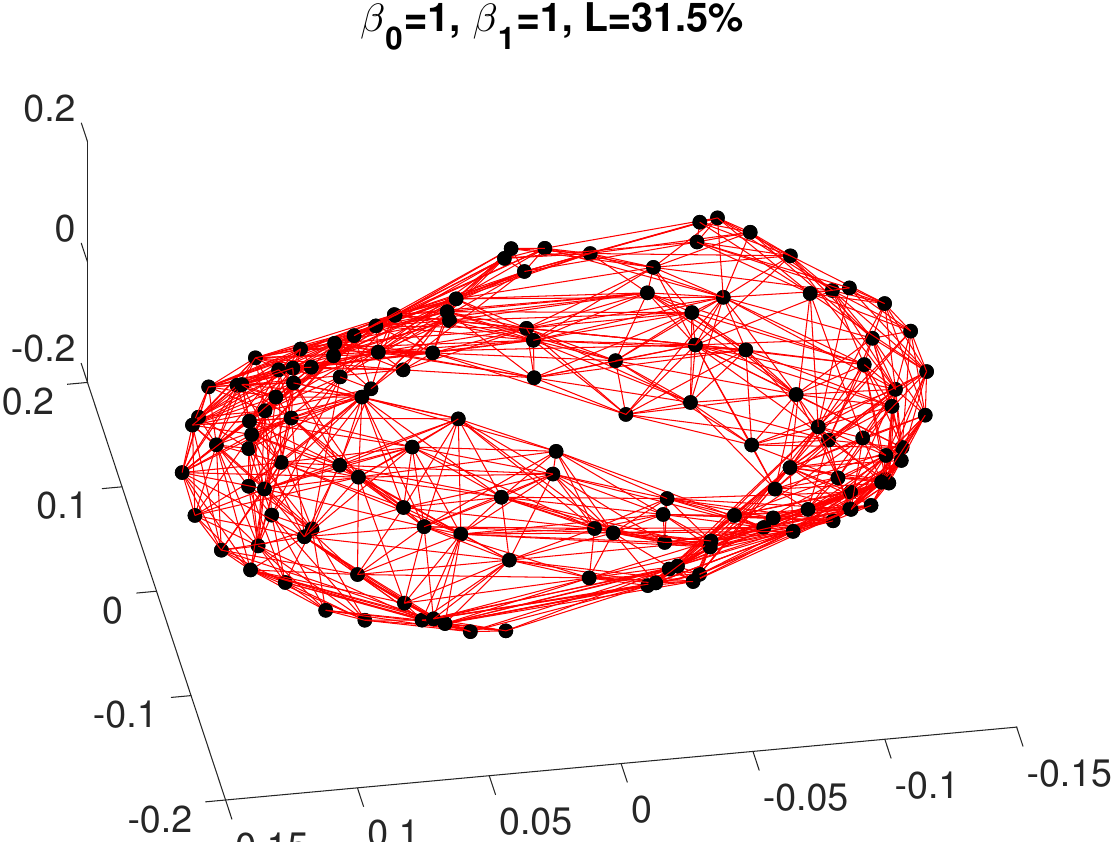}}\\
\subfigure[]{\includegraphics[width=.35\linewidth]{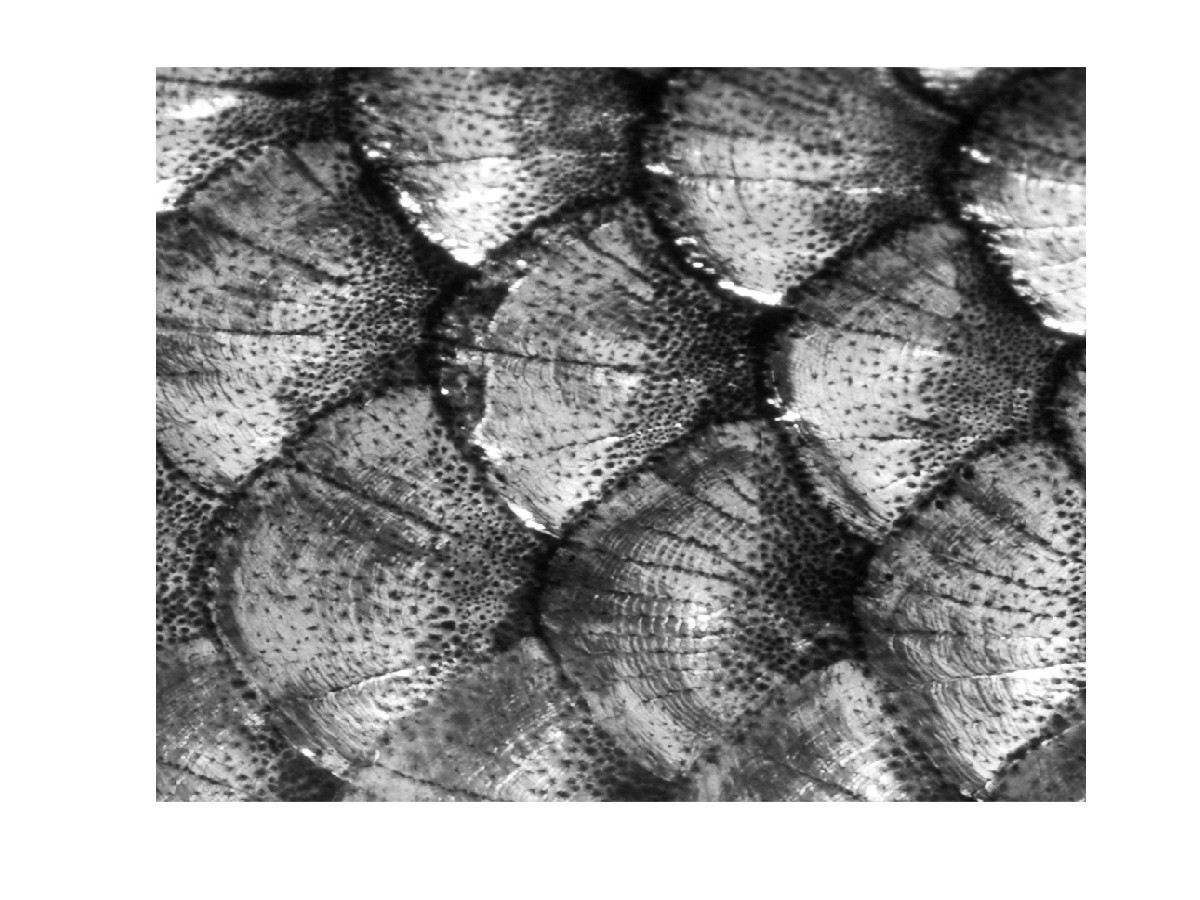}}
\subfigure[]{\includegraphics[width=.295\linewidth]{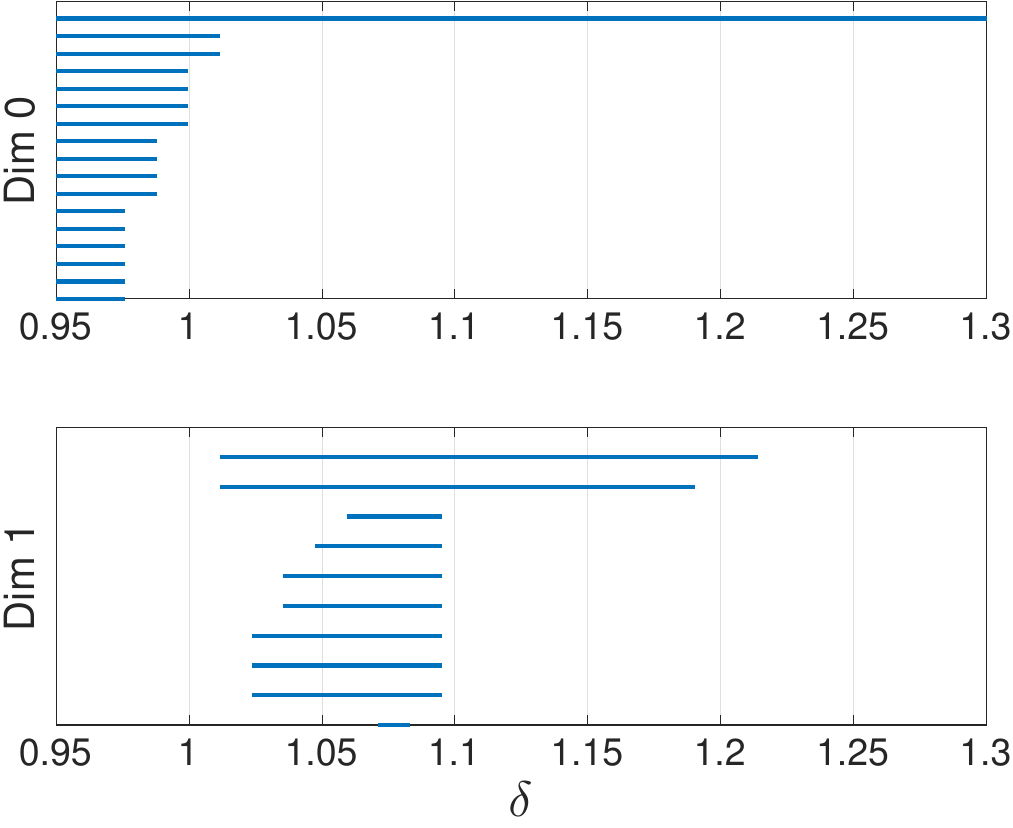}}
\subfigure[]{\includegraphics[width=.33\linewidth]{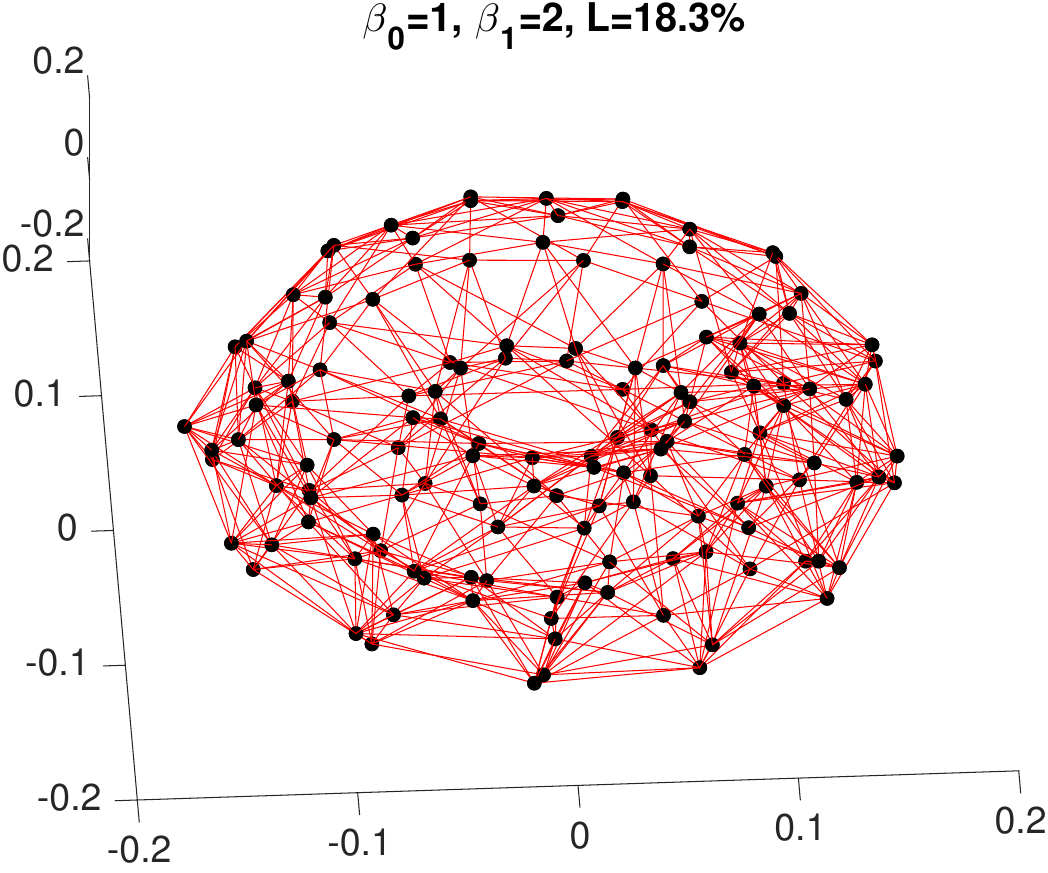}}\\
\end{center}
\vspace*{-8pt}\caption{ \rm (a) Image of zebra stripes \cite{Zebra}.  (b) Persistence diagram for the space of subimages. (c) The CkNN graph construction on the subimage space with $\delta$ chosen from the longest region where the homology is constant.  (d) Image of fish scales \cite{Fish}.  (e) Persistence diagram. (f) The CkNN graph construction on the PCA projection of the subimage space with $\delta$ chosen from the longest region where the homology is constant.  The homology was computed with JavaPlex \cite{Javaplex}. \label{patternFig2}}\vspace*{-8pt}
\end{figure}

\begin{examp}\rm

We applied the CkNN graph construction to identify patterns in real images of zebra stripes and fish scales in Figure \ref{patternFig2}.  The images shown in Fig.~\ref{patternFig2} were taken from larger images \cite{Zebra,Fish}.  In order to analyze the subimage spaces we first decimated the images to reduce the resolution, (by a factor of 2 for the stripes and factor of 25 for the scales) in each case to yield a $40\times 40$ pixel image.  We then formed the set of all 23-pixel by 23-pixel subimages shifting by two pixels in the vertical and horizontal directions to obtain 136 subimages, considered as points in $\mathbb{R}^{529}$.  We built the rescaled distance matrix $\frac{||x-y||}{\sqrt{||x-x_k||\,||y-y_k||}}$ using $k=5$.  In Fig.~\ref{patternFig2}(b,e) we show the persistent homology of the VR complex associated to the respective distance matrices in terms of the parameter $\delta$ of the CkNN.  Using this diagram, we chose the maximal interval of $\delta$ for which the homology was stable. In Fig.~\ref{patternFig2}(c,f) we show the CkNN graph for $\delta$ chosen from this region along with the correct Betti numbers.

\end{examp}

\section{Convergence of graph Laplacians}\label{background}

We approach the problem of consistent graph representations of manifolds by assuming we have data points which are sampled from a smooth probability distribution defined on the manifold.  We view this assumption as establishing a ``geometric prior'' for the problem.  Our main goal is  to approximate the Laplace-Beltrami operator on the manifold, independent of the sampling distribution, and using as little data as possible. This is a natural extension of ideas developed by \cite{belkin2003laplacian} and Coifman and collaborators \cite{diffusion,nadler2005diffusion,diffcoords,nadler2008diffusion,NadlerCluster}.

The proof of consistency for any graph construction has three parts, two statistical and one analytic: (1)  showing that the (discrete) graph Laplacian is a consistent statistical estimator of a (continuous) integral operator (either pointwise or spectrally), (2) showing that this estimator has finite variance, and (3) an asymptotic expansion of the integral operator that reveals the Laplace-Beltrami operator as the leading order term.  The theory of convergence of kernel weighted graph Laplacians to their continuous counterparts was initiated with \cite{belkin2003laplacian} which proved parts (1) and (3) for uniform sampling on compact manifolds, and part (2) was later completed in \cite{SingerEstimate}.  Parts (1) and (3) were then extended to non-uniform sampling in \cite{diffusion}, and part (2) was completed in \cite{BH14}.  The extension to noncompact manifolds was similarly divided.  First, \cite{heinthesis} provided the proof of parts (1) and (3), and introduced the necessary additional geometric assumptions which were required for the asymptotic analysis on non-compact manifolds.  However, \cite{BH14} showed that the pointwise errors on non-compact manifolds could be unbounded and so additional restrictions had to be imposed on the kernel used to construct the graph Laplacian.  In order to construct the desired operators on non-compact manifolds, \cite{BH14} showed that variable bandwidth kernels were required (part (1) for variable bandwidth kernels was previously achieved in \cite{Ting2010}).  In all of this previous work, parts (1) and (2) are always proven pointwise, despite the fact that most applications require spectral convergence.

The geometric prior assumes that the set of points that have positive sampling density is a smooth manifold $\mathcal{M} \equiv \{x \in \mathbb{R}^n \, : \, q(x)>0\}$ where $q$ is a smooth sampling density.  (This is a natural definition since regions of zero density will not be observed in data sets.)  Some weak assumptions on the manifold $\mathcal{M}$ are required.  The theory of \cite{belkin2003laplacian,diffusion} assumes that the manifold is compact, implying that the density $q$ must be bounded away from zero on $\mathcal{M}$.  The theory of \cite{heinthesis,Hein1,hein2} showed that this assumption could be relaxed, and together with the statistical analysis in \cite{BH14} allows a large class of noncompact manifolds with densities that are not bounded away from zero.  (Note that since $q$ is continuous, if $\mathcal{M}$ is compact the minimum value of $q$ cannot be zero since it is attained on $\mathcal{M}$ which is defined to have nonzero density.)  In this article we require $\mathcal{M}$ to have injectivity radius bounded below and curvature (intrinsic and extrinsic) bounded above; these technical assumptions hold for all compact manifolds and were introduced to allow application to noncompact manifolds in \cite{heinthesis}.

\begin{assum}\label{assumptions} The data points are sampled from a density $q:\mathbb{R}^n\to [0,\infty)$ such that $\mathcal{M} \equiv \{x \in \mathbb{R}^n \, : \, q(x)>0\}$ is a $C^3$ Riemannian manifold with finitely many connected components and $q$ restricted to $\mathcal{M}$ is a $C^3$ function.  The manifold $\mathcal{M}$ inherits a metric $g$ from the ambient space, and the conformal metric $q^{2/d}g$ has injectivity radius bounded away from zero and curvature and second fundamental form bounded above. (Note that $\mathcal{M}$ may be noncompact and may have a boundary.)
\end{assum}

There are several algorithms for estimating the Laplace-Beltrami operator, however the particularly powerful construction in \cite{BH14} is currently the only estimator which allows the sampling density $q$ to be arbitrarily close to zero.  Since we are interested in addressing the problem of non-uniform sampling, we will apply the result of \cite{BH14} for variable bandwidth kernels.  However, the method of \cite{BH14} used a special weighted graph Laplacian in order to approximate the Laplace-Beltrami operator.  The weighted graph Laplacian uses additional normalizations which were first introduced in the diffusion maps algorithm \cite{diffusion} in order to counteract the effect of the sampling density.  The goal of this paper is to use an unweighted graph Laplacian to approximate the Laplace-Beltrami operator, since this allows us to compute the topology of the graph using fast combinatorial algorithms.  In fact, the unweighted graph Laplacian will converge to the Laplace-Beltrami operator of the embedded manifold in the special case of uniform sampling.  The power of our new graph construction is that we can recover a Laplace-Beltrami operator for the manifold from an unweighted graph Laplacian, even when the sampling is not uniform.

Let $q(x)$ represent the sampling density of a data set $\{x_i\}_{i=1}^N \subset \mathcal{M}\subset \mathbb{R}^n$. We combine a global scaling parameter $\delta$ with a local scaling function $\rho(x)$ to define the combined bandwidth $\delta \rho(x)$.  Consider the symmetric variable bandwidth kernel
\begin{equation}\label{e1b}
W_\delta(x,y) = h\left(\frac{||x-y||^2}{\delta^2\rho(x)\rho(y)}\right)
\end{equation}
for any shape function $h:[0,\infty) \to [0,\infty)$ that has exponential decay.  For a function $f$ and any point $x\in \mathcal{M}$ we can form a Monte-Carlo estimate of the integral operator given by
\begin{equation}\label{mc} \mathbb{E}\left[ \frac{1}{N}\sum_{j=1}^N W_{\delta}(x,x_j)f(x_j) \right] = \int_{\mathcal{M}}W_{\delta}(x,y)f(y)q(y)\, dV(y). \end{equation}
From \cite{BH14} the expression in \eqref{mc} has the asymptotic expansion
\begin{eqnarray}\label{bh14eq} \delta^{-m}\int_{\mathcal{M}}&&\hspace*{-.3in}W_{\delta}(x,y)f(y)q(y)\, dV(y) \nonumber\\ &=& m_0 f q \rho^m + \delta^2 \frac{m_2 \rho^{m+2}}{2}\left[\omega f q \rho^{-2} + \mathcal{L}(fq) \right] + \mathcal{O}(\delta^4)  \end{eqnarray}
where
\[ \mathcal{L}h \equiv -\Delta h + (m+2)\nabla \log(\rho) \cdot \nabla h, \]
$\Delta$ is the positive definite Laplace-Beltrami operator, and $\nabla$ is the gradient operator, both with respect to the Riemannian metric that $\mathcal{M}$ inherits from the ambient space $\mathbb{R}^n$.  (Note that in \cite{BH14} the expansion \eqref{bh14eq} is written with respect to the negative definite Laplace Beltrami operator, whereas here we consider the positive definite version.)  In the right-hand side of \eqref{bh14eq}, all functions are evaluated at $x$. The function $\omega = \omega(x)$ depends on the shape function $h$ and the curvature of the manifold at $x$.

\subsection{Pointwise and integrated bias and variance}

The standard graph Laplacian construction starts with a symmetric affinity matrix $W$ whose $ij$ entry quantifies similarity between nodes $x_i$ and $x_j$. We assume in the following that $W=W_\delta$ from \eqref{e1b}, which includes the CkNN construction as a special case. Define the diagonal normalization matrix $D$ as the row sums of $W$, i.e. $D_{ii} = \sum_j W_{ij}$.   The unnormalized graph Laplacian matrix is then
\begin{equation} \label{defLun}
L_{\rm un} = D-W.
\end{equation}
We interpret this matrix as the discrete analog of an operator.  Given a function $f$, by forming a vector $\vec f_j = f(x_j)$ the matrix vector product is
\begin{equation}\label{LunDef} \left(L_{\rm un}\vec f \right)_i = f_i D_{ii} - \sum_{j=1}^N W_{ij}f_j = \sum_{j\neq i} W_{ij}(f_i - f_j) \end{equation}
so that $L_{\rm un}\vec f$ is also a vector that represents a function on the manifold.  Theorem \ref{pointwiseLun} below makes this connection rigorous by showing that
for an appropriate factor $c$ that depends on $N$ and $\delta$, the vector
 $c^{-1}(L_{\rm un}\vec f)_i$ is a consistent statistical estimator of the differential operator
\begin{equation}\label{LunOp} \mathcal{L}_{q,\rho}f \equiv \rho^{m+2} \left[ f\mathcal{L}q - \mathcal{L}(fq) \right] = q \rho^{m+2} \left[\Delta f - \nabla \log\left(q^2\rho^{m+2}\right) \cdot \nabla f \right] \end{equation}
where all functions are evaluated at the point $x_i$.  The last equality in \eqref{LunOp} follows from the definition of $\mathcal{L}$ and applying the product rules for positive definite Laplacian $\Delta(fg) = f\Delta g + g\Delta f - 2\nabla f \cdot \nabla g$ and the gradient $\nabla(fg) = f\nabla g+g\nabla f$.  Theorem \ref{pointwiseLun} shows that $L_{\rm un}$ is a \emph{pointwise} consistent statistical estimator of a differential operator.  In fact consistency was first shown in \cite{Ting2010} and the bias of this estimator was first computed in \cite{BH14}. Here we include the variance of this estimator as well.

\begin{thm}[Pointwise bias and variance of $L_{\rm un}$] \label{pointwiseLun} Let $\mathcal{M}$ be an $m$-dimensional Riemannian manifold embedded in $\mathbb{R}^n$, let $q:\mathcal{M}\to\mathbb{R}$ be a smooth density function, and define $L_{\rm un}$ as in \eqref{defLun}.  For $\{x_i\}_{i=1}^N$ independent samples of $q$, and $f\in C^3(\mathcal{M})$ we have
\begin{equation}\label{pwLun} \mathbb{E}\left[ c^{-1} \left(L_{\rm un}\vec f \right)_i \right] = \mathcal{L}_{q,\rho}f(x_i) + \mathcal{O}(\delta^2) \end{equation}
and
\begin{align}\label{Lunvar} \textup{var}\left[c^{-1} \left(L_{\rm un}\vec f \right)_i \right] = a \frac{\delta^{-m-2}}{N-1} \rho(x_i)^{m+2} q(x_i) ||\nabla f(x_i)||^2 + \mathcal{O}(N^{-1},\delta^2) \end{align}
where $c \equiv \frac{m_2}{2}(N-1) \delta^{m+2}$ and $a \equiv \frac{4 m_{2,2}}{m_2^2}$ are constants defined in the table below.
\end{thm}

Before proving Theorem \ref{pointwiseLun}, we comment on the constants such as $m_0, m_2$ and $m_{2,2}$, which depend on the shape of the kernel function.
  It was originally shown in \cite{diffusion} that the expansion \eqref{bh14eq} holds for any kernel with exponential decay at infinity.  A common kernel used in geometric applications is the Gaussian $h(||z||^2) = \exp(-||z||^2/2)$, due to its smoothness.  For topological applications, we are interested in building an unnormalized graph, whose construction corresponds to a kernel defined by the indicator function $\mathbbm{1}_{||z||^2<1}$.  The sharp cutoff function enforces the fact that each pair of points is either connected or not.  (The indicator kernel satisfies  \eqref{bh14eq} since it has compact support and thus has exponential decay.)

  In the table below we give formulas for all the constants which appear in the results, along with their values for the Gaussian and indicator functions (note that $B_m$ is the unit ball in $\mathbb{R}^m$).

\begin{center}
\begin{tabular}{|c|c|c|c|} \hline Constant & Formula & $h(x)=e^{-x/2}$ & $h(x) = \mathbbm{1}_{x<1}$ \\ \hline
$m_0$ & $\int_{\mathbb{R}^m} h(||z||^2)\, dz$ & $(2\pi)^{m/2}$ & $\textup{vol}(B_m)$ \\ \hline
$m_2$ & $\int_{\mathbb{R}^m} z_1^2 h(||z||^2)\, dz$ & $(2\pi)^{m/2}$ & $(m+2)^{-1}\textup{vol}(B_m)$ \\ \hline
$m_{2,2}$ & $\int_{\mathbb{R}^m} z_1^2 h(||z||^2)^2\, dz$ & $2^{-1}\pi^{m/2}$ & $(m+2)^{-1}\textup{vol}(B_m)$ \\ \hline
$a$ & $4 m_{2,2}(m_2)^{-2}$ & $2^{1-m}\pi^{-m/2}$ & $4(m+2)\textup{vol}(B_m)^{-1}$ \\ \hline
\end{tabular}
\end{center}

It turns out that the variance of the statistical estimators considered below are all proportional to the constant $a$.  As a function of the intrinsic dimension $m$, the constant $a$ decreases exponentially for the Gaussian kernel.  On the other hand, since the volume of a unit ball decays like $m^{-m/2-1/2}$ for large $m$, the constant $a$ increases exponentially for the indicator kernel.

\vspace*{4pt}\noindent{\it Proof of Theorem \ref{pointwiseLun}}.
Notice that the term $i=j$ is zero and can be left out of the summation, this allows us to consider the expectation of $L_{\rm un}\vec f$ only over the terms $i\neq j$ which are identically distributed.  From \eqref{bh14eq} the $i$-th entry of the vector $\left(L_{\rm un}\vec f\right)_i$ has expected value
\[ \mathbb{E}\left[ \left(L_{\rm un}\vec f \right)_i \right] = \frac{m_2}{2}(N-1)(\delta \rho)^{m+2} \left( f\mathcal{L}q - \mathcal{L}(fq) \right) + \mathcal{O}(N\delta^{m+4}) \]
and by \eqref{LunOp} we have $\mathcal{L}_{q,\rho} \equiv \rho^{m+2}(f\mathcal{L}q - \mathcal{L}(fq))$.  Dividing by the constant $c$ yields \eqref{pwLun}, which shows that $L_{\rm un}$ is a pointwise consistent estimator with bias of order $\delta^2$.

We can also compute the variance of this estimator defined as
\begin{align} \textup{var}\left(c^{-1} \left(L_{\rm un}\vec f \right)_i \right) &\equiv \mathbb{E}\left[ \left( c^{-1} \left(L_{\rm un}\vec f \right)_i - \mathbb{E}\left[ c^{-1}\left(L_{\rm un}\vec f \right)_i \right]  \right)^2 \right] \nonumber\\&= c^{-2}\mathbb{E}\left[ \left( \sum_{j\neq i, j=1}^N f_i W_{ij} - W_{ij}f_j - \mathbb{E}\left[ f_i W_{ij} - W_{ij}f_j \right] \right)^2\right]. \end{align}
Since $x_j$ are independent we have
\begin{align} &\textup{var}\left(c^{-1} \left(L_{\rm un}\vec f \right)_i \right)= c^{-2}(N-1) \mathbb{E}\left[ \left(f_i W_{ij} - W_{ij}f_j - \mathbb{E}\left[ f_i W_{ij} - W_{ij}f_j \right] \right)^2\right] \nonumber \\
&= c^{-2}(N-1) \mathbb{E}\left[ \left(f_i W_{ij} - W_{ij}f_j \right)^2 \right] - c^{-2}(N-1)\mathbb{E}\left[ f_i W_{ij} - W_{ij}f_j \right]^2 \nonumber \end{align}
Notice that for the last term we have
\begin{align}& c^{-2}(N-1)\mathbb{E}\left[ f_i W_{ij} - W_{ij}f_j \right]^2= \frac{1}{N-1} \mathbb{E}\left[ \frac{N-1}{c} \left(f_i W_{ij} - W_{ij}f_j\right) \right]^2 \\&= \frac{1}{N-1} (\mathcal{L}_{q,\rho}f(x_i))^2 + \mathcal{O}(\delta^2) \nonumber \end{align}
this term will be higher order so we summarize it as $\mathcal{O}(N^{-1},\delta^2)$.  Computing the remaining term we find the variance to be
\begin{align}\label{th1eq1} \textup{var}\left(c^{-1} \left(L_{\rm un}\vec f \right)_i \right) &= c^{-2}(N-1) \mathbb{E}\left[ \left(f_i W_{ij} - W_{ij}f_j \right)^2\right] + \mathcal{O}(N^{-1},\delta^2) \nonumber \\
&= c^{-2}(N-1) \mathbb{E}\left[ f_i^2 W_{ij}^2 - 2 f_i W_{ij}^2 f_j + W_{ij}^2f_j^2 \right] + \mathcal{O}(N^{-1},\delta^2)  \end{align}
Since $W_{ij}^2$ is also a local kernel with moments $m_{0,2}$ and $m_{2,2}$ the above asymptotic expansions apply and we find
\begin{align}  &\mathbb{E}\left[ f_i^2 W_{ij}^2 - 2 f_i W_{ij}^2 f_j + W_{ij}^2f_j^2 \right] \nonumber \\ &= \frac{m_{2,2}}{2}(\delta \rho)^{m+2} \left( f^2 \mathcal{L}q - 2 f\mathcal{L}(fq) + \mathcal{L}(f^2q) \right) + \mathcal{O}(\delta^{m+4}) \nonumber \\
&= m_{2,2}(\delta \rho)^{m+2} q ||\nabla f||^2 + \mathcal{O}(\delta^{m+4}) \nonumber \end{align}
where the last equality follows from the applying the product rule.  Substituting the previous expression into \eqref{th1eq1} verifies \eqref{Lunvar}.
\qed

Theorem \ref{pointwiseLun} gives a complete description of the pointwise convergence of the discrete operator $L_{\rm un}$.
While the expansion \eqref{pwLun} shows that the matrix $c^{-1}L_{\rm un}$ approximates the operator $\mathcal{L}_{q,\rho}$, it does not tell us how the eigenvectors of $c^{-1}L_{\rm un}$ approximate the eigenfunctions of $\mathcal{L}_{q,\rho}$.  That relationship is the subject of Theorem \ref{spectralLun} below.

Since $c^{-1}L_{\rm un}$ is a symmetric matrix, we can interpret the eigenvectors $\vec f$ as the sequential orthogonal minimizers of the functional
\begin{equation}\label{functional} \Lambda(f) = \frac{\vec f \,^\top c^{-1}L_{\rm un} \vec f}{\vec f \,^\top \vec f}. \end{equation}
By dividing the numerator and denominator by $N$, we can interpret the inner product $N^{-1} \vec f \,^\top \vec f$ as an integral over the manifold since
\[ \mathbb{E}\left[N^{-1} \vec f \,^\top \vec f\right] = \mathbb{E}\left[ \frac{1}{N} \sum_{i=1}^N f(x_i)^2 \right] = \int_{\mathcal{M}} f(x)^2 q(x)\, dV(x) = \left<f^2,q\right>_{dV}. \]
It is easy to see that the above estimator has variance $N^{-1}\left[\left<f^4,q\right>_{dV} - \left<f^2,q\right>^2_{dV}\right]$.
Similarly, in Theorem \ref{spectralLun} we will interpret $\frac{1}{N}\vec f^\top c^{-1}L_{\rm un} \vec f$ as approximating an integral over the manifold.

\begin{thm}[Integrated bias and variance of $L_{\rm un}$]\label{spectralLun} Let $\mathcal{M}$ be an $m$-dimensional Riemannian manifold embedded in $\mathbb{R}^n$ and let $q:\mathcal{M}\to\mathbb{R}$ be a smooth density function.  For $\{x_i\}_{i=1}^N$ independent samples of $q$, and $f\in C^3(\mathcal{M})$ we have
\begin{equation}\label{pwLunInt} \mathbb{E}\left[ (cN)^{-1} \vec f \,^\top L_{\rm un}\vec f \right] = \left<f,q\mathcal{L}_{q,\rho}f \right>_{dV} + \mathcal{O}(\delta^2) \end{equation}
and
\begin{align}\label{LunvarInt} &\textup{var}\left((cN)^{-1} \vec f \,^\top L_{\rm un}\vec f \right)  \nonumber \\=& a \frac{ \delta^{-m-2}}{N(N-1)} \left<f^2,q\mathcal{L}_{q,\rho}(f^2) \right>_{dV} + \frac{4}{N} \left<f^2 ,q (\mathcal{L}_{q,\rho}f)^2 \right>_{dV} + \mathcal{O}(\delta^2).  \end{align}
So assuming the inner products are finite (for example if $\mathcal{M}$ is a compact manifold) we have $ \textup{var}\left((cN)^{-1} \vec f \,^\top L_{\rm un}\vec f \right) = \mathcal{O}(N^{-2}\delta^{-m-2},N^{-1},\delta^2)$.
\end{thm}
\begin{proof}
By definition we have
\[ \mathbb{E}\left[ \frac{c^{-1}}{N}\vec f^\top L_{\rm un} \vec f \right] =  \mathbb{E}\left[c^{-1} \sum_{i \neq j} \left( L_{\rm un}\right)_{ij} \vec f_i \vec f_j \right] =
 \mathbb{E}\left[f(x_i) c^{-1}\left( L_{\rm un}\vec f \right)_i \right]  \]
where the term $i=j$ is zero and so the sum is over the $N(N-1)$ terms where $i\neq j$.  Since $x_i$ are sampled according to the density $q$ we have
\begin{align} \mathbb{E}\left[ \frac{c^{-1}}{N}\vec f^\top L_{\rm un} \vec f \right] &= \int f q^2 \rho^{m+2} \left( \Delta f + \nabla \log(q^2\rho^{m+2}) \cdot \nabla f \right) dV + \mathcal{O}(\delta^2) \nonumber \\ &= \left<f,q^2\rho^{m+2}\left( \Delta f + \nabla \log(q^2\rho^{m+2}) \cdot \nabla f \right) \right>_{L^2(\mathcal{M},dV)}  + \mathcal{O}(\delta^2) \nonumber \\
&= \left<f,q \mathcal{L}_{q,\rho} f\right>_{dV}  + \mathcal{O}(\delta^2). \end{align}
We can now compute the variance of the estimator $ \frac{c^{-1}}{N}\vec f^\top L_{\rm un} \vec f$ which estimates $\left<f,q \mathcal{L}_{q,\rho} f\right>_{dV}$.  To find the variance we need to compute
\begin{align}\label{ES} \mathbb{E}\left[ \left( \frac{c^{-1}}{N}\vec f^\top L_{\rm un} \vec f \right)^2 \right]  = \frac{c^{-2}}{N^2}\mathbb{E}\left[ \left( \sum_{i \neq j} \left(L_{\rm un}\right)_{ij}f_i f_j \right)\left( \sum_{k \neq l} \left(L_{\rm un}\right)_{kl}f_k f_l \right) \right]
\end{align}
Notice that when $i,j,k,l$ are all distinct, by independence we can rewrite these terms of \eqref{ES} as
\begin{align}  \frac{a_1}{N^2}\mathbb{E}\left[ \sum_{i\neq j} c^{-1}\left(L_{\rm un}\right)_{ij}f_i f_j \right] \mathbb{E}\left[ \sum_{k\neq l}c^{-1}\left(L_{\rm un}\right)_{kl}f_k f_l  \right] = a_1\left<f,q\mathcal{L}_{q,\rho}f\right>^2_{dV}. \end{align}
The constant $a_1 \equiv \frac{(N-2)(N-3)}{N(N-1)}$ accounts for the fact that of the $N^2(N-1)^2$ total terms in \eqref{ES}, only $N(N-1)(N-2)(N-3)$ terms have distinct indices.
Since $i\neq j$ and $k\neq l$, we next consider the terms where either $i\in\{k,l\}$ or $j\in \{k,l\}$ but not both.  Using the symmetry of $L_{\rm un}$, by changing index names we can rewrite all four combinations as $i=k$ so that these terms of \eqref{ES} can be written as
\begin{align}  4 \mathbb{E}_{x_i} &\hspace{-1pt} \left[ \frac{1}{N^2} \sum_i \left(  \mathbb{E}_{x_j}\left[ \sum_{j} c^{-1}\left(L_{\rm un}\right)_{ij}f_i f_j \right] \mathbb{E}_{x_l}\left[ \sum_{l}c^{-1}\left(L_{\rm un}\right)_{il}f_i f_l  \right] \right) \right] \nonumber\\&= \frac{4}{N^2}\mathbb{E}_{x_i}\left[\sum_i f(x_i)^2 (\mathcal{L}_{q,\rho}f(x_i))^2 \right]  \nonumber \\
&= \frac{4}{N} \left<f^2,q(\mathcal{L}_{q,\rho}f)^2 \right>_{dV} \end{align}
Finally, we consider the terms where $i\in\{k,l\}$ and $j\in \{k,l\}$.  By symmetry we rewrite the two possibilities as $i=k$ and $j=l$ and these terms become,
\begin{align}  \frac{2c^{-2}}{N^2}\mathbb{E}\left[ \left( \sum_{i \neq j} \left(L_{\rm un}\right)_{ij}f_i f_j \left(L_{\rm un}\right)_{ij}f_i f_j \right) \right] = \frac{2 a_2 c^{-1}}{N} \left<f^2,q \mathcal{L}_{q,\rho}(f^2) \right>_{dV} \end{align}
where the constant $a_2 \equiv \frac{m_{2,2}}{m_2}$ is the ratio between the second moment $m_2$ of the kernel $W$ and the second moment $m_{2,2}$ of the squared kernel $W^2$.

We can now compute the variance of the estimator
\begin{align} \textup{var}\left(\frac{c^{-1}}{N}\vec f^\top L_{\rm un} \vec f \right) &= \mathbb{E}\left[ \left( \frac{c^{-1}}{N}\vec f^\top L_{\rm un} \vec f \right)^2 \right] - \mathbb{E}\left[ \frac{c^{-1}}{N}\vec f^\top L_{\rm un} \vec f  \right]^2 \nonumber\\& = \mathbb{E}\left[ \left( \frac{c^{-1}}{N}\vec f^\top L_{\rm un} \vec f \right)^2 \right] - \left<f,q\mathcal{L}_{q,\rho}f\right>^2_{dV} + \mathcal{O}(\delta^2) \nonumber \\
&= \frac{-4N+6}{N(N-1)} \left<f,q\mathcal{L}_{q,\rho}f\right>^2_{dV} + \frac{4}{N} \left<f^2 ,q (\mathcal{L}_{q,\rho}f)^2 \right>_{dV}  \nonumber\\& +4 a_2 m_2^{-1}\frac{ \delta^{-m-2}}{N(N-1)} \left<f^2, q\mathcal{L}_{q,\rho}(f^2) \right>_{dV} + \mathcal{O}(\delta^2). \nonumber
\end{align}
In particular this says that
\[ \textup{var}\left(\frac{c^{-1}}{N}\vec f \,^\top L_{\rm un} \vec f \right) = \mathcal{O}\left(N^{-2} \delta^{-m-2},N^{-1},\delta^2 \right) \]
assuming that all the inner products are finite.
\end{proof}

We should note that determining whether the inner product is finite is nontrivial when the manifold in question is unbounded and when the sampling density $q$ is not bounded away from zero.  We will return to this issue below.  First, we compute the bias and variance of the spectral estimates.  For wider applicability, we consider the generalized eigenvalue problem, $c^{-1}L_{\rm un} \vec f = \lambda M \vec f$ for any diagonal matrix $M_{ii} =\mu(x_i)$ which corresponds to the functional
\[ \Lambda(f) = \frac{\vec f \,^\top c^{-1}L_{\rm un} \vec f}{\vec f \,^\top M \vec f}. \]
Notice that $N^{-1}\mathbb{E}[\vec f \,^\top M \vec f \, ] = \left<f^2,\mu q \right>_{dV}$ and this estimator has variance
\[ \textup{var}\left(N^{-1}\vec f \,^\top M \vec f \right) = N^{-1}\left(\left<f^4,\mu^2 q\right>_{dV} - \left<f^2, \mu q\right>_{dV}^2 \right) = \mathcal{O}(N^{-1}).\]
A particular example which draws significant interest is the so-called `normalized graph Laplacian' where $M=D$,  implying that $\mu(x_i) = q(x_i)\rho(x_i)^m + \mathcal{O}(\delta^2)$.

\begin{thm}[Spectral bias and variance]\label{spectral} Under the same assumptions as Theorem \ref{spectralLun} we have
\begin{equation}\label{LunSpecBias} \mathbb{E}[\Lambda(f)] = \mathbb{E}\left[ \frac{ \vec f \,^\top c^{-1} L_{\rm un}\vec f}{\vec f\,^\top M \vec f} \right] = \frac{\left<f,q\mathcal{L}_{q,\rho}f \right>_{dV}}{\left<f^2,\mu q\right>_{dV}} + \mathcal{O}(\delta^2,N^{-1}) \end{equation}
and
\begin{align}\label{LunSpecVar} \textup{var}\left(\Lambda(f) \right) = a \frac{ \delta^{-m-2}}{N(N-1)} \frac{\left<f^2,q\mathcal{L}_{q,\rho}(f^2) \right>_{dV}}{\left<f^2,\mu q\right>_{dV}^2} + \frac{4}{N} \frac{\left<f^2 ,q (\mathcal{L}_{q,\rho}f)^2 \right>_{dV}}{\left<f^2,\mu q\right>_{dV}^2} + \mathcal{O}(\delta^2,N^{-1}).  \end{align}
So assuming the inner products are finite (for example if $\mathcal{M}$ is a compact manifold) we have $ \textup{var}\left((cN)^{-1} \vec f \,^\top L_{\rm un}\vec f \right) = \mathcal{O}(N^{-2}\delta^{-m-2},N^{-1},\delta^2)$.
\end{thm}
\begin{proof} We consider $\Lambda(f)$ to be a ratio estimator of the form $\Lambda(f) = \frac{a}{b}$ where $a = N^{-1}\vec f \,^\top c^{-1} L_{\rm un}\vec f$ and $b = N^{-1} \vec f\,^\top M \vec f$.  The correlation of $a$ and $b$ is given by
\[ \mathbb{E}[(a-\overline a)(b-\overline b)] = \frac{m_2 \delta^{m+2}}{2 N^2(N-1)} \sum_{i \neq j,k} f(x_i) ( L_{\rm un})_{ij} f(x_j) f(x_k)^2 \mu(x_k) - \overline a \overline b = \mathcal{O}(N^{-1}) \]
since the sum of the terms with $k=i$ or $k=j$ is clearly order $N^{-1}$, and when both $k\neq i$ and $k\neq j$ the expectation is equal to $\overline a \overline b$ by independence.  Since the variance of $b$ and the correlation are both order $N^{-1}$ by the standard ratio estimates, we have
\[ \mathbb{E}\left[\frac{a}{b}\right] = \frac{\overline a}{\overline b} + \mathcal{O}(N^{-1})  \]
and
\[ \textup{var}\left(\frac{a}{b}\right) = \frac{\textup{var}(a)}{\overline b^2} + \mathcal{O}(N^{-1}). \]
Combined with Theorem \ref{spectralLun}, these equations yield the desired result.
\end{proof}

Comparing Theorems \ref{pointwiseLun} and \ref{spectral} we find a surprising result, namely that the optimal $\delta$ for spectral approximation of the operator $\mathcal{L}_{q,\rho}$ is different from the optimal $\delta$ for pointwise approximation.  To our knowledge this has not been noted before in the literature.  We find the optimal choice of $\delta$ by balancing the squared bias with the variance of the respective estimators.  For the optimal pointwise approximation we need $\delta^4 = N^{-1}\delta^{-m-2}$ so the optimal choice is $\delta \propto N^{-1/(m+6)}$ and the combined error of the estimator is then $\mathcal{O}(N^{-2/(m+6)})$.  In contrast, for optimal spectral approximation we need $\delta^4 = N^{-2}\delta^{-m-2}$ so the optimal choice is
\[ \delta \propto N^{-2/(m+6)} \] and the combined error (bias and standard deviation) is
\[ \mathcal{O}(N^{-4/(m+6)}). \]
The one exception to this rule is the case $m=1$, where the second term in the spectral variance dominates, so that the optimal choice is  $\delta \propto N^{-1/3}$ and the combined error is order $N^{-1}$.

Since graph Laplacians are often used spectrally (for example to find low-dimen\-sional coordinates with the diffusion map embedding \cite{belkin2003laplacian,diffusion}, spectral clustering \cite{NJW,von2008consistency}, applications to time series analysis \cite{GiannakisPNAS,DMDC}, and the spectral exterior calculus \cite{berry2018spectral}) this implies that the choice of bandwidth $\delta$ should be significantly smaller in these applications than suggested in the literature \cite{SingerEstimate,BH14}.  We demonstrate this for the cutoff kernel on a circle in the example below.
{
\begin{figure}
\begin{center}
\subfigure[]{\includegraphics[width=.45\linewidth]{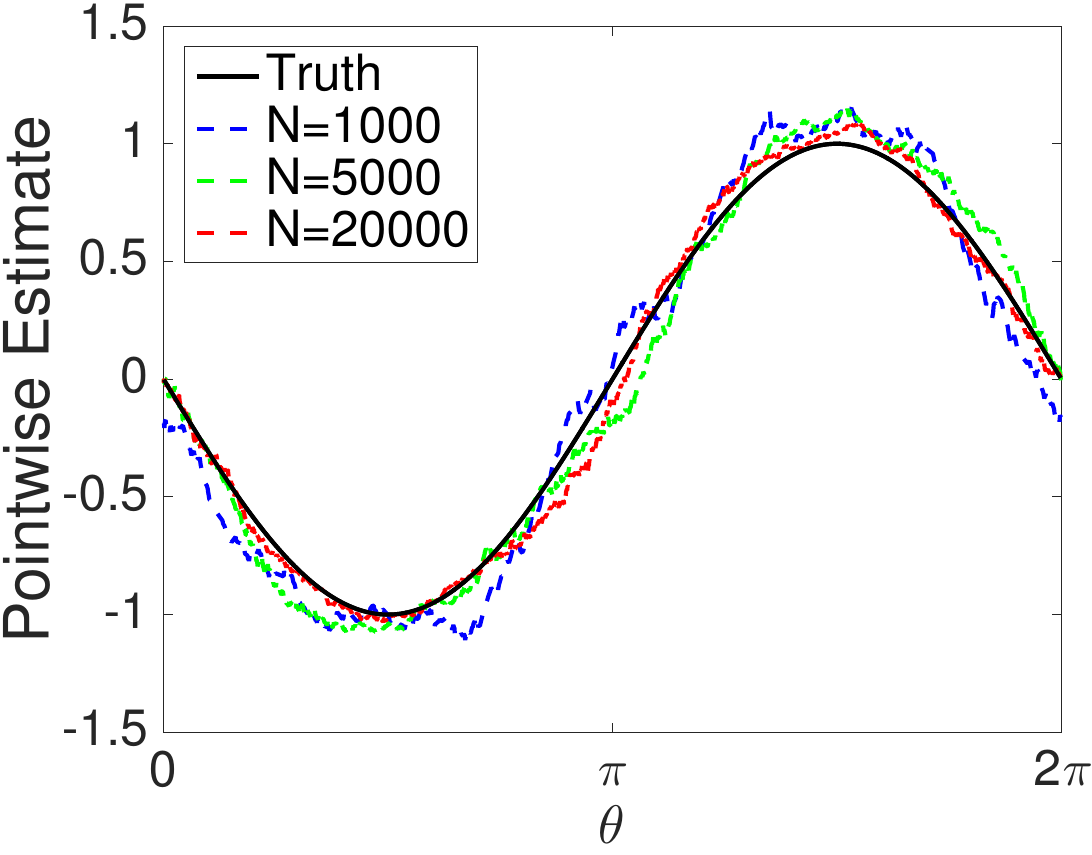}}
\subfigure[]{\includegraphics[width=.45\linewidth]{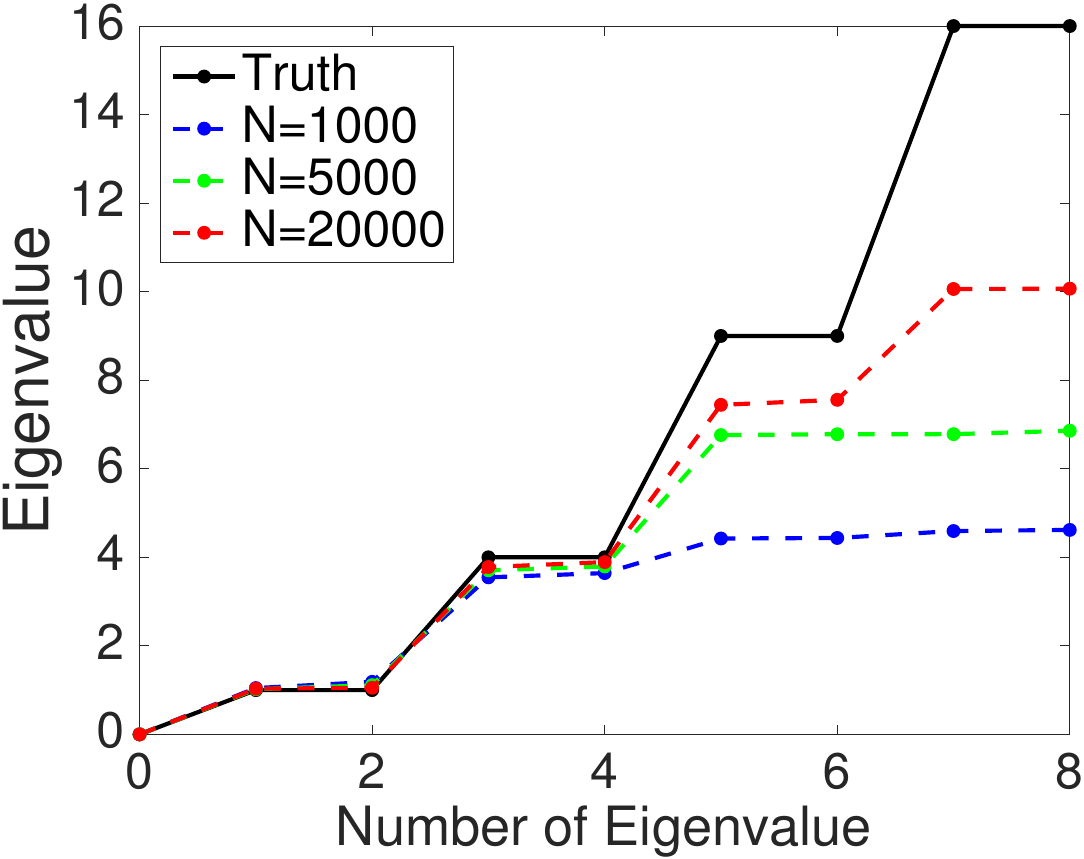}} \\
\subfigure[]{\includegraphics[width=.45\linewidth]{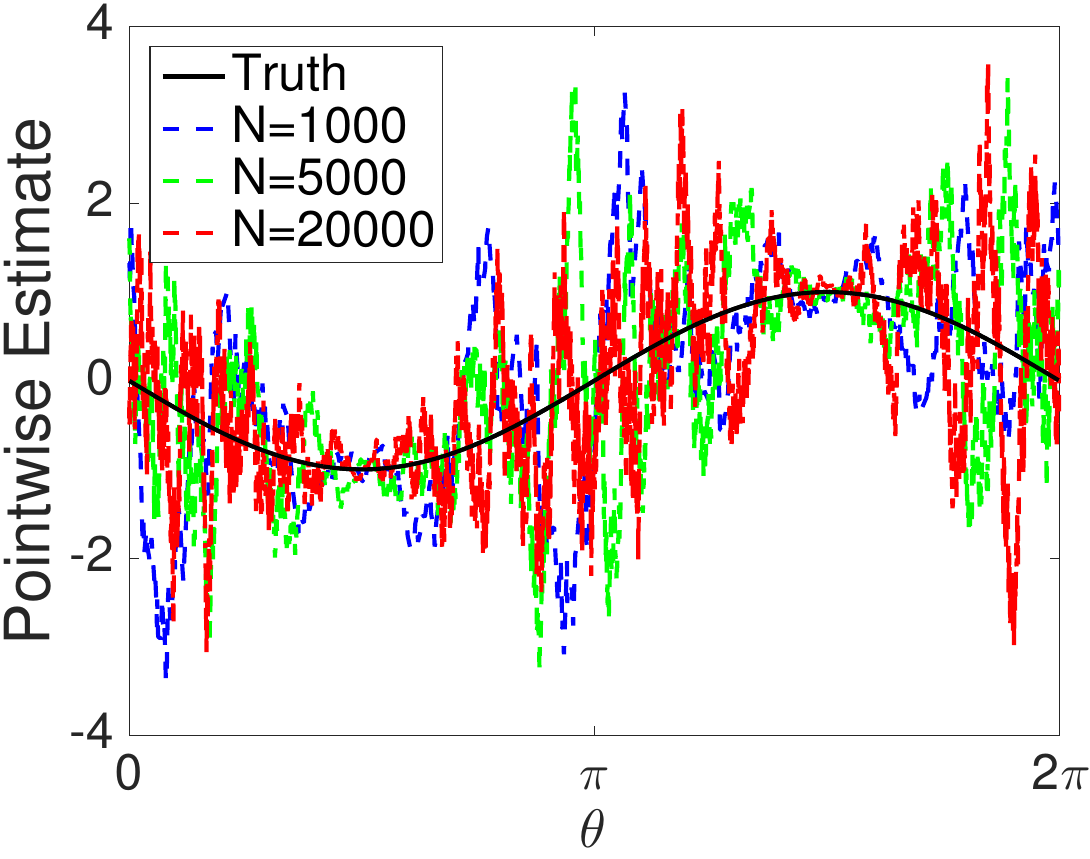}}
\subfigure[]{\includegraphics[width=.45\linewidth]{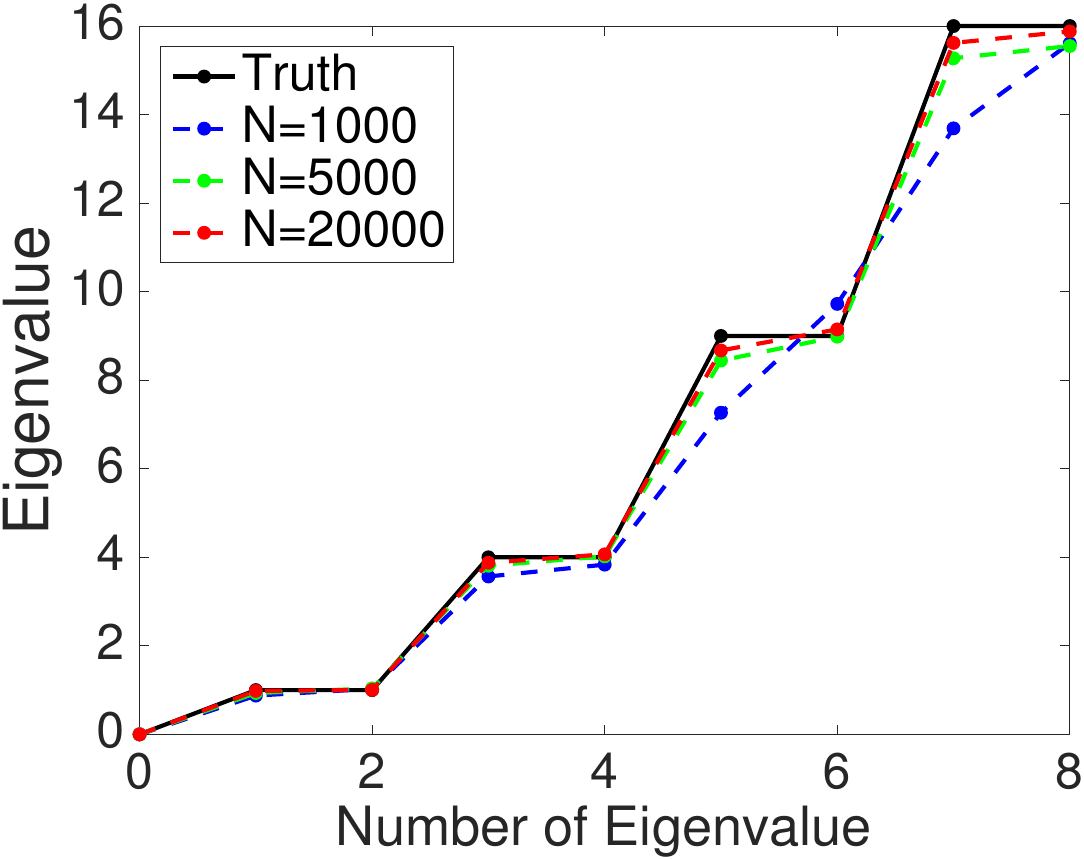}}
\end{center}
\caption{\label{figureApp} \rm For $N \in \{1000,5000,20000\}$ uniformly random data points on a unit circle. (a) Using $\delta=3N^{-1/7}$ and letting $\vec f_i = \sin(\theta_i)$ we compare the pointwise estimate $\left(c^{-1}L_{\rm un}\vec f \right)_i$ (red, dashed) to the true operator $\Delta f(x_i) = -\sin(\theta_i)$ (black, solid). (b) Using $\delta=3N^{-1/7}$ we compare the first 9 eigenvalues of $c^{-1}L_{\rm un}$ (red, dashed) to the first 9 eigenvalues of $\Delta$ (black, solid). (c,d) Same as (a,b) respectively but with $\delta=3N^{-1/3}$.}
\end{figure}

\begin{examp}[Pointwise and spectral estimation on the unit circle] \rm
We illustrate by example the difference in bandwidth required for pointwise versus spectral approximation that are implied by Theorems \ref{pointwiseLun} and \ref{spectral}.
Consider data sets consisting of $N\in\{1000,5000,20000\}$ points $\{x_i = (\sin(\theta_i),\cos(\theta_i))^\top \}_{i=1}^{N}$ sampled uniformly from the unit circle and the unnormalized graph Laplacian $c^{-1}L_{\rm un}$ constructed using the cutoff kernel.  We first chose $\delta=3N^{-1/(m+6)} = 3N^{-1/7}$ (the constant 3 was selected by hand tuning) which is optimal for pointwise estimation of the Laplace-Beltrami operator $\Delta$ as shown in Figure \ref{figureApp}(a)(b).  Next, we set $\delta=3N^{-1/3}$ which is optimal for spectral estimation as shown in Figure \ref{figureApp}(c)(d).  This demonstrates how spectral estimation is optimal for a much smaller value of $\delta$ that pointwise estimation.  Despite the extremely poor pointwise estimation when $\delta =3N^{-1/3}$, the spectral estimation, which involves an integrated quantity, is superior using this significantly smaller value of $\delta$.  Notice that the relative error in the eigenvalues increases as the eigenvalues increase. This phenomenon is explained at the end of the next section.

In general we only know that the optimal choice is $\delta \propto N^{-2/(m+6)}$ for $m>1$ and $\delta \propto N^{-1/3}$ for $m=1$, and the constants for the optimal choice are quite complex.  However, this does indicate the correct order of magnitude for $\delta$, especially for large data sets.  Figure \ref{figureApp} dramatically illustrates the differences in optimal pointwise and optimal spectral estimation.  As mentioned above, graph Laplacians are often used for spectral approximation, so previous analyses which tune $\delta$ for optimal pointwise estimation \cite{SingerEstimate,BH14} are misleading for many applications.

\begin{figure}
\begin{center}
\subfigure[]{\includegraphics[width=.45\linewidth]{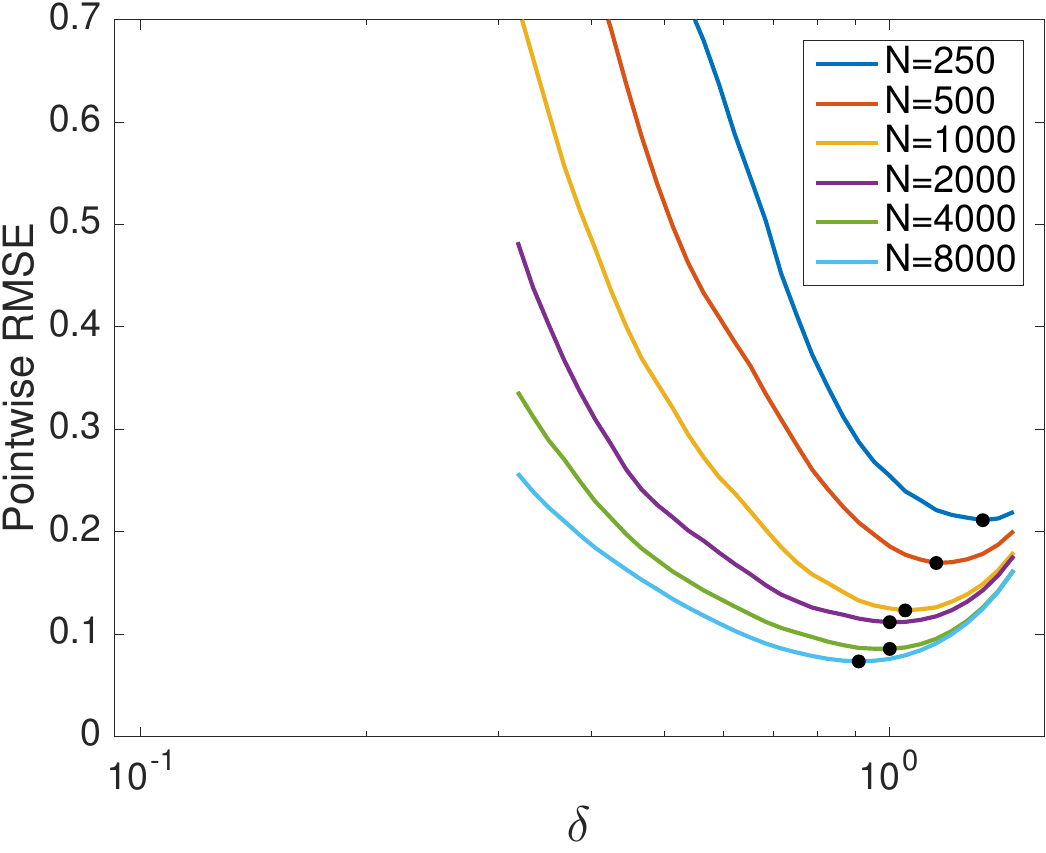}}
\subfigure[]{\includegraphics[width=.45\linewidth]{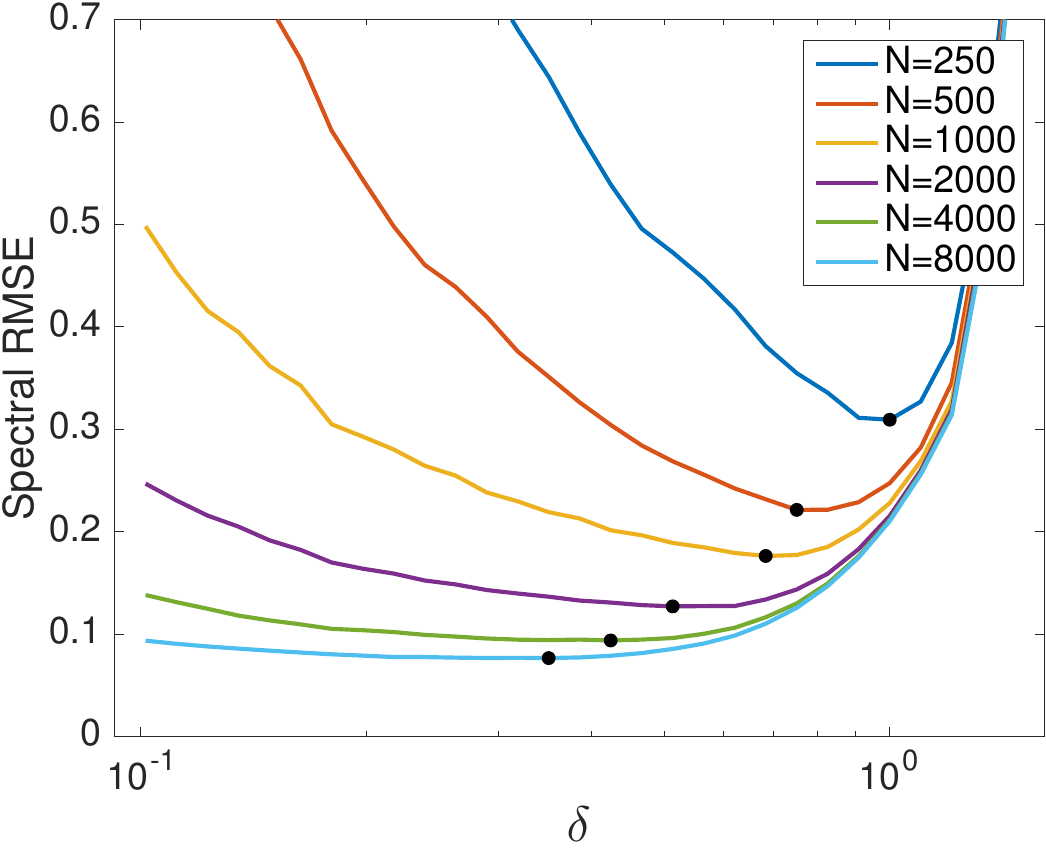}} \\
\subfigure[]{\includegraphics[width=.45\linewidth]{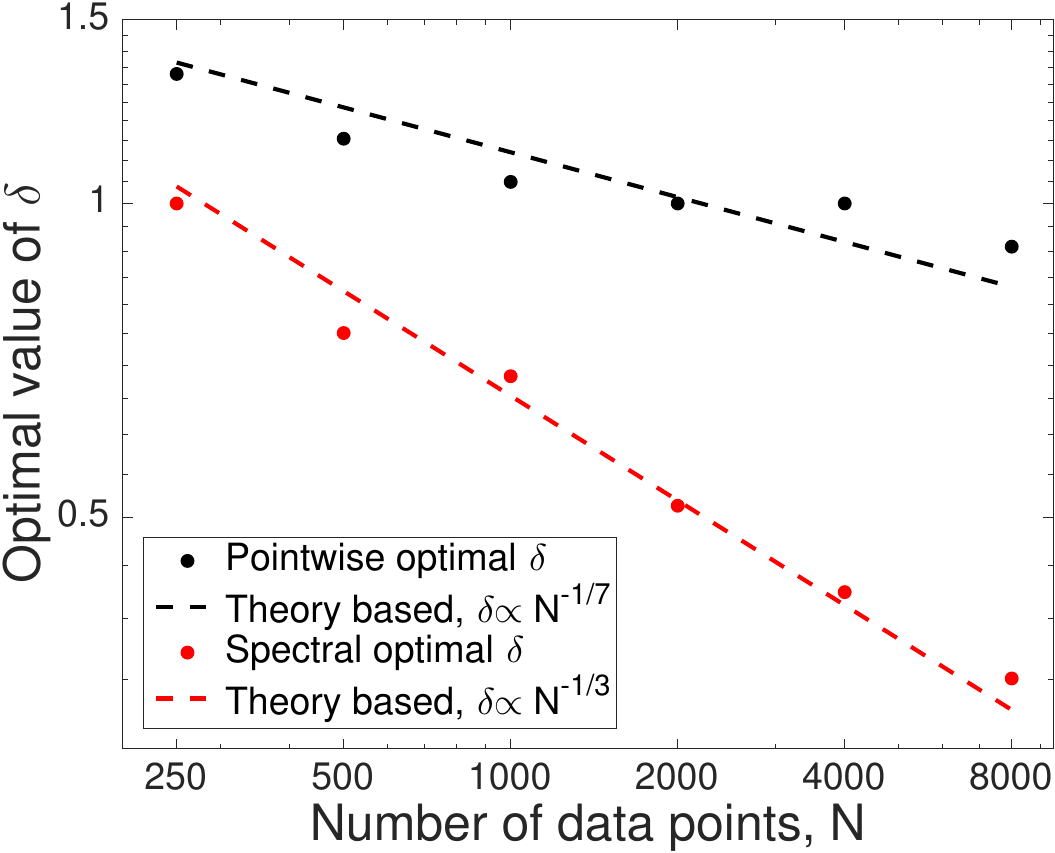}}
\end{center}
\caption{\label{figureApp2} \rm For $N \in \{250,500,1000,2000,4000,8000\}$ we show (a) the root mean squared error (RMSE) of the pointwise approximation of $c^{-1}L_{\rm un}\vec f \approx -\sin(x_i)$ where $\vec f_i = \sin(x_i)$ and (b) the RMSE of the first 5 eigenvalues of $c^{-1}L_{\rm un}$.  Both (a) and (b) are averaged over 10 random data sets for each $N$ and each $\delta$ and the minimum RMSE in each curve is highlighted with a black point.  (c) For each $N$, we plot the optimal $\delta$ for pointwise approximation (black points) and the theoretical power law $\delta \propto N^{-1/7}$ (black, dashed line) and the optimal $\delta$ for spectral approximation (red points) and the theoretical power law $\delta \propto N^{-1/3}$ (red, dashed line). }
\end{figure}

In Figure \ref{figureApp2} we verify the power laws for the optimal choice of $\delta$ for pointwise and spectral estimation.  For $N\in \{250,500,1000,2000,4000,8000\}$ we compute the pointwise and spectral root mean squared error (RMSE) for a wide range of values of $\delta$ and then plot the optimal value of $\delta$ as a function of $N$.  To estimate the pointwise error for a fixed $N$ and $\delta$ we generated $N$ uniformly random points on a circle $\{x_i = (\sin(\theta_i),\cos(\theta_i))^\top \}_{i=1}^{N}$ and then construct the unnormalized graph Laplacian $c^{-1}L_{\rm un}$ using the cutoff kernel.  We then multiplied this Laplacian matrix by the vector $\vec f_i = \sin(x_i)$ so that $\left(c^{-1}L_{\rm un}\vec f \right)_i \approx \Delta \sin(x_i) = -\sin(x_i)$ and we computed the RMSE
\[ \left(\sum_{i=1}^N \left(\left(c^{-1}L_{\rm un}\vec f \right)_i - (-\sin(x_i))\right)^2 \right)^{1/2} \]
between the estimator and the limiting expectation.  We repeated this numerical experiment 10 times for each value of $N$ and $\delta$ and the average RMSE is shown in Figure \ref{figureApp2}(a).  To estimate the spectral error, we computed the smallest 5 eigenvalues of $c^{-1}L_{\rm un}$ and found the RMSE between these eigenvalues and the true eigenvalues $[0,1,1,4,4]$ of the limiting operator $\Delta$.  This numerical experiment was repeated 10 times for each value of $N$ and $\delta$ and the average RMSE is shown in Figure \ref{figureApp2}(b).  Finally, in Figure \ref{figureApp2}(c) we plot the value of $\delta$ which minimized the pointwise error (black) and the spectral error (red) for each value of $N$ and we compare these data points to the theoretical power laws for the pointwise error $\delta \propto N^{-1/7}$ (black, dashed line) and spectral error $\delta \propto N^{-1/3}$ (red, dashed line) respectively.
\end{examp}

\subsection{Limiting geometries and spectral convergence}

We now show that the operator that is approximated spectrally in Theorem \ref{spectral} is a Laplace-Beltrami operator on the manifold $\mathcal{M}$, but with respect to a conformal change of metric.  Consider the functional approximated spectrally by $L_{\rm un}$ as shown in \eqref{LunSpecBias} which is
\begin{equation}\label{SpecFunctional} \Lambda(f) \to \frac{\left<f,q\mathcal{L}_{q,\rho}f \right>_{dV}}{\left<f^2,\mu q\right>_{dV}} = \frac{\left<f,q^2\rho^{m+2}\left(\Delta f - \nabla \log q^2 \rho^{m+2} \cdot \nabla f \right)\right>_{dV}}{\left<f^2,\mu q\right>_{dV}}. \end{equation}
A conformal change of metric corresponds to a new Riemannian metric $\tilde g \equiv \varphi g$, where $\varphi(x)>0$, and which has Laplace-Beltrami operator
\begin{equation}\label{DeltildeApp} \Delta_{\tilde g}f = \frac{1}{\varphi}(\Delta f - (m-2)\nabla \log \sqrt{\varphi} \cdot \nabla f).  \end{equation}
In order to rewrite the operator in \eqref{SpecFunctional} as a Laplace-Beltrami operator, the function $\varphi$ must be chosen to be $\varphi^{\frac{m-2}{2}} = q^2\rho^{m+2}$.  Moreover, we need to change the inner product in \eqref{SpecFunctional} to $d\tilde V$ which is the volume form of the new metric given by
\begin{equation}\label{volform2} d\tilde V = \sqrt{|\tilde g|} = \sqrt{|\varphi g|} = \varphi^{m/2} \sqrt{|g|} = \varphi^{m/2} \, dV. \end{equation}
Changing the volume form and substituting $\Delta_{\tilde g}$ in \eqref{SpecFunctional} we find
\[ \Lambda(f) \to \frac{\left<f,q^2\rho^{m+2} \varphi^{1-m/2} \Delta_{\tilde g}f\right>_{d\tilde V}}{\left<f^2,\mu q \varphi^{-m/2} \right>_{d\tilde V}} = \frac{\left<f,\Delta_{\tilde g}f\right>_{d\tilde V}}{\left<f^2,\mu q \varphi^{-m/2} \right>_{d\tilde V}}\]
where the second equality follows from the definition $\varphi^{\frac{m-2}{2}} = q^2\rho^{m+2}$.  Finally, in order for the denominator to represent the appropriate spectral normalization we require
\[ \mu = q^{-1}\varphi^{m/2} = q^{\frac{m+2}{m-2}}\rho^{m\left(\frac{m+2}{m-2}\right)}. \]
which implies that
\[ \Lambda(f) \to \frac{\left<f,\Delta_{\tilde g}f\right>_{d\tilde V}}{\left<f,f \right>_{d\tilde V}}. \]
This immediately shows that when $m\neq 2$ we can always choose $\rho,\mu$ to estimate the operator $\Delta_{\tilde g}$ for any conformally isometric geometry $\tilde g = \varphi g$.  In the special case $m=2$ this can also be achieved by setting $\rho = q^{-1/2}$ and $\mu = q\varphi^{-1}$.  However, if we assume that $\rho=q^{\beta}$ for some power $\beta$, then there are three choices of $\rho,\mu$ which have the same geometry for every dimension $m$, we summarize these in the table below.
\begin{center}
\begin{tabular}{|l|c|c|c|c|} \hline Geometry & $\tilde g$ & $d\tilde V$ & $\rho$ & $\mu$  \\ \hline
Sampling measure geometry & $q^{2/d}g$ & $q\,dV$ & $q^{-1/m}$ & $1$ \\ \hline
Embedding geometry  & $g$ & $dV$ & $q^{-2/(m+2)}$ & $q^{-1}$ \\ \hline
Inverse sampling geometry & $q^{-1}g$ & $q^{-d/2}\,dV$ & $q^{-1/2}$ & $q^{m/2+1}$ \\ \hline
\end{tabular}
\end{center}
For the choice $\beta = -1/m$ we find $\mu = 1$, which is explored in the main body of the paper. The choice $\mu=1$ is the only one allowing an unweighted graph construction. To reconstruct the embedding geometry, one can estimate the Laplace-Beltrami operator $\Delta_g$ using $\beta = \frac{-2}{m+2}$. Finally, if we select $\beta = -1/2$ we find $\mu = q^{m/2 + 1}$, which is closely related to the results of \cite{BH14}.

The above construction shows that an appropriately chosen graph Laplacian is a consistent estimator of the normalized Dirichlet energy
\[  \Lambda(f) \to \frac{\left<f,\Delta_{\tilde g}f\right>_{d\tilde V}}{\left<f,f \right>_{d\tilde V}} = \frac{\int_{\mathcal{M}}||\nabla_{\tilde g}f||^2 d\tilde V }{\int_{\mathcal{M}} f^2 d\tilde V} \]
for any conformally isometric geometry $\tilde g = \varphi g$.  The minimizers of this normalized Dirichlet energy are exactly the eigenfunctions of the Laplace-Beltrami operator $\Delta_{\tilde g}$, so Theorem \ref{spectral} is the first step towards spectral convergence.  However, in the details of Theorem \ref{spectral} there is a significant barrier to spectral convergence.  This barrier is a very subtle effect of the second term in the variance of the estimator.  Rewriting \eqref{LunSpecVar} from Theorem \ref{spectral} in terms of the new geometry we find
\begin{align}\label{LunVarNew} &\textup{var}\left(\Lambda(f) \right) \nonumber \\=& a \frac{ \delta^{-m-2}}{N(N-1)} \frac{\left<f^2,\Delta_{\tilde g}(f^2) \right>_{d\tilde V}}{\left<f,f\right>_{d\tilde V}^2} + \frac{4}{N} \frac{\left<f^2 ,q^{-1}\varphi^{m/2}(\Delta_{\tilde g}f)^2 \right>_{d\tilde V}}{\left<f,f\right>_{d\tilde V}^2} + \mathcal{O}(\delta^2,N^{-1}). \end{align}
The first term in \eqref{LunVarNew} normally controls the bias-variance trade-off since it diverges as $\delta\to 0$, and the second term is typically higher order.  However, the integral which defines the constant in the second error term has the potential to be infinite when $q$ is not bounded below.  Ignoring the terms which depend on $f$, we need $q^{-1}\varphi^{m/2}\,d\tilde V = q^{-1}\varphi^m \,dV$ to be integrable in order for the variance of the estimator to be well-defined, which proves the following result.

\begin{thm}[Spectral Convergence, Part 1]\label{spec1} Let $\{x_i\}_{i=1}^N$ be sampled from a density $q$.  Consider a conformally equivalent metric $\tilde g = \varphi g$ such that $q^{-1}\varphi^m$ is integrable with respect to $dV$.  Define the unnormalized Laplacian $L_{\rm un}$ using a kernel $K_{\delta}(x,y) = h\left(\frac{||x-y||^2}{\delta^2\rho(x)\rho(y)}\right)$ for any $h$ with exponential decay where
\[ \rho = \left\{ \begin{array}{ll} q^{\frac{2}{m+2}}\varphi^{\frac{m-2}{m+2}} & m\neq 2 \\ q^{-1/2} & m=2 \end{array} \right. \]
and define $M_{ii}=\mu(x_i)$ where
\[ \mu = \left\{ \begin{array}{ll} q^{-1}\varphi^{m/2} & m\neq 2 \\ q\varphi^{-1} & m=2 \end{array} \right. \]
then for $f\in C^3(\mathcal{M},\tilde g)$ the functional $\Lambda(f) = \frac{\vec f \,^\top c^{-1}L_{\rm un} \vec f}{\vec f \,^\top M \vec f}$ which corresponds to the generalized eigenvalue problem $c^{-1} L_{\rm un} \vec f = \lambda M \vec f$ is a consistent estimator of the normalized Dirichlet energy functional
\[ \mathbb{E}[\Lambda(f)] \to_{N\to\infty}  \frac{\left<f,\Delta_{\tilde g}f\right>_{d\tilde V}}{\left<f,f \right>_{d\tilde V}} = \frac{\int_{\mathcal{M}}||\nabla_{\tilde g}f||^2 d\tilde V }{\int_{\mathcal{M}} f^2 d\tilde V}  \]
and the estimator $\Lambda(f)$ has bias of leading order $\delta^2$ and finite variance which is of leading order $\delta^{-m-2}N^{-2}$.  The optimal choice $\delta \propto N^{-2/(m+6)}$ yields a root mean squared error of order $N^{-4/(m+6)}$.
\end{thm}

In particular, for the choice $\rho = q^{-2/(m+2)}$ we find $q^{-1}\varphi^m = q^{-1}$ which will often not be integrable on a non-compact manifold when $q$ is not bounded away from zero.  This implies that we cannot spectrally approximate the Laplace-Beltrami operator with respect to the embedding metric for many non-compact manifolds.  Notice, that this variance barrier only applies to spectral convergence, so we can still approximate the operator pointwise using Theorem \ref{pointwiseLun}.  Similarly, when $\beta = -1/2$ we find $q^{-1}\varphi^m = q^{-1-m}$ and the exponent on $q$ is negative, also leading to the possibility of divergence on non-compact manifolds.  This shows yet another advantage of the choice $\beta = -1/m$, since $q^{-1}\varphi^m = q$ which is simply the sampling measure and is always integrable with respect to $dV$ by definition.

Theorem \ref{spec1} shows that the functional $\Lambda(f)$ converges to the normalized Dirichlet energy functional whose minimizers are the eigenfunctions of the Laplace-Beltrami operator $\Delta_{\tilde g}$.  It remains only to show that vectors which minimize the discrete functional $\Lambda(f)$ converge to the minimizers of the normalized Dirichlet energy functional.  In fact this has already been shown in \cite{von2008consistency} assuming that the spectrum of $\Delta_{\tilde g}$ is discrete.

Next, will address an issue of spectral convergence that was introduced in \cite{von2008consistency}, which suggests that unnormalized graph Laplacians have worse spectral convergence properties than normalized Laplacians.  The theory in that article is  the basis of our spectral convergence result below.  However, a subtle detail reveals that the unnormalized Laplacian $c^{-1}L_{\rm un}$ does not suffer from the spectral convergence issue they consider.  For an unnormalized Laplacian, $L_{\rm un}=D-W$ where $D_{ii} = \sum_{j=1}^N W_{ij}$ is the degree function, ``Result 2" of \cite{von2008consistency} states that if the first $r$ eigenvalues of the limiting operator $\mathcal{L}_{q,\rho}$ do not lie in the range of the degree function
\[ d(x) = \lim_{N\to\infty} \sum_{j=1}^N W(x,x_j), \]
then the first $r$ eigenvalues of the unnormalized Laplacian $L_{\rm un}$ converge to those of the limiting operator (spectral convergence).  This would suggest that the spectral convergence only holds for eigenvalues which are separated from the range of $D_{ii}$.  However, notice that we divide $L_{\rm un}$ by the normalization constant $c = \mathcal{O}(N\delta^{m+2})$ whereas $\lim_{N\to\infty} \sum_{j=1}^N W(x,x_j) \propto q(x) N\delta^m$ which implies that
\[ c^{-1}d(x) = \mathcal{O}(\delta^{-2}). \]
Since $\delta \to 0$ as $N\to \infty$, this implies that the range of the true degree function $c^{-1}d(x)$ approaches $\infty$ as $N$ grows.  This special class of unnormalized Laplacians avoids this difficulty because the first order term of the degree function in exactly cancelled by the first order term of the kernel $W_{ij}$ in the limit of large data, which is why the constant $c$ is higher order than the degree function in terms of $\delta$.  

 Finally, while Theorem \ref{spec1} gives an optimal bias-variance tradeoff of $\delta \propto N^{-2/(m+6)}$, this may not be sufficient for spectral convergence in high dimensions.  Taking $\delta\to 0$ and $N\to \infty$ simultaneously was first addressed in \cite{spec0} and more recently in \cite{spec1,spec2,spec3}.  The best results \cite{spec2} require a bounded manifold and also $\left( N^{-1}\log N \right)^{1/m} < \delta$ but the optimal bias-variance tradeoff only satisfies this constraint for $m<6$, so for $m\geq 6$ we require $\delta = \left(N^{-1}\log N \right)^{1/m}$. While \cite{spec2} has the best current result, we conjecture that spectral convergence still holds on unbounded manifolds when the spectrum $\Delta_{\tilde g}$ is discrete and also when $\delta \propto N^{-2/(m+6)}$ for $m\geq 6$.

\begin{thm}[Spectral Convergence, Part 2]\label{spectralconv} Under the assumptions of Theorem \ref{spec1}, in the limit of large data the eigenvalues of $c^{-1}M^{-1}L_{\rm un}$ converge to those of the limiting operator $\Delta_{\tilde g}$, assuming the spectrum is discrete and $\mathcal{M}$ is bounded.

In particular, for a manifold without boundary or with a smooth boundary, spectral convergence holds for the unnormalized Laplacian when $\rho = q^{-1/m}$ and $\mu=1$.
\end{thm}
\begin{proof}
Notice that the eigenvalues of $c^{-1}M^{-1}L_{\rm un}\vec f = \lambda \vec f$ exactly correspond to the minimizers of the functional $\Lambda(f)$ in Theorem \ref{spec1}.  The first claim follows from the convergence of the functional $\Lambda(f)$ to the normalized Dirichlet energy combined with convergence results of \cite{spec2} as discussed above.

The guarantee of spectral convergence for $\rho=q^{-1/m}$ and $\mu=1$ follows from the fact that $q^{-1}\varphi^m = q$ is always integrable.  Moreover, the volume of the manifold with respect to $d\tilde V = q \, dV$ is always $\textup{vol}_{d\tilde V}(\mathcal{M}) = \int_{\mathcal{M}} q\, dV =  1$, and for any manifold of finite volume with a smooth boundary the spectrum of the Laplace-Beltrami operator is discrete \cite{cianchi2011}.
\end{proof}

Since the matrix $c^{-1}M^{-1}L_{\rm un}$ is not symmetric it is numerically preferable to solve the symmetric generalized eigenvalue problem $c^{-1}L_{\rm un}\vec f = \lambda M \vec f$. 

Finally, it is often noticed empirically that the error in the spectral approximations increases as the value of the eigenvalue increases.  To understand this phenomenon, we will use the variance formula \eqref{LunVarNew}.  First, if $f$ is an eigenfunction with $\Delta_{\tilde g}f = \lambda f$ we note that
\begin{align} \left<f^2,\Delta_{\tilde g}(f^2)\right>_{d\tilde V} &=\left<f^2,2f\Delta_{\tilde g}f - 2\nabla_{\tilde g}f \cdot \nabla_{\tilde g} f\right>_{d\tilde V}\nonumber \\&= 2\lambda\left<f^2,f^2\right>_{d\tilde V}  - 2\left<f^2,\nabla_{\tilde g}f \cdot \nabla_{\tilde g} f\right>_{d\tilde V}   \nonumber \\
&= 2\lambda\left<f^2,f^2\right>_{d\tilde V}  - 2\left<-\textup{div}_{\tilde g}\left(f^2\nabla_{\tilde g} f \right),f\right>_{d\tilde V}\nonumber \\&= 2\lambda\left<f^2,f^2\right>_{d\tilde V}  - \frac{2}{3}\left<-\textup{div}_{\tilde g}\left(\nabla_{\tilde g} f^3 \right),f\right>_{d\tilde V} \nonumber \\
&= 2\lambda\left<f^2,f^2\right>_{d\tilde V}  - \frac{2}{3}\left<\Delta_{\tilde g}( f^3),f\right>_{d\tilde V}\nonumber \\&= 2\lambda\left<f^2,f^2\right>_{d\tilde V}  - \frac{2}{3}\left<f^3,\Delta_{\tilde g}f\right>_{d\tilde V} = \frac{4}{3}\lambda \left<f^2,f^2\right>_{d\tilde V}.\nonumber
\end{align}
Applying this formula to \eqref{LunVarNew} we find $\hat \lambda \equiv \Lambda(f)$ is a consistent estimator of the eigenvalue $\lambda$ with relative variance

\ \vspace*{-10pt}
\begin{align}\label{SpecVar1} \frac{\textup{var}(\hat \lambda )}{\lambda} &= a \frac{ \delta^{-m-2}}{N(N-1)} \frac{4\left<f^2,f^2 \right>_{d\tilde V}}{3\left<f,f\right>_{d\tilde V}^2} + \frac{4 \lambda}{N} \frac{\left<f^2 ,q^{-1}\varphi^{m/2} f^2 \right>_{d\tilde V}}{\left<f,f\right>_{d\tilde V}^2} + \mathcal{O}(\delta^2,N^{-1}). \nonumber \end{align}
Although the first term (which is typically the leading order term, especially on compact manifolds) can be tuned to give a constant relative error, the second term still grows proportionally to the eigenvalue $\lambda$.  We should expect the spectrum to be fairly accurate until the order of the second term is equal to that of the first term, at which point the relative errors in the eigenvalues will grow rapidly.  When $\varphi = q^{2/m}$ as in the CkNN construction, the inner products in the two terms are the same, so the spectrum will be accurate when
\[ \lambda < \lambda_{\max} \equiv \frac{ a \delta^{-m-2} }{3(N-1)} = \frac{a N^{2(m+2)/(m+6)}}{3(N-1)} = \mathcal{O}\left(N^{\frac{m-2}{m+6}} \right) \]
where the second and third equalities hold for the optimal choice $\delta \propto N^{-2/(m+6)}$.  This constraint does not apply to the case $m=1$, where the second term in \eqref{LunVarNew} dominates, or to the case $m=2$ where the second and first terms are equal order for the optimal choice of $\delta$.  In all cases, the relative error increases as the eigenvalues increase.

\section{Conclusion}\label{conclusion}

We have introduced a new method called continuous k-nearest neighbors as a way to construct a single graph from a point cloud, that is provably consistent on connected components.  By proving the consistency of the geometry (spectral convergence of the graph Laplacian to the Laplace-Beltrami operator) we support our conjecture that CkNN is topologically consistent, meaning that correct topology can be extracted in the large data limit. For many finite-data examples from compact Riemannian manifolds, we have shown that CkNN compares favorably to persistent homology approaches. 

The proposed method replaces a small $\epsilon$ radius, or $k$ in the k-nearest neighbors method, with a unitless continuous parameter $\delta$. The theory proves the existence of a correct choice of $\delta$, and it needs to be tuned in specific examples.
While the difference between the CkNN and the kNN constructions is fairly simple, the crucial difference is that the approximation \eqref{kNNdensity} only holds for $k$ small, relative to $N$.  By varying the parameter $\delta$ and holding $k$ constant at a small value we can construct multi-scale approximations of our manifold that are still consistent with the underlying manifold.  This contrasts with the standard kNN approach, where the parameter $k$ is varied and both the coarseness of the manifold approximation and the underlying manifold geometry are changing simultaneously (because the scaling function is changing).

Surprisingly, a careful analysis of the bias and variance of the graph Laplacian as a spectral estimator of the Laplace-Beltrami operator is a key element of the proof of consistency. The variance can be infinite on non-compact manifolds, depending on the geometry, creating a previously unknown barrier to spectral convergence. The variance calculation also allows us to explain why the relative error of the eigenvalue increases along with the eigenvalue, and we determine the part of the spectrum that can estimated with constant relative error, as a function of the data size $N$. 


We would like to thank D.~Giannakis for helpful conversations.
\providecommand{\href}[2]{#2}
\providecommand{\arxiv}[1]{\href{http://arxiv.org/abs/#1}{arXiv:#1}}
\providecommand{\url}[1]{\texttt{#1}}
\providecommand{\urlprefix}{URL }

\end{document}